\documentclass{amsart}
\usepackage{amsmath}
\usepackage{amssymb}
\usepackage{amsthm}
\usepackage{graphicx}
\usepackage{wrapfig}
\usepackage{float}
\usepackage{epsfig}
\usepackage{psfrag}
\usepackage[margin=1in]{geometry}
\newtheorem{theorem}{Theorem}[section]
\newtheorem{lemma}[theorem]{Lemma}
\newtheorem{proposition}[theorem]{Proposition}
\newtheorem{corollary}[theorem]{Corollary}
\theoremstyle{definition}
\newtheorem{definition}[theorem]{Definition}
\newtheorem{example}[theorem]{Example}
{\bf }
\theoremstyle{remark}
\newtheorem{remark}[theorem]{Remark}
\numberwithin{equation}{section}
\newcommand{\abs}[1]{\lvert#1\rvert}

\newcommand{\comment}[1]{} 

\begin{document}
\title{Planar algebras: a category theoretic point of view}
\author{Shamindra Kumar Ghosh}
\begin{abstract}
We define Jones's planar algebra as a map of multicategories and constuct a
planar algebra starting from a $1$-cell in a pivotal strict $2$-category. We
prove finiteness results for the affine representations of finite depth
planar algebras. We also show that the radius of convergence of the
dimension of an affine representation of the planar algebra associated to a
finite depth subfactor is at least as big as the inverse-square of the
modulus.
\end{abstract}
\maketitle
\section{Introduction}
In \cite{subfac ind}, Vaughan Jones introduced the notion of index for type
$II_{1}$ subfactors. To any finite index subfactor $N \subset M$ one can
associate a tower of $II_1$ factors
$N\subset M\subset M_{1}\subset M_{2}\subset \cdots$.
The standard invariant of the subfactor is then given by the grid of finite
dimensional algebras of relative commutants (see \cite{GHJ}, \cite{JS}, \cite{
Pop1}, \cite{lambda lattice})
\begin{equation*}
\begin{tabular}{ccccccccc}
$N\cap N$ & $\subset $ & $N^{\prime }\cap M$ & $\subset $ & $N^{\prime }\cap
M_{1}$ & $\subset $ & $N^{\prime }\cap M_{2}$ & $\subset $ & $\cdots $ \\
&  & $\cup $ &  & $\cup $ &  & $\cup $ &  &  \\
&  & $M^{\prime }\cap M$ & $\subset $ & $M^{\prime }\cap M_{1}$ & $\subset $
& $M^{\prime }\cap M_{2}$ & $\subset $ & $\cdots $%
\end{tabular}
\end{equation*}
Sorin Popa in \cite{lambda lattice} studied the question of which families
$\{Aij:-1\leq i\leq j\leq \infty \}$ of finite-dimensional
$C^{\ast }$-algebras could arise as the tower of relative commutants of an
extremal finite-index subfactor, that is, when does there exist such a
subfactor $M_{-1}\subset M_{0}$ such that $A_{ij}=M_{i}\cap M_{j}$. He
obtained a beautiful algebraic axiomatization of such families, which he called
$\lambda $\textit{-lattices}. Ocneanu gave a combinatorial description of
the standard invariant as so called {\it paragroups} (see \cite{EK}).
Subsequently, Jones gave a geometric reformulation of
the standard invariant, which he called \textit{planar algebras }(see
\cite{pln alg}). Jones then introduced the notion of `modules over a planar
algebra' in \cite{ann rep} and computed the irreducible modules
over the Temperley-Lieb planar algebras for index greater than $4$. Planar
algebras became a powerful tool to construct subfactors of index less than
$4$. In particular, a new construction of the subfactors with principal graph,
$E_{6}$ and $E_{8}$ could be given (see \cite{ann rep}). The author (see
\cite{Ghosh}) established a one-to-one correspondence
of all modules over the group planar algebra, that is, the planar algebra
associated to the subfactor arising from the action of a finite group, and
the representations of a non-trivial quotient of the quantum double of the
group over a certain ideal. The reason for the appearance of a quotient of
the quantum double instead of just the quantum double was allowing rotation
of internal discs in the definition of the modules over a planar
algebra. Similar results also appeared in the field of TQFTs.
Kevin Walker and Michael Freedman proved that the representations of
the {\em annularization} of a tensor category satisfying suitable
conditions that allow one to perform the Reshetikhin-Turaev construction of
TQFT, is equivalent to the representations of the quantum double of
the category.
The author (in \cite{Ghosh}) also showed that the radius of convergence of the 
dimension of a module is at least as big as the inverse-square of the modulus 
in the case of group planar algebras and thus answering a question in
\cite{ann rep}.

Subfactors have been extensively studied from the point of view of the
associated bicategory of $N-N$, $N-M$, $M-N$, $M-M$ bimodules (see for example,
\cite{Bis}, \cite{Mug1}, \cite{Mug2}, \cite{hypgrp}, \cite{Wen}). It is
natural to expect a correspondence between the bicategory and the planar
algebra associated to the subfactor. One of the main objectives of this paper
is to construct a planar algebra directly from a bicategory.

From \cite{Ghosh}, it follows that if the
modules over a planar algebra are defined with rigid internal disc then they
are more interesting because of the connection with quantum double in the
case of group planar algebras. Another objective is to find such modules
(called \textit{affine representations}) and prove finiteness results
of affine representation for finite depth planar algebras.

Next, we give a section-wise summary of the paper; all results in this paper
appeared in a PhD thesis (2006) of the author submitted in University of New
Hampshire. In the first section, we discuss the preliminaries from basic category
theory. The first subsection recalls the definition of multicategories and
maps between them from \cite{Leinster}. We introduce the notion of the
stucture of \textit{empty objects} in a multicategory; the trivial
example, namely, the multicategory of sets or vector spaces possess the 
structure
of empty objects. In the second subsection, we discuss basics of bicategory
theory and several structures related to a bicategory, namely functors,
transformation between functors and rigidity.

We construct a new example of a multicategory with the structure of empty
objects which we call \textit{Planar Tangle Multicategory} in the second
section. We re-define Jones's planar algebra simply as a map of
multicategories from the Planar Tangle Multicategory to the multicategory of
vector spaces; in fact, this was motivated by Jones's idea of putting
the planar algebra as well as its dual in the definition itself. In the end, we
discuss more structures (modulus, connectedness, local finiteness, $C^{\ast }$-
structure, etc.) on a planar algebra.

In the third section, we start with fixing a
$1$-cell in a pivotal strict $2$-category and construct a planar algebra. Some
of the techniques used here are similar to Jones's construction of a planar
algebra from a subfactor. However, we would like to mention that this
construction is totally algebraic and heavily depends on the \textit{graphical
calculus} of the $2$-cells and the pivotal structure plays a key role here.

Motivated with the connection of annular representation of the group planar
algebra with the representations of a certain quotient of the quantum double
of the group, we considered \textit{affine representations} of
a planar algebra in the fourth section; this was introduced by Jones and 
Reznikoff in \cite{Aff TL} and Graham and Lehrer in \cite{GL}. We also discuss 
the general theory of such representations.

In the fifth and the final section, we discuss affine representations
of a planar algebra associated to a finite depth subfactor. We find a bound on
the weights of these representations which is dependent on the depth of the
planar algebra. We also prove that at each weight, the number of isomorphism 
classes of irreducible affine representations is finite. We answer Jones's
question on the radius of convergence of the dimension of affine 
representations for finite depth subfactor planar algebras.
\vskip 0.5cm
\noindent{\bf Acknowledgement}: The author is grateful to Professor Dmitri
Nikshych for teaching him all he knows about bicategories and multicategories,
and many useful discussions. The author would like to thank Professors V. S.
Sunder and Vijay Kodiyalam for their constant help in working on `modules over
planar algebras' which lead to the results in Section \ref{fdpa}, and also
Ved Prakash Gupta for numerous suggestions and corrections.
\section{Preliminaries}
\subsection{Multicategories}In this subsection, we revisit the
definition of \textit{multicategory} and an \textit{algebra for a
multicategory} (introduced in \cite{Leinster}). We introduce the
structure of \textit{empty object} in a multicategory which will be
useful in the subsequent sections.
\begin{definition}
A multicategory $\mathcal{C}$ consists of:\\
(i) a class $\mathcal{C}_{0}$ whose elements are called {\it
  objects} of $\mathcal{C}$,\\
(ii) for all $n\in \mathbb{N}, \; \underline{a}=\left(
  a_{1},a_{2},\cdots ,a_{n}\right) \in (\mathcal{C}_{0})^{n}, \; a\in
  \mathcal{C}_{0}$, there exists a class
  $\mathcal{C}(\underline{a};a)$ whose elements are called {\it
    morphisms} or {\it arrows} from $\underline{a}$ to $a$, together
  with a distinguished arrow $1_a \in \mathcal{C}(a;a)$ called {\it
    identity morphism} for $a$,\\
(iii) for all $n \in \mathbb{N}, \; k_{1}, k_{2}, \cdots, k_{n} \in
  \mathbb{N}, \; \underline{a_{i}}=(a_{i}^{1},a_{i}^{2},\cdots
  ,a_{i}^{k_{i}}) \in (\mathcal{C}_{0})^{k_{i}}$ where $i\in
  \{1,2,\cdots ,n\}, \; \underline{a}=(a_{1},a_{2},\cdots ,a_{n}) \in
  (\mathcal{C}_{0})^{n}, \; a \in \mathcal{C}_{0}$, there exists a
  {\it composition map} $\circ $ denoted in the following way:

\noindent $\mathcal{C}(\underline{a};a)\times \mathcal{C}(\underline{a_{1}}%
;a_{1})\times \mathcal{C}(\underline{a_{2}};a_{2})\times \cdots \mathcal{C}(%
\underline{a_{n}};a_{n})\ni (\theta ,\theta _{1},\theta _{2},\cdots ,\theta
_{n})\overset{\circ }{\mapsto }\theta \circ (\theta _{1},\theta _{2},\cdots
,\theta _{n}) \in \mathcal{C}((\underline{a_{1}},\underline{a_{2}},\cdots ,%
\underline{a_{n}});a)$\\
\noindent where $(\underline{a_{1}},\underline{a_{2}},\cdots ,\underline{%
a_{n}})=(a_{1}^{1},a_{1}^{2},\cdots
,a_{1}^{k_{1}},a_{2}^{1},a_{2}^{2},\cdots ,a_{2}^{k_{2}},\cdots
,a_{n}^{1},a_{n}^{2},\cdots ,a_{n}^{k_{n}})\in \left( \mathcal{C}_{0}\right)
^{k_{1}+\cdots +k_{n}}$ satisfies the following conditions:

(a) {\it Associativity axiom}:
$\theta \circ (\theta _{1}\circ (\theta _{1}^{1},\theta
_{1}^{2},\cdots ,\theta _{1}^{k_{1}}),\theta _{2}\circ (\theta
_{2}^{1},\theta _{2}^{2},\cdots ,\theta _{2}^{k_{2}}),\cdots ,\theta
_{n}\circ (\theta _{n}^{1},\theta _{n}^{2},\cdots ,\theta
_{n}^{k_{n}}))$\\
$= (\theta \circ (\theta _{1},\theta _{2},\cdots ,\theta
_{n}))\circ (\theta _{1}^{1},\theta _{1}^{2},\cdots ,\theta
_{1}^{k_{1}},\theta _{2}^{1},\theta _{2}^{2},\cdots ,\theta
_{2}^{k_{2}},\cdots ,\theta _{n}^{1},\theta _{n}^{2},\cdots ,\theta
_{n}^{k_{n}})$ whenever the composites make sense,

(b) {\it Identity axiom}:
$\theta \circ (1_{a_{1}},1_{a_{2}},\cdots ,1_{a_{n}})=\theta =1_{a}\circ
\theta $ for all $\theta \in \mathcal{C}((a_{1},a_{2},\cdots ,a_{n});a)$.
\end{definition}
\begin{remark}
The associativity and identity axioms are easier to understand with
pictorial notation of arrows (see \cite{Leinster}).
\end{remark}
\begin{example}
The collection of sets $\mathcal{MS}et$ (resp. vector spaces $\mathcal{MV}ec$)
forms a multicategory where arrows are given by maps from cartesian
product of finite collection of sets to another set (resp. multilinear maps
from a finite collection of vector spaces to another vector space).
\end{example}
\begin{example}
Any tensor category $\mathcal{C}$ has an inbuilt multicategory structure in
the obvious way by setting $\mathcal{C}_{0}=ob(\mathcal{C})=$ set of objects
of $\mathcal{C}$ and $\mathcal{C}((a_{1},a_{2},\cdots ,a_{n});a)=Mor_{%
\mathcal{C}}((\cdots ((a_{1}\otimes a_{2})\otimes a_{3})\otimes \cdots
\otimes a_{n-1})\otimes a_{n},a)$.
\end{example}
\begin{definition}
Let $\mathcal{C}$ and $\mathcal{C}^{\prime }$ be multicategories. A map of
multicategories $f:\mathcal{C\rightarrow C}^{\prime }$ consists of a map $f:%
\mathcal{C}_{0}\mathcal{\rightarrow C}_{0}^{\prime }$ together with another
map
\begin{equation*}
f:\mathcal{C}(a_{1},a_{2},\cdots ,a_{n};a)\rightarrow \mathcal{C}^{\prime
}(f(a_{1}),f(a_{2}),\cdots ,f(a_{n});f(a))
\end{equation*}
such that composition of arrows and identities are preserved. (If $\mathcal{C%
}$ and $\mathcal{C}^{\prime }$ are multicategories with each morphism space
being vector space and composition being multilinear, then we will assume
that the map of multicategories is linear between the morphism spaces.)
\end{definition}
\begin{definition}
Let $\mathcal{C}$ be a multicategory. A $\mathcal{C}$-algebra is simply a
map of multicategories from $\mathcal{C}$ to $\mathcal{MS}et$. (If $\mathcal{%
C}$ is a multicategory with each morphism space being vector space and
composition being multilinear, then we will consider a $\mathcal{C}$-algebra
to be a map of multicategories from $\mathcal{C}$ to $\mathcal{MV}ec$.)
\end{definition}
\begin{definition}
A multicategory $\mathcal{C}$ is said to be {\it symmetric} if the following
conditions hold:
\begin{itemize}
\item for all $n\in \mathbb{N}$, $\underline{a}\in (\mathcal{C}_{0})^{n}$, $%
a\in \mathcal{C}_{0}$, $\sigma \in S_{n}$, there exists a map $-\cdot \sigma
:\mathcal{C}(\underline{a};a)\rightarrow \mathcal{C}(\underline{a}\cdot
\sigma ;a)$ (where $(a_{1,}a_{2},\cdots ,a_{n})\cdot \sigma =(a_{\sigma
(1),}a_{\sigma (2)},\cdots ,a_{\sigma (n)})$)satisfying:
\item $(\theta \cdot \sigma )\cdot \rho =\theta \cdot (\sigma \cdot \rho )$
and $\theta =\theta \cdot 1_{S_{n}}$ for all $n\in \mathbb{N}$, $\underline{a%
}\in (\mathcal{C}_{0})^{n}$, $a\in \mathcal{C}_{0}$, $\sigma $,$\rho \in
S_{n},$ $\theta \in \mathcal{C}(\underline{a};a)$,
\item $(\theta \cdot \sigma )\circ (\theta _{\sigma (1)}\cdot \pi _{\sigma
(1)},\theta _{\sigma (2)}\cdot \pi _{\sigma (2)},\cdots ,\theta _{\sigma
(n)}\cdot \pi _{\sigma (n)}) = (\theta \circ (\theta _{1},\theta _{2},\cdots,
\theta _{n}))\cdot (\widetilde{\sigma }\cdot (\pi _{\sigma (1)},
\pi _{\sigma (2)},\cdots ,\pi_{\sigma (n)}))$\\
for all $n$,$k_{i}\in \mathbb{N}$, $a\in \mathcal{C}_{0}$, $\underline{a}%
=(a_{1},a_{2},\cdots ,a_{n})\in (\mathcal{C}_{0})^{n}$, $\underline{a_{i}}%
\in (\mathcal{C}_{0})^{k_{i}}$, $\theta \in \mathcal{C}(\underline{a};a)$, $%
\theta _{i}\in \mathcal{C}(\underline{a_{i}};a_{i})$, $\sigma \in S_{n}$, $%
\pi _{i}\in S_{k_{i}}$, for $1\leq i\leq n$, where $\widetilde{\sigma }$ and
$(\pi _{\sigma (1)},\pi _{\sigma (2)},\cdots ,\pi _{\sigma (n)}))$ are
permutations in $S_{k_{1}+k_{2}+\cdots +k_{n}}$ defined by:\\
$\widetilde{\sigma }\left( j+\overset{i-1}{\underset{l=0}{\sum }}k_{\sigma
(l)}\right) =
j+\overset{\sigma (i)-1}{\underset{l=0}{\sum}}k_{l} (\pi _{\sigma
(1)},\pi_{\sigma (2)},\cdots ,\pi_{\sigma (n)}))\left( j+%
\overset{i-1}{\underset{l=0}{\sum }}k_{\sigma (l)}\right) =\left( \pi
_{\sigma (i)}(j)+\overset{i-1}{\underset{l=0}{\sum }}k_{\sigma (l)}\right)$\\
for all $1\leq i\leq n$, $1\leq j\leq k_{\sigma (i)}$ assuming $\sigma
(0)=0=k_{0}$.
\end{itemize}
\end{definition}
It will be easier to understand the axioms of symmetricity in pictorial
notation as in \cite{Leinster}.
\begin{remark}
Clearly, the multicategories $\mathcal{MS}et$, $\mathcal{MV}ec$ and the one
arising from a symmetric tensor category are symmetric.
\end{remark}
\begin{definition}
A multicategory $\mathcal{C}$ is said to have the structure of {\it empty
object} if for all $a\in \mathcal{C}_{0}$, there exists a class
$\mathcal{C}(\emptyset;a)$ such that the composition in $\mathcal{C}$ extends
in the following way:\\
for all $n\in \mathbb{N}$, $1\leq s\leq n$, $a\in \mathcal{C}_{0}$, $%
\underline{a}=(a_{1},a_{2},\cdots ,a_{n})\in (\mathcal{C}_{0})^{n}$, $\theta
\in \mathcal{C}(\underline{a};a)$, $\theta _{s}\in \mathcal{C}(\emptyset
;a_{s})$, and\\
for all $k_{i}\in \mathbb{N}$, $\underline{a_{i}}\in (\mathcal{C}%
_{0})^{k_{i}}$, $\theta _{i}\in \mathcal{C}(\underline{a_{i}};a_{i})$ where
$i\in \{1,2,\cdots ,n\} \setminus \{s\}$,\\
$\begin{array}{rcl}
\mathcal{C}(\underline{a};a)\times \mathcal{C}(\underline{a_{1}}; a_{1}) \times
\cdots \times \mathcal{C}(\emptyset ;a_{s})\times \cdots
\times \mathcal{C}(\underline{a_{n}};a_{n}) & \overset{\circ}{\longrightarrow}
& \mathcal{C}\left( \left( \underline{a_{1}}, \cdots,
\underline{a_{s-1}},\underline{a_{s+1}},\cdots ,\underline{a_{n}}
\right) ;a\right)\\
(\theta ,\theta _{1},\cdots
,\theta _{s},\cdots ,\theta _{n}) & \overset{\circ}{\longmapsto} &
\theta \circ (\theta _{1},\cdots ,\theta_{s},\cdots ,\theta _{n})
\end{array}$\\
such that this composition map is associative and $1_a \circ \theta =\theta$
for all $\theta \in \mathcal{C}(\emptyset;a)$.
\end{definition}
Both $\mathcal{MS}et$ and $\mathcal{MV}ec$ indeed have the structure of empty 
object; for instance, $\mathcal{MS}et ( \emptyset ; X) = X$ for any set $X$. 
We demand that a map of multicategories both having the structure of 
empty object, should preserve this structure.
\subsection{Bicategories} In this subsection, we will recall the definition of
\textit{bicategories}
and various other notions related to bicategories which will be useful in
section 3. Most of the materials in this section can be found in any
standard textbook on bicategories.
\begin{definition}
A bicategory $\mathcal{B}$ consists of:
\begin{itemize}
\item a class $\mathcal{B}_{0}$ whose elements are called objects or $0$-cells,
\item for each $A$, $B \in \mathcal{B}_{0}$, there exists a category
$\mathcal{B}(A,B)$ whose objects $f$ are called $1$-cells of $\mathcal{B}$ and
denoted by $A\overset{f}{\longrightarrow}B$ and whose morphisms $\gamma $ are
called $2$-cells of $\mathcal{B}$ and denoted by $f_{1}\overset{\gamma }{%
\longrightarrow }f_{2}$ where $f_{1}$, $f_{2}$ are $1$-cells in $\mathcal{B}%
\left( A,B\right) $,
\item for each $A$, $B$, $C\in \mathcal{B}_{0}$, there exists a functor $%
\otimes :\mathcal{B}(B,C)\times \mathcal{B}(A,B)\rightarrow \mathcal{B}(A,C)$,
\item Identity object: for each $A\in \mathcal{B}_{0}$, there exists an
object $1_{A}\in ob(\mathcal{B}(A,A))$ (the identity on $A$),
\item Associativity constraint: for each triple $A\overset{f}{%
\longrightarrow }B$, $B\overset{g}{\longrightarrow }C$, $C\overset{h}{%
\longrightarrow }D$ of $1$-cells, there exists an isomorphism $(h\otimes
g)\otimes f\overset{\alpha _{h,g,f}}{\longrightarrow }h\otimes (g\otimes f)$
in $Mor(\mathcal{B}(A,D))$,
\item Unit Constraints: for each $1$-cell $A\overset{f}{\longrightarrow }B$,
there exist isomorphisms $1_{B}\otimes f\overset{\lambda _{f}}{%
\longrightarrow }f$ and $f\otimes 1_{A}\overset{\rho _{f}}{\longrightarrow }%
f $ in $Mor(\mathcal{B}(A,B))$ such that $\alpha _{h,g,f}$, $\lambda _{f}$
and $\rho _{f}$ are natural in $h$, $g$, $f$, and satisfy the pentagon and
the triangle axioms (which are exactly similar to the ones in the definition
of a tensor category).
\end{itemize}
\end{definition}
On a bicategory $\mathcal{B}$, one can perform the operation $op$ (resp. $co$)
and obtain a new bicategory $\mathcal{B}^{op}$ (resp. $\mathcal{B}^{co}$) by
setting (i) $\mathcal{B}^{op}_0=\mathcal{B}_0=\mathcal{B}^{co}_0$, (ii) 
$\mathcal{B}^{op}(B,A)=
\mathcal{B}(A,B) = \left(\mathcal{B}^{co}(A,B) \right)^{op}$ as categories (where 
$op$ of a category is basically reversing the directions of the morphisms).

A bicategory will be called a \textit{strict }$2$\textit{-category} if the
associativity and the unit constraints are identities. An \textit{abelian}
(resp. \textit{semisimple}) bicategory $\mathcal{B}$ is a bicategory such
that $\mathcal{B}(A,B)$ is an abelian (resp. semisimple) category for every $%
A$, $B\in \mathcal{B}_{0}$ and the functor $\otimes $ is additive.
\begin{remark}
$\mathcal{B}(A,A)$ is a tensor category and $\mathcal{B}(A,B)$ is a
$(\mathcal{B}(B,B),\mathcal{B}(A,A))$-bimodule category for $0$-cells $A$, $B$.
(See \cite{ENO}, \cite{Ostrik} for definition of module category.)
\end{remark}
\begin{example}
A bicategory with only one $0$-cell is simply a tensor category.
\end{example}
\begin{example}
A bicategory can be obtained by taking rings as $0$-cells, $1$-cells $%
A\rightarrow B$ being $(B,A)$-bimodules and $2$-cells being bimodule maps.
The tensor functor is given by the obvious tensor product over a ring.
\end{example}
\begin{definition}
Let $\mathcal{B}$, $\mathcal{B}^{\prime }$ be bicategories. A {\it weak
functor} $ F=(F,\varphi ):\mathcal{B\rightarrow B}^{\prime }$ consists of:
\begin{itemize}
\item a function $F:\mathcal{B}_{0}\mathcal{\rightarrow B}_{0}^{\prime }$,
\item for all $A$,$B\in \mathcal{B}_{0}$, there exists a functor $F^{A,B}:%
\mathcal{B}(A,B)\rightarrow \mathcal{B}^{\prime }(F(A),F(B))$ written simply
as $F$,
\item for all $A$, $B$, $C\in \mathcal{B}_{0}$, there exists a natural
isomorphism $\varphi ^{A,B,C}:\otimes ^{\prime }\circ (F^{B,C}\times
F^{A,B})\rightarrow F^{A,C}\circ \otimes $ written simply as $\varphi $
(where $\otimes $ and $\otimes ^{\prime }$ are the tensor products of $%
\mathcal{B}$ and$\mathcal{\ B}^{\prime }$ respectively),
\item for all $A\in \mathcal{B}_{0}$, there exists an invertible (with
respect to composition) $2$-cell $\varphi _{A}:1_{F(A)}\rightarrow F(1_{A})$,
\end{itemize}
satisfying commutativity of certain diagrams (consisting of $2$-cells) which
are analogous to the hexagonal and rectangular diagrams appearing in the
definition of a tensor functor.
\end{definition}
\comment{
A \textit{coweak functor} $F:\mathcal{B\rightarrow B}^{\prime }$ satisfies
all conditions of a weak functor except the functor
$F^{A,B}:\mathcal{B}(A,B)\rightarrow \mathcal{B}^{\prime }(F(A),F(B))$ is
contravariant and the natural isomorphism $\varphi^{A,B,C}$ is replaced by
$\varphi ^{A,B,C}:\otimes ^{\prime }\circ (F^{A,B}\times
F^{B,C})\rightarrow F^{A,C}\circ \otimes \circ (flip)$ where $(flip)$ is the
functor obtained by interchanging the co-ordinates.}
\begin{definition}
Let $F=(F,\varphi )$, $G=(G,\psi ):\mathcal{B\rightarrow B}^{\prime }$ be
weak functors. A weak transformation $\sigma :F\rightarrow G$ consists of:
\begin{itemize}
\item for all $A\in \mathcal{B}_{0}$, there exists a $1$-cell $\sigma
_{A}\in ob(\mathcal{B(}F(A),G(A))$,
\item for all $A$, $B\in \mathcal{B}_{0}$, there exists a natural
transformation $\sigma ^{A,B}:(\sigma _{B}\otimes ^{\prime
}F^{A,B})\rightarrow G^{A,B}\otimes ^{\prime }\sigma _{A}$ written simply as
$\sigma $ (where $(\sigma _{B}\otimes ^{\prime }F^{A,B})$, $G^{A,B}\otimes
^{\prime }\sigma _{A}:\mathcal{B(}A,B)\rightarrow \mathcal{B(}F(A),G(B))$
are functors defined in the obvious way), satisfying the follwoing:
\item for all $x\in ob(\mathcal{B}(B,C))$, $y\in ob(\mathcal{B}(A,B))$ where
$A$, $B$, $C\in \mathcal{B}_{0}$, the following diagrams commute:%
\begin{equation*}
\begin{tabular}{ccccc}
$\sigma _{C}\otimes ^{\prime }F(x)\otimes ^{\prime }F(y)$ & $\overset{\sigma
_{x}\otimes ^{\prime }1_{F(y)}}{\longrightarrow }$ & $G(x)\otimes ^{\prime
}\sigma _{B}\otimes ^{\prime }F(y)$ & $\overset{1_{G(x)}\otimes ^{\prime
}\sigma _{y}}{\longrightarrow }$ & $G(x)\otimes ^{\prime }G(y)\otimes
^{\prime }\sigma _{A}$ \\
$1_{\sigma _{C}}\otimes ^{\prime }\varphi _{x,y}\downarrow $ &  &  &  & $%
\downarrow \psi _{x,y}\otimes ^{\prime }1_{\sigma _{A}}$ \\
$\sigma _{C}\otimes ^{\prime }F(x\otimes y)$ &  & $\underset{
\sigma _{x\otimes y}}{\longrightarrow }$ &  & $G(x\otimes y)\otimes^{\prime}
\sigma _{A}$
\end{tabular}
\end{equation*}
\begin{equation*}
\begin{tabular}{ccccc}
$\sigma _{A}\otimes ^{\prime }1_{F(A)}$ & $\overset{\rho _{\sigma
_{A}}^{\prime }}{\longrightarrow }$ & $\sigma _{A}$ & $\overset{\lambda
_{\sigma _{A}}^{\prime }}{\longleftarrow }$ & $1_{G(A)}\otimes ^{\prime
}\sigma _{A}$\\
$1_{\sigma _{A}}\otimes ^{\prime }\varphi _{A}\downarrow $ &  &  &  & $
\downarrow \psi _{A}\otimes ^{\prime }1_{\sigma _{A}}$ \\
$\sigma _{A}\otimes ^{\prime }F(1_{A})$ &  & $\underset{\sigma _{1_{A}}}{
\longrightarrow }$ &  & $G(1_{A})\otimes ^{\prime }\sigma _{A}$
\end{tabular}
\end{equation*}
where $\lambda ^{\prime }$, $\rho ^{\prime }$ are the left and right unit
constraints of $\mathcal{B}^{\prime }$.
\end{itemize}
\end{definition}
\begin{remark}
Composition of weak functors and weak transformations follows
exactly from composition of functors and natural transformations in
categories.
One can also extend the notion of natural isomorphisms in categories to weak
isomorphism in bicategories.
\end{remark}
\begin{theorem}
(Coherence Theorem for Bicategories) Let $\mathcal{B}$ be a bicategory. Then
there exists a strict $2$-category $\mathcal{B}^{\prime }$ and functors $F:%
\mathcal{B\rightarrow B}^{\prime }$, $G:\mathcal{B}^{\prime }\rightarrow
\mathcal{B}$ such that $id_{\mathcal{B}}$ (resp. $id_{\mathcal{B}^{\prime }}$%
) is weakly isomorphic to $G\circ F$ (resp. $F\circ G$).
\end{theorem}
See \cite{Leinster} for a proof.

Let $A\overset{f}{\longrightarrow }B$ be a $1$-cell in a bicategory $
\mathcal{B}$. A \textit{right} (resp. \textit{left}) \textit{dual of }$f$ is
an $1$-cell $B\overset{f^{\ast }}{\longrightarrow }A$ (resp. $B\overset{
^{\ast }f}{\longrightarrow }A$) such that there exists $2$-cells $f^{\ast
}\otimes f\overset{e_{f}}{\longrightarrow }1_{A}$ and $1_{B}\overset{c_{f}}{
\longrightarrow }f\otimes f^{\ast }$ (resp.
$f\otimes \: ^{\ast}f \overset{_{f}e
}{\longrightarrow }1_{B}$ and $1_{A}\overset{_{f}c}{\longrightarrow}\: ^{\ast
}f\otimes f$) such that the following identities (ignoring the associativity
and unit contraints) are satisfied:
\begin{equation*}
\begin{tabular}{rcl}
$(1_{f}\otimes e_{f})\circ \left( c_{f}\otimes 1_{f}\right) =1_{f}$ & and
& $(e_{f}\otimes 1_{f^{\ast }})\circ \left( 1_{f^{\ast }}\otimes c_{f}\right)
=1_{f^{\ast }}$\\
(resp. $(1_{f} \otimes \: _{f}e)\circ \left( _{f} c \otimes 1_{^{\ast}\! f}
\right)=1_{^{\ast} \! f}$ & and &
$(_{f} e\otimes 1_{^{\ast} \! f})\circ \left( 1_{f}\otimes \: _{f}c\right) =1_{f}$
)
\end{tabular}
\end{equation*}
(Here $e$ stands for evaluation and $c$ stands for coevaluation.) One can
show that two right (resp. left) duals are isomorphic via an isomorphism
which is compatible with the evaluation and coevaluation maps. A bicategory
is said to be \textit{rigid} if right and left dual exists for every $1$
-cell. Further, in a rigid bicategory $\mathcal{B}$, one can consider right
dual as a weak funtor $\ast =(\ast ,\varphi ):\mathcal{B}\rightarrow
\mathcal{B}^{op\, co}$ in the following way:
\begin{itemize}
\item for each $1$-cell $f$, we fix a triple $(f^{\ast },e_{f},c_{f})$ so
that when $f=1_{A}$ where $A\in \mathcal{B}_{0}$, then $f^{\ast }=1_{A}$, $
e_{f}=\lambda _{1_{A}}$ ($=\rho _{1_{A}}$, see \cite{Kassel} for proof), $
c_{f}=\lambda _{1_{A}}^{-1}=\rho _{1_{A}}^{-1}$,
\item $\ast $ induces identity map on $\mathcal{B}_{0}$,
\item for all $A$, $B\in \mathcal{B}_{0}$, $f$, $g\in ob(\mathcal{B}(A,B))$
and $2$-cell $\gamma :f\rightarrow g$, define the contravariant functor $
\ast :\mathcal{B}(A,B)\rightarrow \mathcal{B}(B,A)$ by $\ast (f)=f^{\ast }$
and $\ast (\gamma )$ denoted by $\gamma ^{\ast }$, is given by the
composition of the following $2$-cells
\begin{equation*}
g^{\ast }\overset{\rho _{g^{\ast }}^{-1}}{\longrightarrow }g^{\ast }\otimes
1_{A}\overset{1_{g^{\ast }}\otimes c_{f}}{\longrightarrow }g^{\ast }\otimes
f\otimes f^{\ast }\overset{1_{g^{\ast }}\otimes \gamma \otimes 1_{f^{\ast }}}
{\longrightarrow }g^{\ast }\otimes g\otimes f^{\ast }\overset{e_{g}\otimes
1_{f^{\ast }}}{\longrightarrow }1_{B}\otimes f^{\ast }\overset{\lambda
_{f^{\ast }}}{\longrightarrow }f^{\ast }\text{,}
\end{equation*}
\item for all $A$, $B$, $C\in \mathcal{B}_{0}$, the natural isomorphism $
\varphi ^{A,B,C}:\otimes \circ (\ast ^{A,B}\times \ast ^{B,C})\rightarrow
\ast ^{A,C}\circ \otimes \circ (flip)$ is defined by:\\
for $f\in ob(\mathcal{B}(A,B))$, $g\in ob(\mathcal{B}(B,C))$, the invertible
$2$-cell $\varphi _{f,g}$ is given by the composition of the following $2$
-cells\\
$ f^{\ast }\otimes g^* \overset{1_{(f^{\ast }\otimes g^{\ast})}\otimes
c_{(g\otimes f)}}{\longrightarrow} f^{\ast }\otimes g^{\ast }\otimes (g\otimes
f)\otimes (g\otimes f)^{\ast } \overset{1_{f^{\ast }}\otimes e_{g}\otimes
1_f \otimes
1_{(g\otimes f)^{\ast }}}{\longrightarrow}(f^{\ast }\otimes f)\otimes (g\otimes
f)^{\ast }\overset{e_{f}\otimes 1_{(g\otimes f)^{\ast }}}{\longrightarrow
}(g\otimes f)^{\ast}$
ignoring the associativity and the unit constraints necessary to make sense
of the composition,
\item for all $A\in \mathcal{B}_{0}$, the invertible $2$-cell $\varphi
_{A}:1_{A}\rightarrow 1_{A}$ is given by identity morphism on $1_{A}$.
\end{itemize}
Similarly, one can define a left dual functor in a rigid bicategory.
\section{Planar Algebras}
In this section, we will introduce a new example of a symmetric
multicategory, namely, the \textit{Planar Tangle Multicategory} ($\mathcal{P}
$) which possesses the additional structure of empty object\textit{.} Any
\textit{planar algebra}, in the sense of \cite{pln alg}, turns out to be a $%
\mathcal{P}$-algebra. In the end, we also exhibit some examples and define
more structures on a planar algebra.

Let us first define \textit{planar tangles }which are the building blocks of
the planar tangle multicategory. Fix $k_0 \in \mathbb{N}_{0}=\mathbb{N\cup \{}%
0\}$ and $\varepsilon \in \{+,-\}.$
\begin{definition}\label{Def of pln tang}
A $(k_0,\varepsilon_0)$-planar tangle is an isotopy
class of pictures containing:
\begin{itemize}
\item an external disc $D_0$ on the Euclidean plane $\mathbb{R}^{2}$ with
$2k_0$
distinct points on the boundary numbered clockwise,
\item finitely many (possibly zero) non-intersecting internal discs $D_{1}$,
$D_{2},\cdots,D_{n}$, lying in the interior of $D_0$ with $2k_{i}$ distinct
points on the boundary of $D_{i}$ numbered clockwise where $k_{i}\in \mathbb{%
N}_{0}$ for $1\leq i\leq n$,
\item a collection $\mathcal{S}$ of smooth non-intersecting oriented curves
(called strings) on
$\left[ D_0 \setminus \left( \overset{n}{\underset{i=1}{\cup }}
D_{i}\right) ^{o}\right] $ such that:

(a) each marked point on the boundaries of $D_0, D_{1}, \cdots , D_{n}$ is
connected to exactly one string,

(b) each string either has no end-points or has exactly two end-points on
the marked points,

(c) the orientations induced on each connected component of $\left[
{D_0}^{o}\setminus \left( \left( \overset{n}{\underset{i=1}{\cup }}D_{i}\right)
\cup \mathcal{S}\right) \right] $ by different bounding strings should be
same,
\item the orientation induced in the connected component of $\left[
{D_0}^{o}\setminus \left( \left( \overset{n}{\underset{i=1}{\cup }}D_{i}\right)
\cup \mathcal{S}\right) \right] $, adjacent to the first and the last marked
point on the boundary of $D_0$, should have orientation positive
(anti-clockwise) or negative (clockwise) according to the sign of
$\varepsilon_0$.
\end{itemize}
\end{definition}
\begin{remark}
For each $i\in \{1,2,\cdots ,n\}$, we can assign $\varepsilon _{i}\in
\{+,-\} $ to the internal disc $D_{i}$ depending on the orientation of the
connected component of $\left[ {D_0}^{o}\setminus \left( \left( \overset{n}{%
\underset{i=1}{\cup }}D_{i}\right) \cup \mathcal{S}\right) \right] $,
adjacent to the first and the last marked points on the boundary of $D_{i}$.
$(k_{i},\varepsilon _{i})$ will be called the colour of $D_{i}$ and $%
(k,\varepsilon )$ will be the colour of $D$.
\end{remark}
Sometimes, instead of numbering each marked point on the boundary of a disc
with colour $(k,\varepsilon )$, we will write $\varepsilon $ very close to
the boundary of the disc and in the connected component adjacent to the
first and the last points. The orientation of the strings is equivalent to
putting checkerboard shading on the connected components such that all
components with negative orientation get shaded.
\begin{figure}[h]
\includegraphics[scale=0.40]{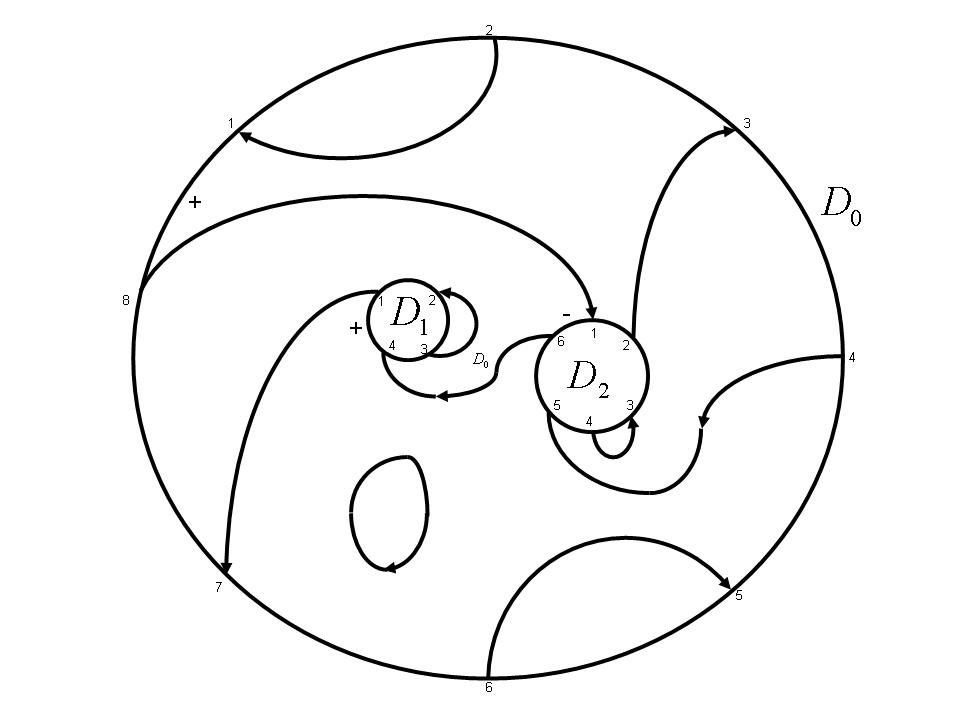}
\caption{Example of a $(4,+)$-planar
tangle with two internal discs $D_{1}$, $D_{2}$ with colours $(2,+)$, $(3,-)$
respectively.}
\label{egtang}
\end{figure}\\
Let $\mathcal{T}((k_{1},\varepsilon _{1}),(k_{2},\varepsilon _{2}),\cdots
,(k_{n},\varepsilon _{n});(k,\varepsilon ))$ be the set of $(k,\varepsilon )$%
-planar tangles with $n$ internal discs $D_{1},D_{2},\cdots ,D_{n}$ with
colours $(k_{1},\varepsilon _{1})$, $(k_{2},\varepsilon _{2})$, $\cdots $, $%
(k_{n},\varepsilon _{n})$ respectively, $\mathcal{T}(\emptyset
;(k,\varepsilon ))$ be the set of $(k,\varepsilon )$-planar tangles with no
internal disc and $\mathcal{T}_{(k,\varepsilon )}$ be the set of all $%
(k,\varepsilon )$-planar tangles . The composition of two tangles
$T\in \mathcal{T}((k_{1},\varepsilon _{1}),(k_{2},\varepsilon
_{2}),\cdots ,(k_{n},\varepsilon _{n});(k,\varepsilon ))$ and $S\in \mathcal{%
T}((l_{1},\delta _{1}),(l_{2},\delta _{2}),\cdots ,(l_{m},\delta
_{m});(k_{i},\varepsilon _{i}))$ (resp. $S\in \mathcal{T}(\emptyset
;(k_{i},\varepsilon _{i}))$), denoted by
$\left( T\underset{i}{\circ }S\right) \in \mathcal{T}%
((k_{1},\varepsilon _{1}),\cdots ,(k_{i-1},\varepsilon _{i-1}),(l_{1},\delta
_{1}),\cdots ,(l_{m},\delta _{m}),(k_{i+1},\varepsilon _{i+1}),\cdots
,(k_{n},\varepsilon _{n});(k,\varepsilon ))$ (resp. $\left( T\underset{i}{%
\circ }S\right) \in $ $\mathcal{T}((k_{1},\varepsilon _{1}),\cdots
,(k_{i-1},\varepsilon _{i-1}),(k_{i+1},\varepsilon _{i+1}),\cdots
,(k_{n},\varepsilon _{n});(k,\varepsilon ))$), is obtained by gluing the
external boundary of $S$ with the boundary of the $i^{\text{th}}$ internal
disc of $T$ preserving the marked points on either of them with the help of
isotopy, and then erasing the common boundary.

The \textit{Planar Tangle Multicategory}, denoted by $\mathcal{P}$, is
defined as:
\begin{itemize}
\item {\it Objects}: $\mathcal{P}_{0}=\left\{ (k,\varepsilon ):k\in
\mathbb{N}_{0},\varepsilon \in \{+,-\}\right\} $,
\item {\it Morphisms}:
$\mathcal{P}((k_{1},\varepsilon _{1}),(k_{2},\varepsilon _{2}),\cdots
,(k_{n},\varepsilon _{n});(k,\varepsilon ))$ (resp. $\mathcal{P}(\emptyset
;(k,\varepsilon ))$) is the vector space generated by $\mathcal{T}%
((k_{1},\varepsilon _{1}),(k_{2},\varepsilon _{2}),\cdots
,(k_{n},\varepsilon _{n});(k,\varepsilon ))$ (resp. $\mathcal{T}(\emptyset
;(k,\varepsilon ))$) as a basis,
\item composition of morphisms is given by the multilinear extension of the
composition of tangles as described above,
\item the identity morphism $1_{(k,\varepsilon )}\in \mathcal{T}%
((k,\varepsilon );(k,\varepsilon ))\subset \mathcal{P}((k,\varepsilon
);(k,\varepsilon ))$ is given by the $(k,\varepsilon)$-planar tangle with
exactly one internal disc with colour $(k,\varepsilon )$, containing
precisely $2k$ strings such that $i^{\text{th}}$ point on the internal disc
is connected to the $i^{\text{th}}$ point on the external disc by a string
for $1\leq i<2k$.
\end{itemize}
\begin{figure}[h]
\includegraphics[scale=0.20]{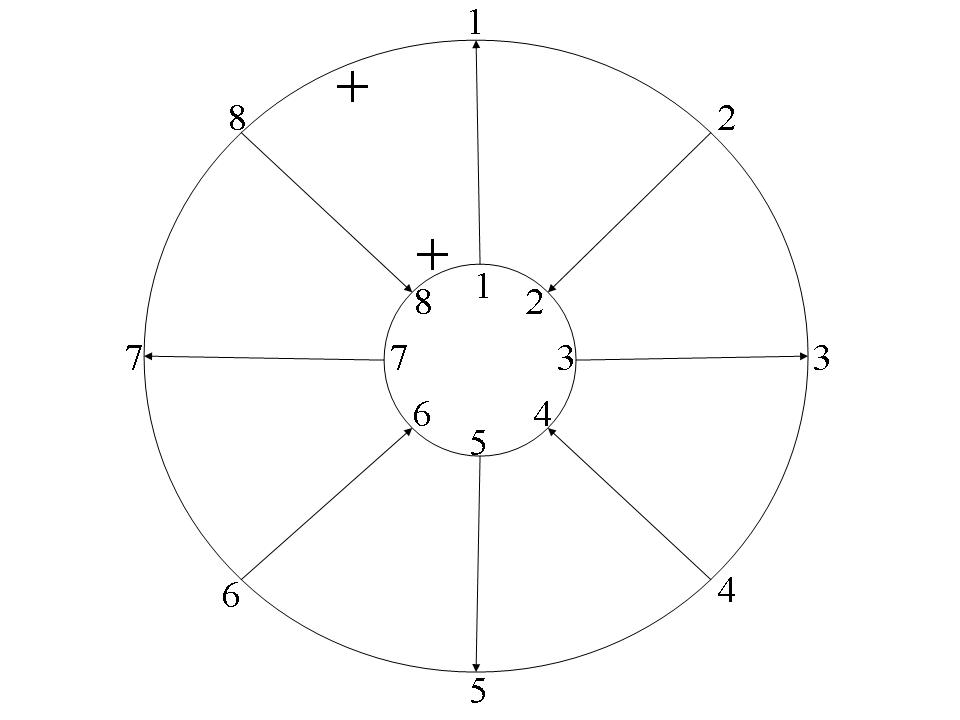}
\caption{$1_{(4,+)}\in \mathcal{P}((4,\varepsilon );(4,\varepsilon ))$}
\end{figure}
We leave the checking of associativity and identity axioms to the reader. A
moment's observation also reveals that $\mathcal{P}$ is symmetric and has
the structure of empty object.
\begin{definition}
A planar algebra $P$ is a $\mathcal{P}$-algebra, that is, a map of
multicategories from $\mathcal{P}$ to $\mathcal{MV}ec$.
\end{definition}
\begin{remark}
The first natural example of planar algebra is the $\mathcal{P}$-algebra which
takes the object $(k,\varepsilon)$ to $\mathcal{P}_{(k,\varepsilon)}$, and
morphisms $T$ to the multilinear map given by left-compostion of $T$. This is
called the \textit{Universal Planar Algebra} in \cite{pln alg}.
\end{remark}
\begin{remark}
For a planar algebra $P$, the collection of vector spaces $\left\{
P(k,\varepsilon )\right\} _{k\in \mathbb{N}_{0}}$ forms a unital filtered
algebra where $\varepsilon \in \{+,-\}$. The multiplication of $%
P(k,\varepsilon )$, inclusion of $P(k-1,\varepsilon )$ inside $
P(k,\varepsilon )$ and identity of $P(k,\varepsilon )$ are induced by the
following tangles:

\hskip -8mm
\includegraphics[scale=0.225]{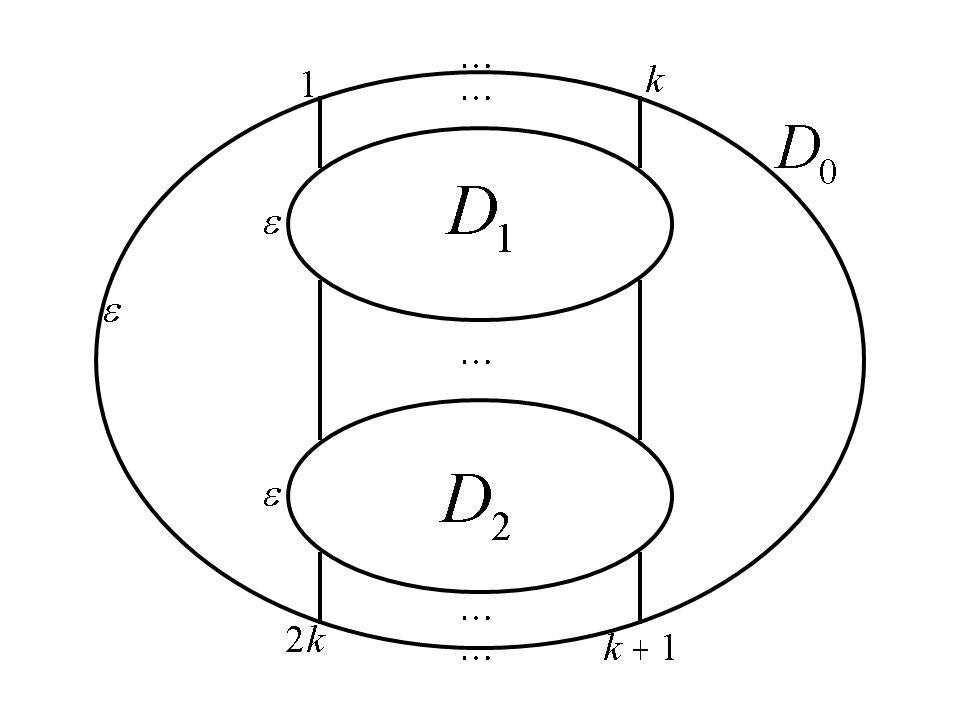}
\includegraphics[scale=0.225]{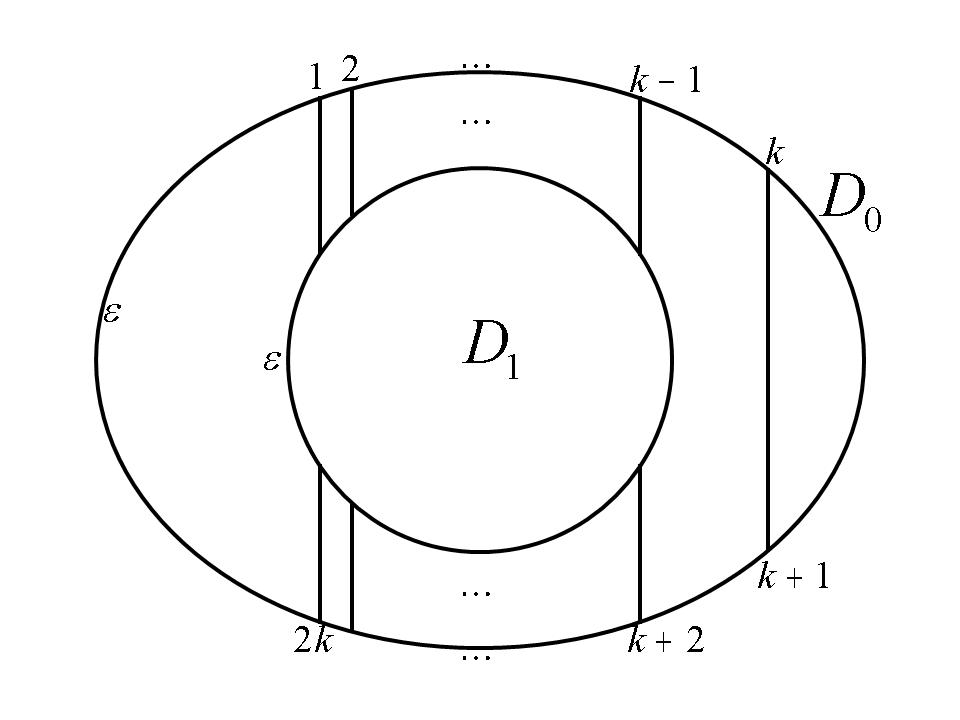}
\includegraphics[scale=0.225]{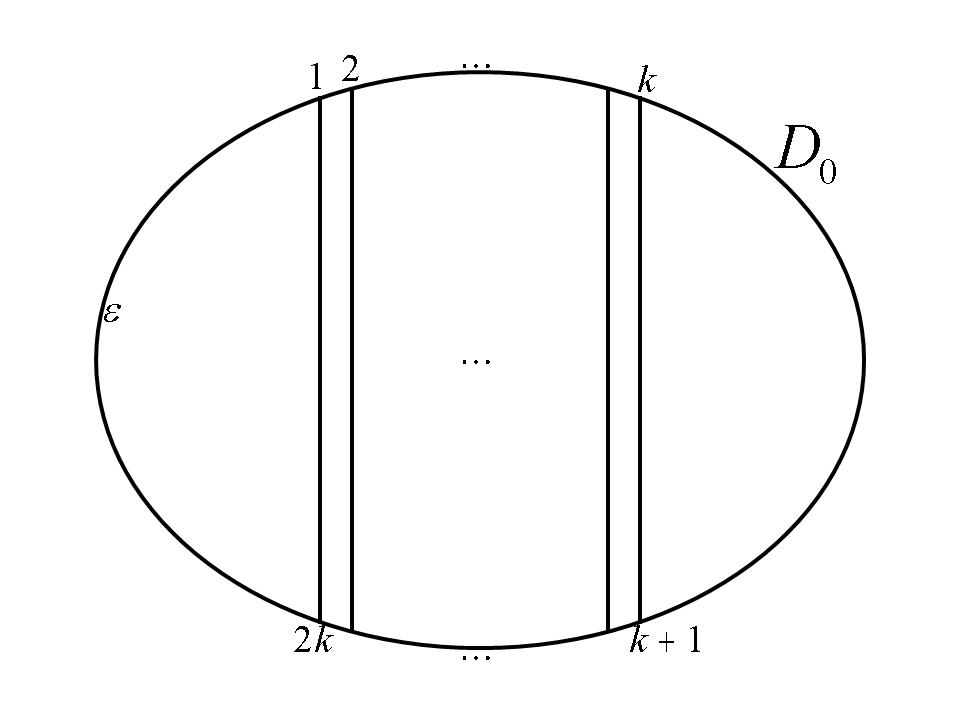}
respectively.
\end{remark}
\comment{
\begin{example}
\bigskip \textbf{Temperley-Lieb Planar Algebra with modulus }$(\delta
_{+,}\delta _{-})$

Let $\delta _{+}$, $\delta _{-}\in \mathbb{C}$. Consider the subpace $%
\mathcal{W}_{(k,\varepsilon )}$ of $\mathcal{P}(\emptyset ;(k,\varepsilon ))$
generated by elements of the form $T_{1}-\delta _{+}^{m}\delta _{-}^{n}T_{2}$
where $T_{1}$, $T_{2}\in \mathcal{T}(\emptyset ;(k,\varepsilon ))$ such that
$T_{1}$ can be isotopically obtained from $T_{2}$ by attaching $m$ many
loops oriented clockwise and $n$ many loops oriented anti-clockwise. Set $%
(TL_{(\delta _{+,}\delta _{-})})_{(k,\varepsilon )}=\frac{\mathcal{P}%
(\emptyset ;(k,\varepsilon ))}{\mathcal{W}_{(k,\varepsilon )}}$.

Define the map of multicategories $TL_{(\delta _{+,}\delta _{-})}:\mathcal{P}%
\rightarrow \mathcal{MV}ec$ by

$TL_{(\delta _{+,}\delta _{-})}(k,\varepsilon )=(TL_{(\delta _{+,}\delta
_{-})})_{(k,\varepsilon )}$,

$TL_{(\delta _{+,}\delta _{-})}(X)([Y_{1}],[Y_{2}],\cdots ,[Y_{n}])=[X\circ
(Y_{1},Y_{2},\cdots ,Y_{n})]$

where $X\in \mathcal{P}((k_{1},\varepsilon _{1}),(k_{2},\varepsilon
_{2}),\cdots ,(k_{n},\varepsilon _{n});(k,\varepsilon ))$, $Y_{i}\in
\mathcal{P}(\emptyset ;(k_{i},\varepsilon _{i}))$ for $1\leq i\leq n$. We
leave the checking of preserving associativity, identity and structure of
empty object to the reader. This is the first basic and non-trivial example
of a planar algebra.
\end{example}
}
We will now define more structures on a planar algebra. A planar algebra $P$
is said to be \textit{connected} (resp. \textit{locally finite}) if $\dim
(P(0,+))=1=\dim (P(0,-))$ (resp. $\dim (P(k,\varepsilon ))<\infty $ for all $%
(k,\varepsilon )$). A planar algebra $P$ is said to have \textit{modulus} $%
(\delta _{+},\delta _{-})$ if $P(T)=\delta _{+}~P(T_{1})$ (resp. $%
P(T)=\delta _{-}~P(T_{1})$) where $T$ is a planar tangle with a contractible
loop oriented clockwise (resp. anti-clockwise) and $T_{1}$ is the tangle $T$
with the loop removed. A connected planar algebra $P$ is called \textit{%
spherical} if two tangles $T_{1}\in \mathcal{T}_{(0,\varepsilon )}$ and $%
T_{2}\in \mathcal{T}_{(0,\eta )}$ induce the same multilinear functional by
expressing the images of $P(T_{1})$ and $P(T_{2})$ as scalar multiples of
the identities of $P(0,\varepsilon )$ and $P(0,\eta )$ respectively whenever
one can obtain $T_{1}$ from $T_{2}$ after embedding them on the unit sphere
and using spherical isotopy.
\begin{remark}
If $P=\{P_n\}_{n \geq 0}$ is a planar algebra in the sense of Jones
(\cite{pln alg}) with modulus $(\delta_+,\delta_-)$ (where $Z_T$
denotes the
action of a tangle $T$ necessarily with positive colors for all discs in it),
then one can define a planar algebra $P:\mathcal{P} \rightarrow \mathcal{MV}ec$
via:

(i) $P(k,\varepsilon)=
\left\{
\begin{array}{ll}
P_k & \text{if } \varepsilon = +\\
Range(Z_{LCE_k}) & \text{if } \varepsilon = -
\end{array}
\right.$\\
where the tangle $LCE_k$ is given in Figure \ref{lcek},
\begin{figure}[h]
\includegraphics[scale=0.24]{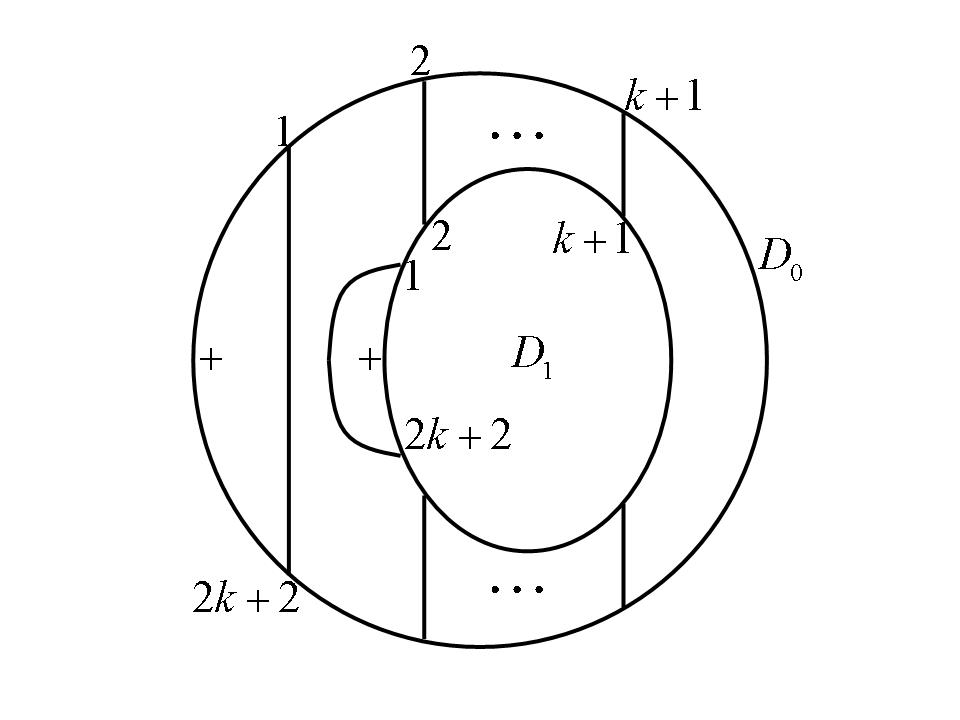}
\caption{Left conditional expectation tangle}
\label{lcek}
\end{figure}

(ii) for $T\in \mathcal{T}((k_{1},\varepsilon _{1}),(k_{2},\varepsilon
_{2}),\cdots ,(k_{n},\varepsilon _{n});(k_0,\varepsilon_0))$, first define
$T^\prime = U_{(k_0,\varepsilon_0)} \circ T \circ (S_{(k_1,\varepsilon_1)},
\cdots, S_{(k_n,\varepsilon_n)})$ where $S_{(k,+)}=1_{(k,+)}=U_{(k,+)}$, and
$S_{(k,-)}$ and $U_{(k,-)}$
are given by the tangles in Figure \ref{leftcapcup}.
\begin{figure}[h]
\includegraphics[scale=0.24]{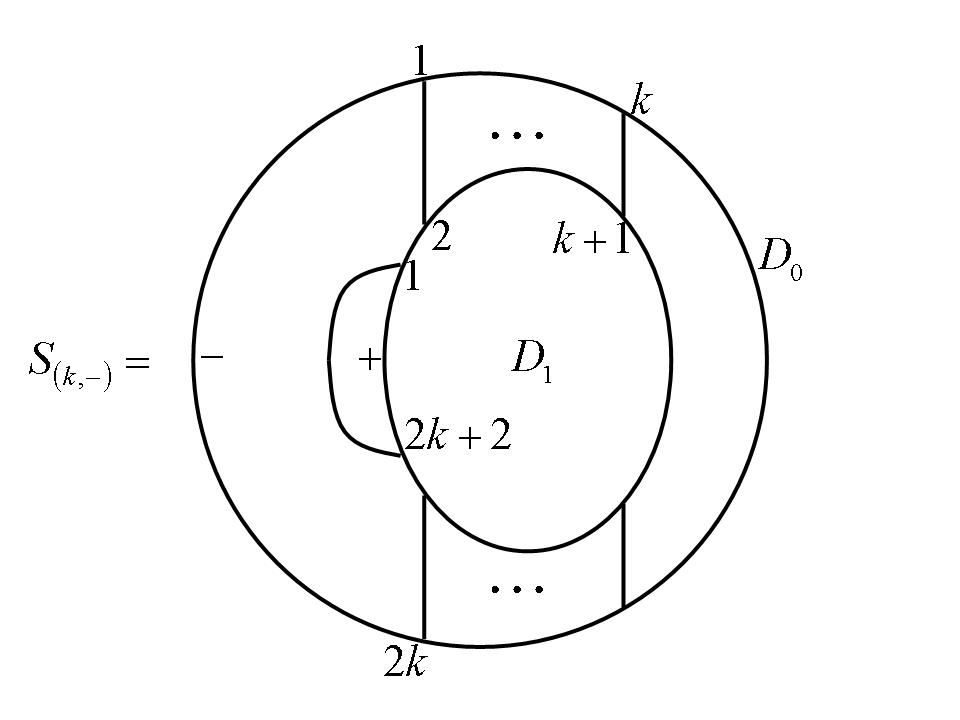}
\includegraphics[scale=0.24]{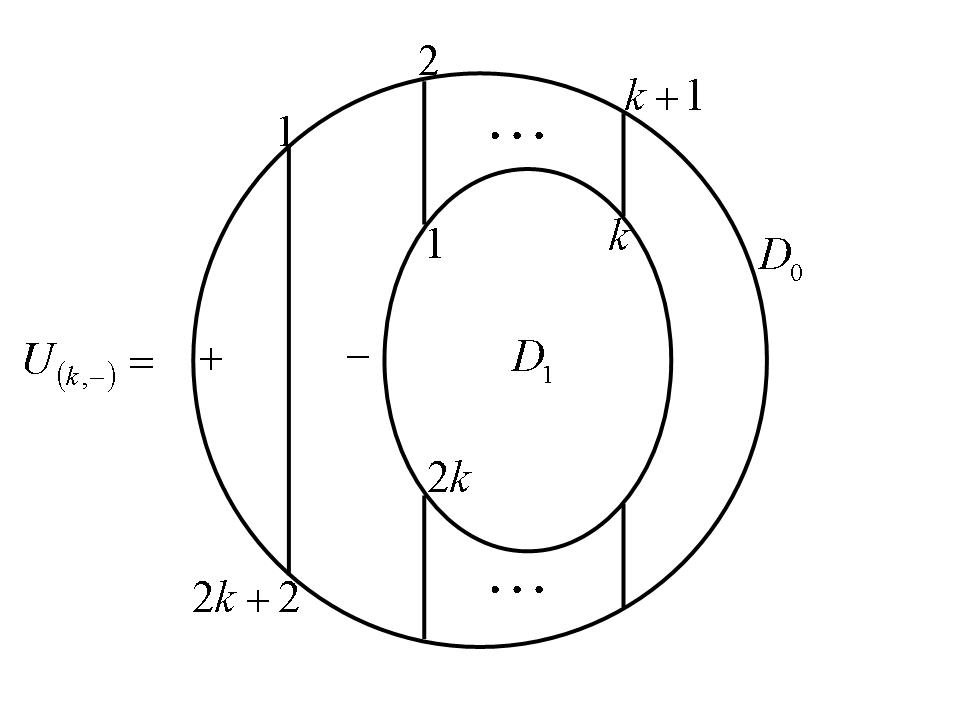}
\caption{}
\label{leftcapcup}
\end{figure}
Note that $T^\prime$ is a tangle with positive colors on each of its internal
discs. Set $P(T) = \delta_-^{-\abs{\{i \geq 1 \: : \:\varepsilon_i=-\}}} \;
Z_{T^\prime}|_{P(k_1,\varepsilon_1) \times \cdots \times 
P(k_n,\varepsilon_n)}$.

It is routine to check that $P$ preserves composition and identity. The 
definition of $P$ as a map of multicategories is motivated by Jones's 
definition of dual planar algebra (see \cite{pln alg}).
\end{remark}
If $T\in \mathcal{T}((k_{1},\varepsilon _{1}), \cdots ,
(k_{n},\varepsilon _{n}); (k,\varepsilon))$ (resp. $T\in
\mathcal{T}(\emptyset ;(k,\varepsilon ))$), then $T^{\ast } \in
\mathcal{T}((k_{1},\varepsilon _{1}),\cdots
,(k_{n},\varepsilon _{n});(k,\varepsilon ))$ (resp. $T^{\ast }\in \mathcal{T}
(\emptyset ;(k,\varepsilon ))$) is defined as the tangle obtained by
reflecting $T$ about any straight line not intersecting $T$, and the first
point of an internal (resp. external) disc of the reflected $T$ is taken to
be the reflected point of the last point of the corresponding internal
(resp. external) disc in $T$ such that the reflection preserves the colour
of each disc. For example, the $\ast$ of the tangle in Figure
\ref{egtang} is given by the tangle in Figure \ref{stareg}
\begin{figure}[h]
\includegraphics[scale=0.40]{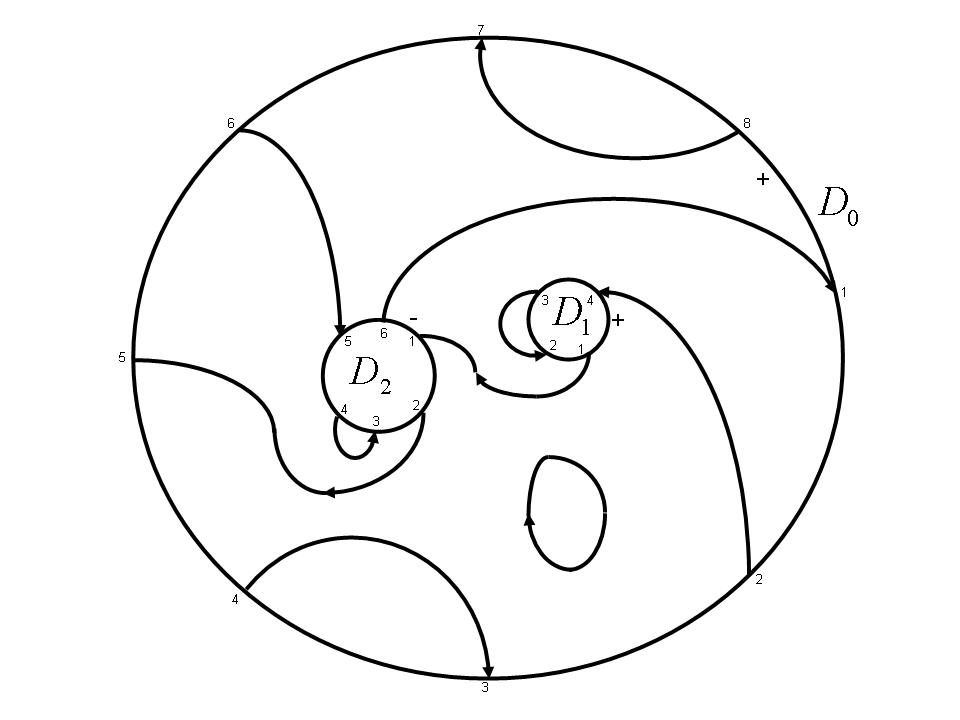}
\caption{}
\label{stareg}
\end{figure}
where we reflect the Figure \ref{egtang} about a vertical line. We
extend the map
$\mathcal{T}_{(k,\varepsilon )}\ni T\mapsto T^{\ast }\in \mathcal{T}
_{(k,\varepsilon )}$ conjugate linearly to $\ast :\mathcal{P}
_{(k,\varepsilon )}\rightarrow \mathcal{P}_{(k,\varepsilon )}$. It is clear
that $\ast $ is an involution. This makes $\left\{ \mathcal{P}
_{(k,\varepsilon )}\right\} _{k\in \mathbb{N}_{0}}$ into a unital filtered $
\ast $-algebra for $\varepsilon \in \{+,-\}$. $P$ is said to be a $\ast $
\textit{-planar algebra} (resp. $C^{\ast }$\textit{-planar algebra}) if $P$
is a planar algebra, $P(k,\varepsilon )$ is a $\ast $-algebra (resp. $
C^{\ast }$-algebra) for each colour $(k,\varepsilon )$ and the map $P$ is $
\ast $ preserving in the sense:\\
if $\theta \in \mathcal{P}((k_{1},\varepsilon _{1}),(k_{2},\varepsilon
_{2}),\cdots ,(k_{n},\varepsilon _{n});(k,\varepsilon ))$ (resp. $\mathcal{P}%
(\emptyset ;(k,\varepsilon ))$) and $f_{i}\in P(k_{i},\varepsilon _{i})$ for
$1\leq i\leq n$, then $P(\theta ^{\ast })(f_{1}^{\ast },\cdots ,f_{n}^{\ast
})=\left( P(\theta )(f_{1},\cdots ,f_{n})\right) ^{\ast }$.

A locally finite spherical $C^{\ast}$-planar algebra is called
\textit{subfactor-planar algebra}.
\begin{theorem}\label{jonthm}
(Jones) Any extremal subfactor with finite index gives rise to a
subfactor-planar algebra. Conversely, any subfactor planar algebra gives
rise to an extremal subfactor with finite index.
\end{theorem}
Jones proved the first part of Theorem \ref{jonthm}
(in \cite{pln alg}) by
prescribing an action of tangles on the standard invariant of a subfactor. 
However, Jones proved the converse using Popa's result on $\lambda$-lattices
(\cite{lambda lattice}). Very recently, another proof of the converse using 
planar algebra techniques appeared in \cite{JSW} and then in \cite{KoSu}.
\section{Planar Algebra arising from a Bicategory}
In this section, we will show how one can construct a planar algebra from a
$1$-cell of an abelian `pivotal' strict $2$-category with exactly two $0$-
cells.
The techniques used in this construction are motivated by Jones's construction
of planar algebra from a subfactor (in \cite{pln alg}).
\subsection{Construction of the Planar Algebra}
Before we proceed towards the construction, we will first state or deduce
some useful results and set up some notations.
\begin{definition}
A bicategory $\mathcal{B}$ is called \textit{pivotal} if $\mathcal{B}$ is
rigid and there exists a weak transformation $a:id_{\mathcal{B}}\rightarrow
\ast \ast $ such that $a_{\varepsilon }=1_{\varepsilon }\in ob(\mathcal{B}%
(\varepsilon ,\varepsilon ))$ for all $\varepsilon \in \mathcal{B}_{0}$,
where $\ast =(\ast ,K)$ is the right dual functor and $\ast \ast
=(\ast \ast ,J)$ is the weak functor $\ast \circ \ast $.
\end{definition}
From now on, we will consider only strict $2$-category instead of general
bicategories unless otherwise mentioned; however all results modified with
appropriate associativity and unit constraints, will hold even in the
absence of the `strict' assumption by the coherence theorem for
bicategories.

\begin{wrapfigure}{r}{40mm}
\begin{center}
\includegraphics[scale=0.175]{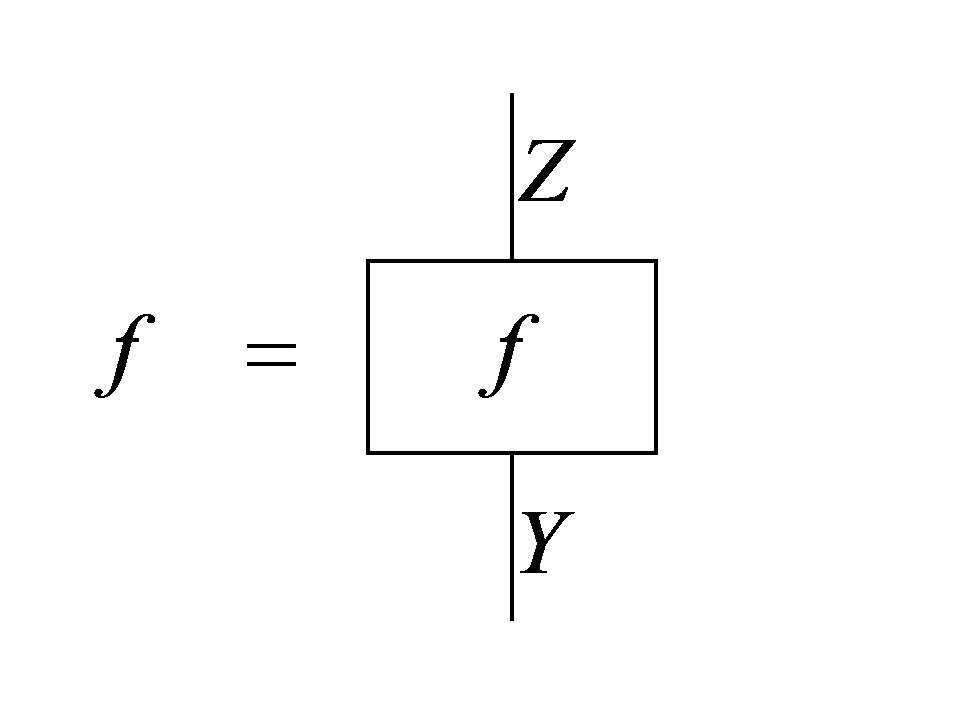}
\end{center}
\end{wrapfigure}
We next set up some pictorial notation to denote $2$-cells
which is analogous to the graphical calculus of morphisms in a tensor
category (see \cite{Kassel}, \cite{BaKi}). Let $\mathcal{B}$ be a pivotal
strict $2$-category as defined above. We denote a $2$-cell $f:Y\rightarrow Z$
by a rectangle labelled with $f$, placed on $\mathbb{R}^{2}$ so that one of
the sides is parallel to the $X$-axis and a vertical line segment labelled
with $Y$ (resp. $Z$) is attached to the top (resp. bottom) side of the
rectangle. Sometimes we will not label
the strings attached to a recatangle labelled with a $2$-cell; the $2$-cell
itself will induce the obvious labelling to the strings.

We list below
pictorial notations of several other $2$-cells which will be the main
constituents of the construction without describing them meticulously in
words like the way we described $f$ above.
\begin{center}
\includegraphics[scale=0.175]{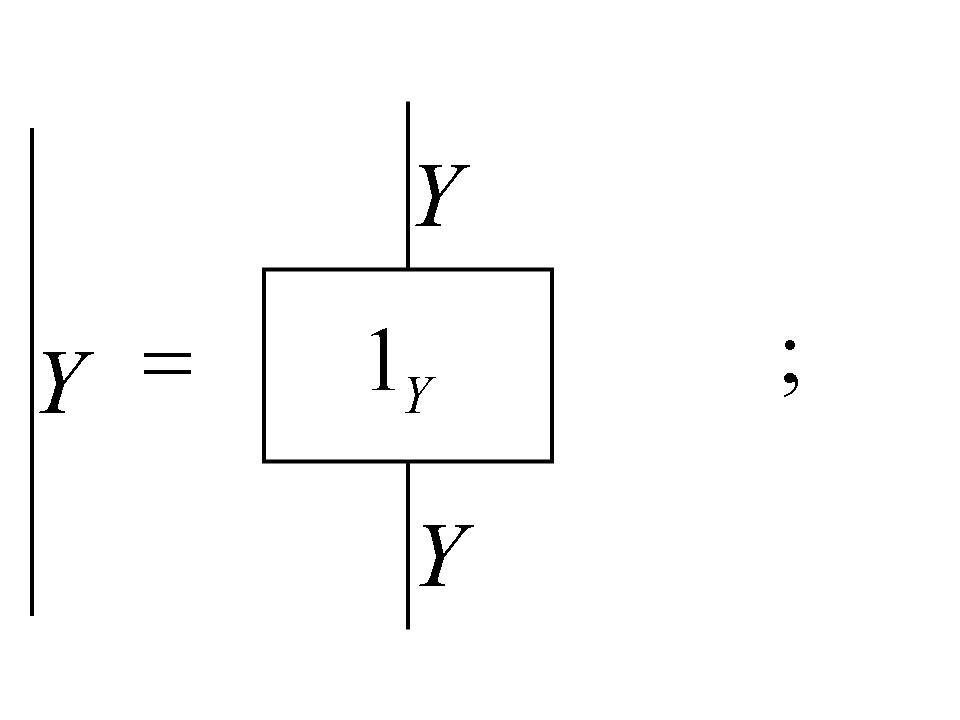}
\includegraphics[scale=0.175]{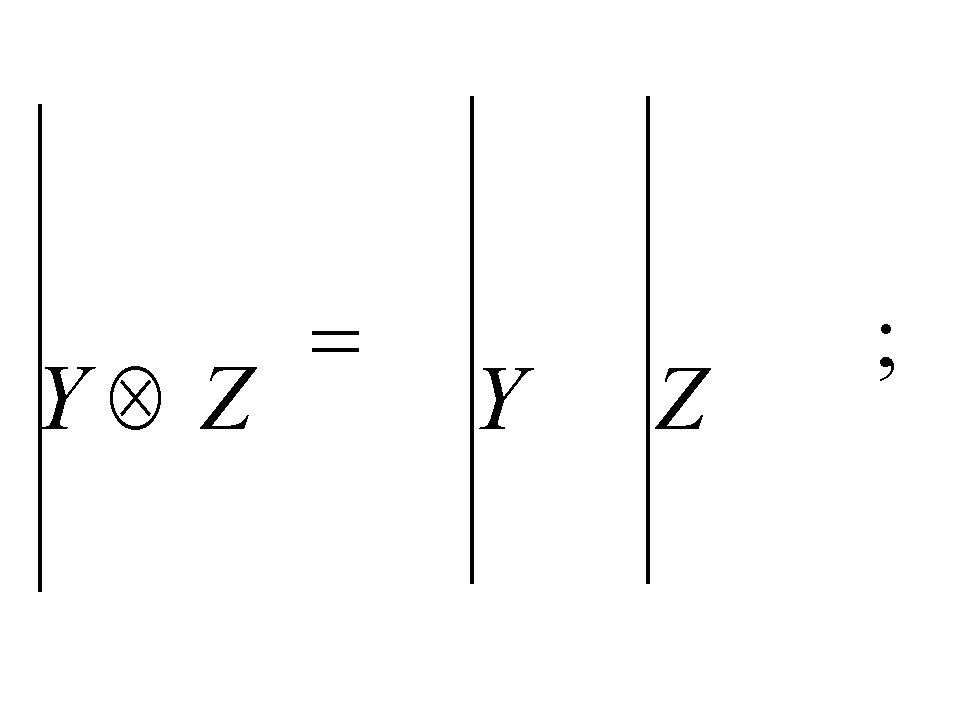}
\includegraphics[scale=0.175]{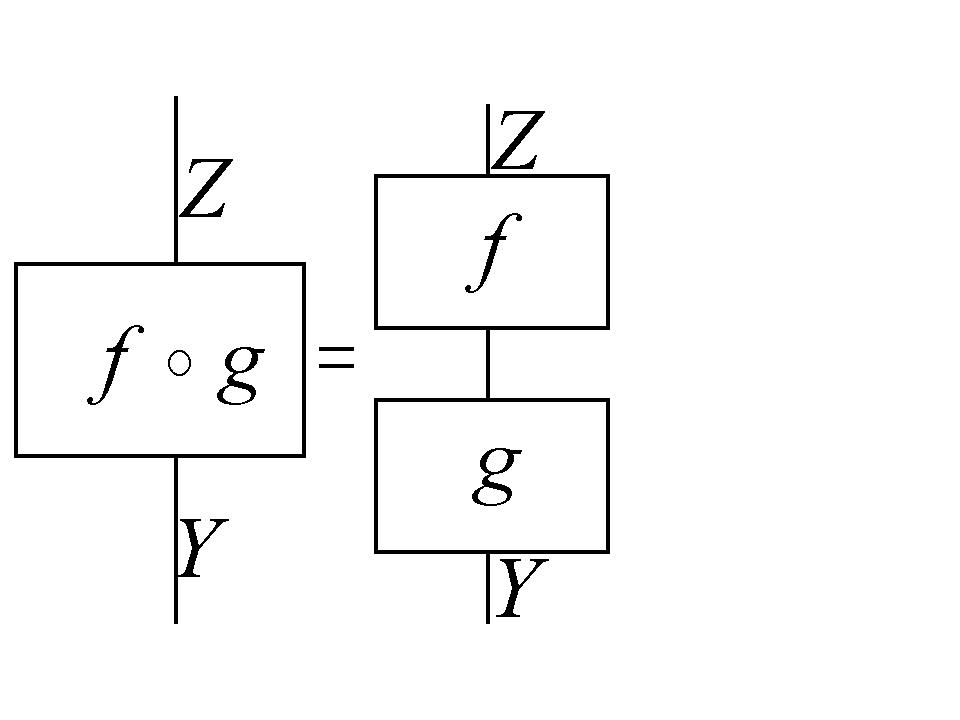}
\end{center}
\noindent where $Y$, $Z$ are $1$-cells and $f$, $g$ are
$2$-cells. To each local maximum or minimum of a string with an orientation
marked at the maximum or minimum and labelled with a $1$-cell $Y\in \mathcal{B}
(A,B)$, we associate a $2$-cell in the following way:
\begin{center}
\includegraphics[scale=0.35]{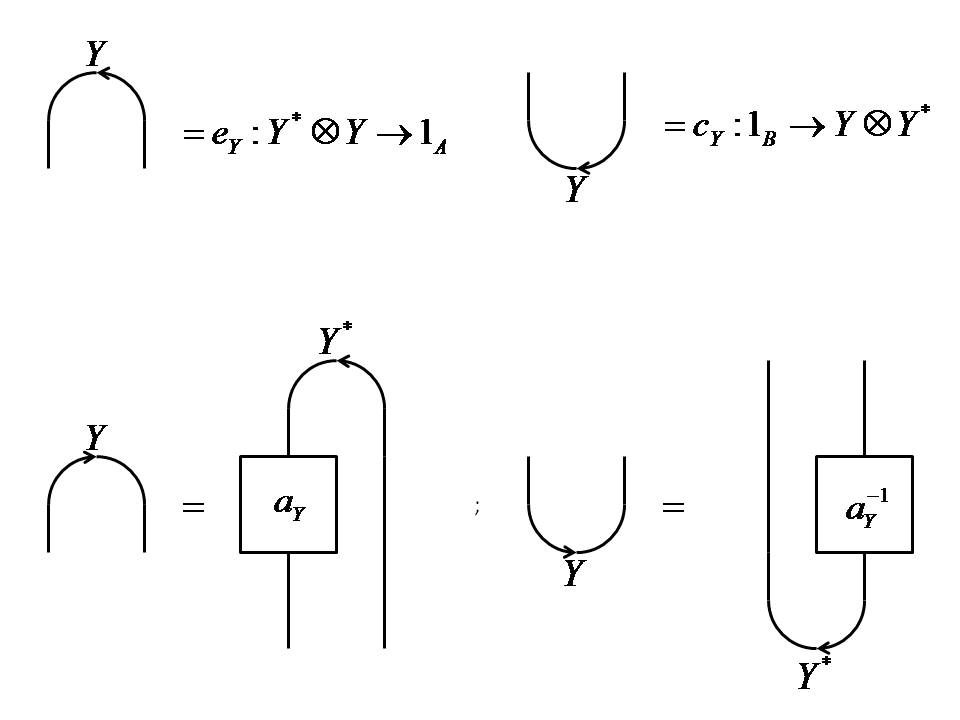}
\end{center}
We will next exhibit some easy consequences in terms of the pictorial notation.
\begin{lemma}\label{Basic Lemma}(i) For any $1$-cell $Y$,
\includegraphics[scale=0.2,bb= 0 250 750 500]
{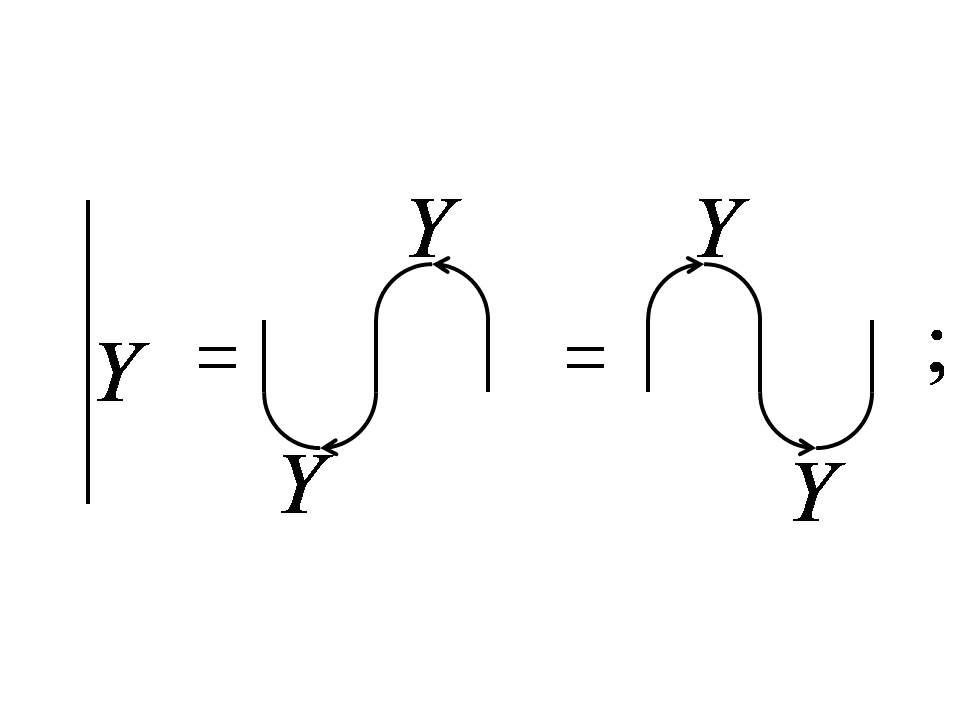}
\includegraphics[scale=0.2,bb= 0 250 750 500]
{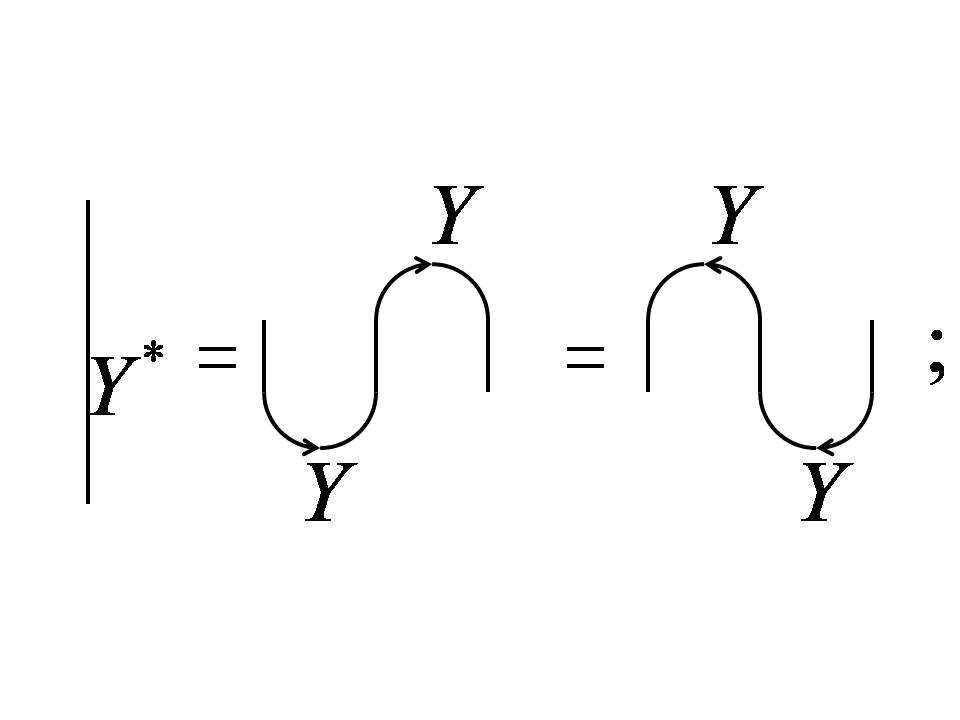}
\vskip 1.6cm
\noindent(ii) for any $2$-cell $f:Y\rightarrow Z$,
\includegraphics[scale=0.2,bb= 0 250 750 550]
{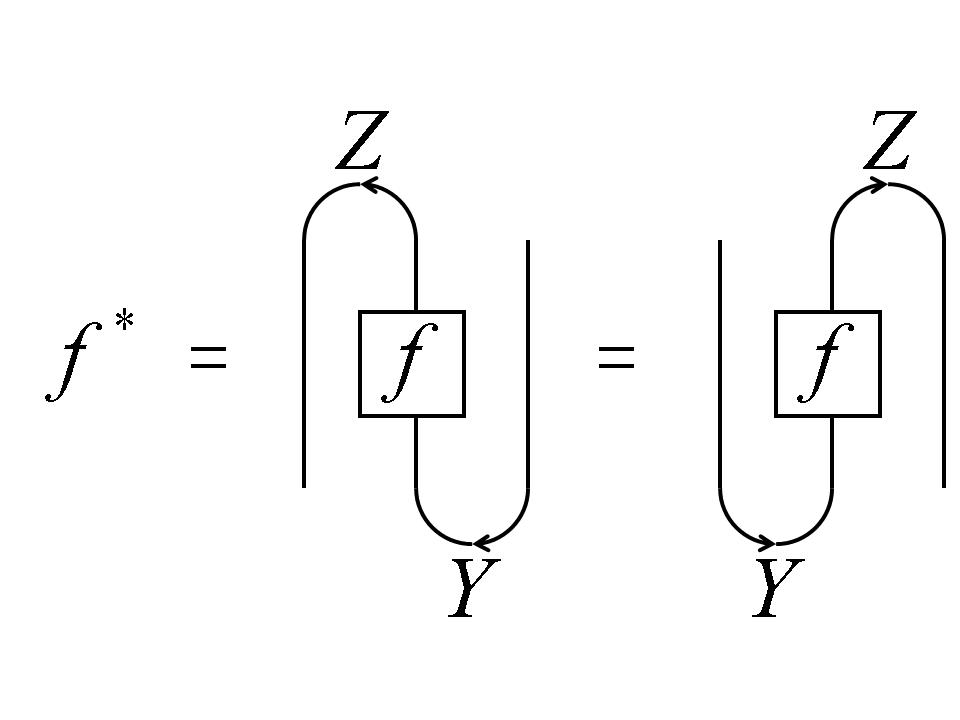} and
\includegraphics[scale=0.2,bb= 0 250 750 550]
{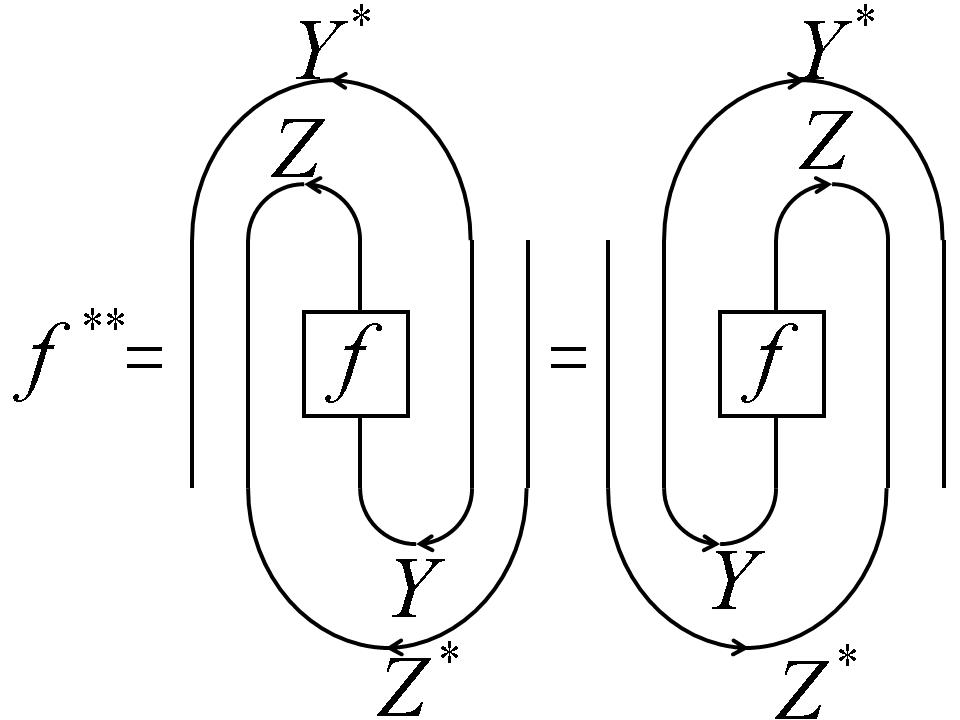}
\vskip 1.6cm
\noindent(iii) $K_{Y,Z} =$
\includegraphics[scale=0.2,bb= 0 250 750 550]
{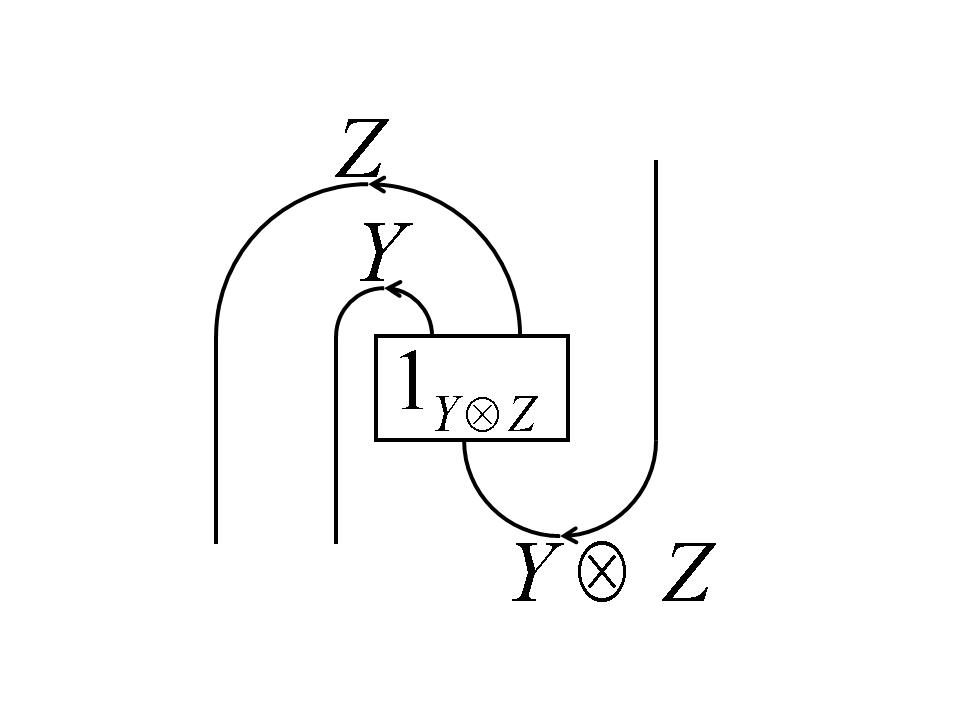} and $K_{Y,Z}^{-1}=$
\includegraphics[scale=0.2,bb= 0 250 750 550]
{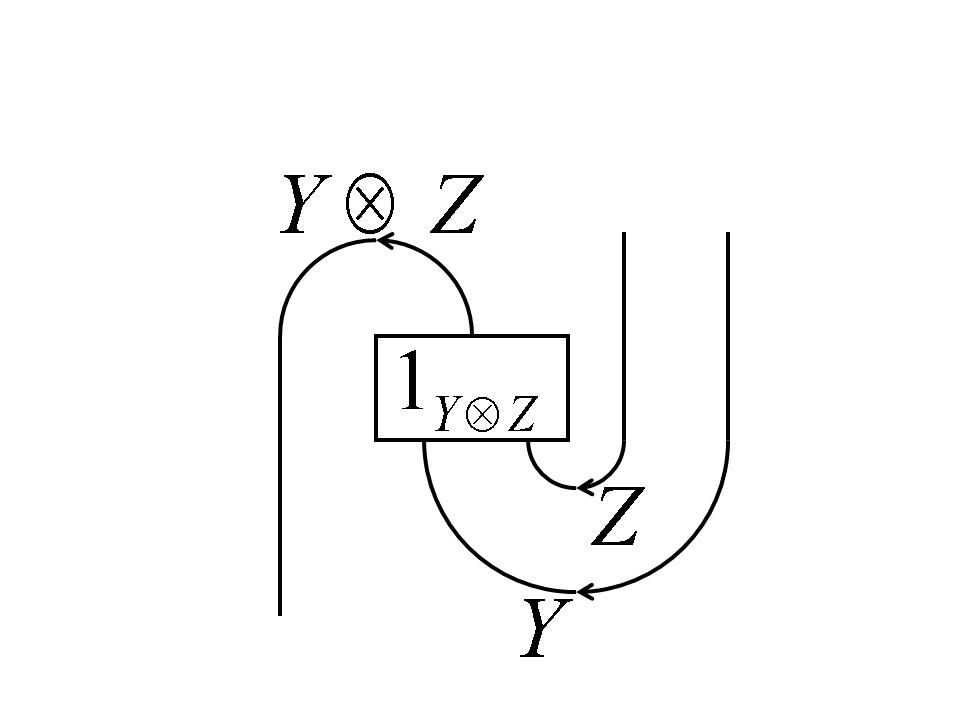}
\vskip 1.6cm
\noindent(iv) $J_{Y,Z}=$
\includegraphics[scale=0.2,bb= 0 250 750 550]
{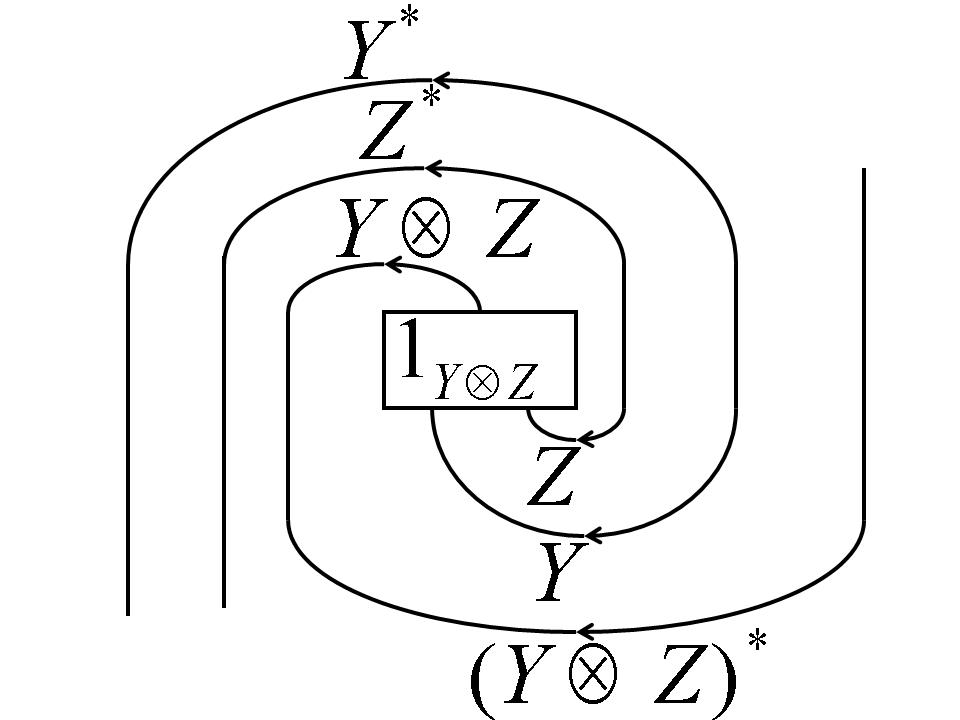}
and $J_{Y,Z}^{-1}=$
\includegraphics[scale=0.2,bb= 0 250 750 550]
{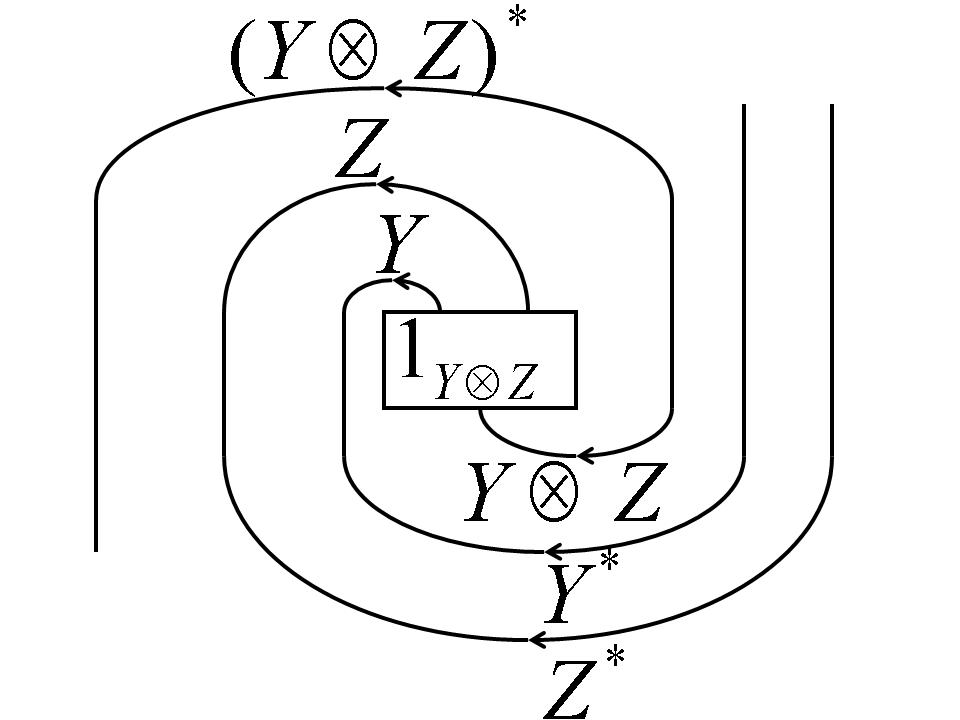}
\vskip 1.6cm
\noindent(v) $J_{Y,Z}\circ (a_{Y}\otimes a_{Z})=a_{Y\otimes Z}$ for all $1$-
cells $Y$
and $Z$.
\end{lemma}
\begin{proof}
(i) follows from the definition of $Y^{\ast }$ being the right dual of $Y$
and $a$ being invertible.

First part of (ii) follows from the definition of $f^{\ast }$ and naturality
of $a$ and the second part easily follows from the first one.

(iii) and (iv) follow from the way the weak functors $\ast $ and $\ast \ast $
are defined.

Definition of the pivotal structure $a$ implies (v).
\end{proof}
\begin{remark}
Parts (iii) and (iv) of the above lemma do not use the pivotal structure $a$
at all. However, with the help of pivotal structure, especially part (v)\ of
the above lemma, one may also prove the following:\\
$K_{Y,Z}\qquad =$
\includegraphics[scale=0.2,bb= 0 250 750 550]
{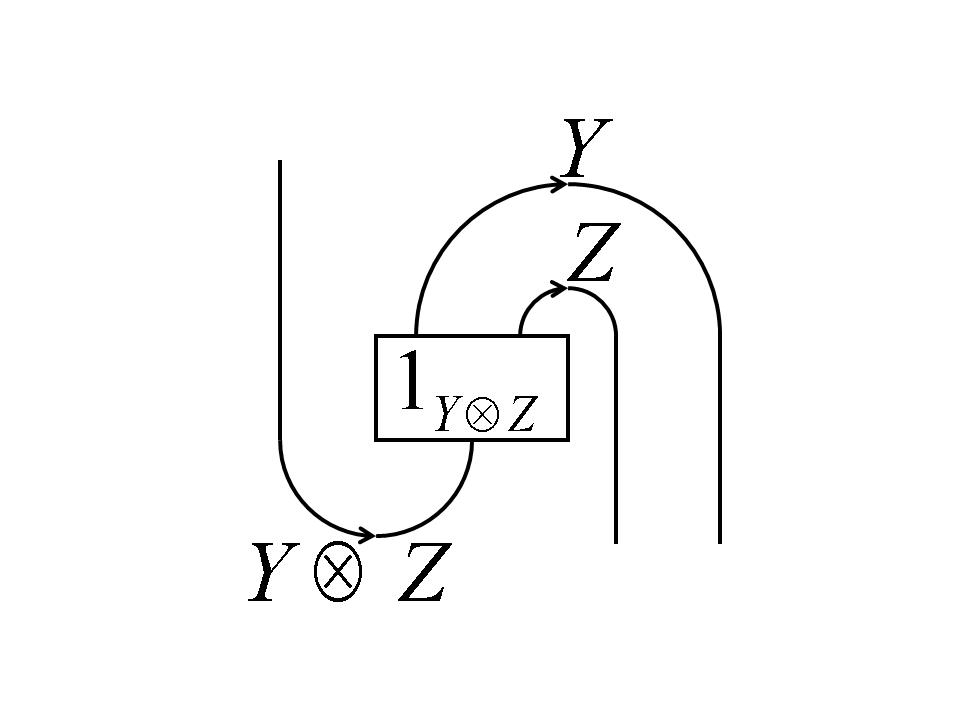}
and $J_{Y,Z}=$
\includegraphics[scale=0.2,bb= 0 250 750 550]
{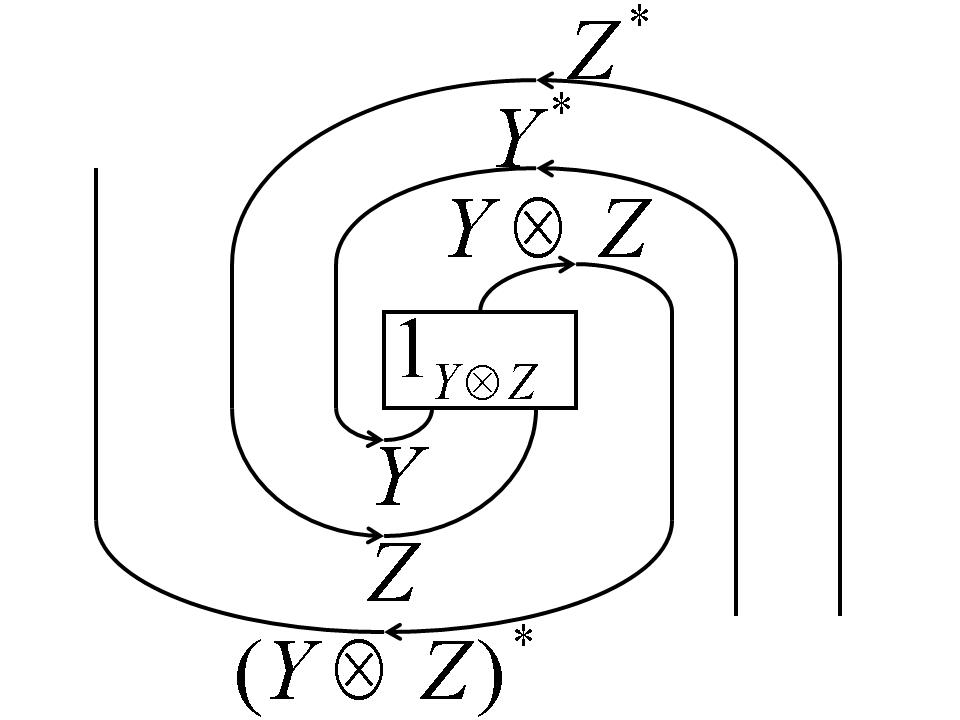}
\end{remark}
\vskip 1.6cm
Using the above graphical calculus, we immediately obtain the following relation
which will be useful later.
\begin{corollary}\label{ax*}
$a^{-1}_{X^{\ast}} = a^*_{X}$ for all $1$-cell $X$.
\end{corollary}
\begin{proposition}\label{Rotation Invariance}
For any $2$-cell $f:Y_{1}\otimes \cdots \otimes
Y_{n}\rightarrow Z_{1}\otimes \cdots \otimes Z_{m}$, the following
identities hold:
\includegraphics[scale=0.3,bb= 0 250 750 550]
{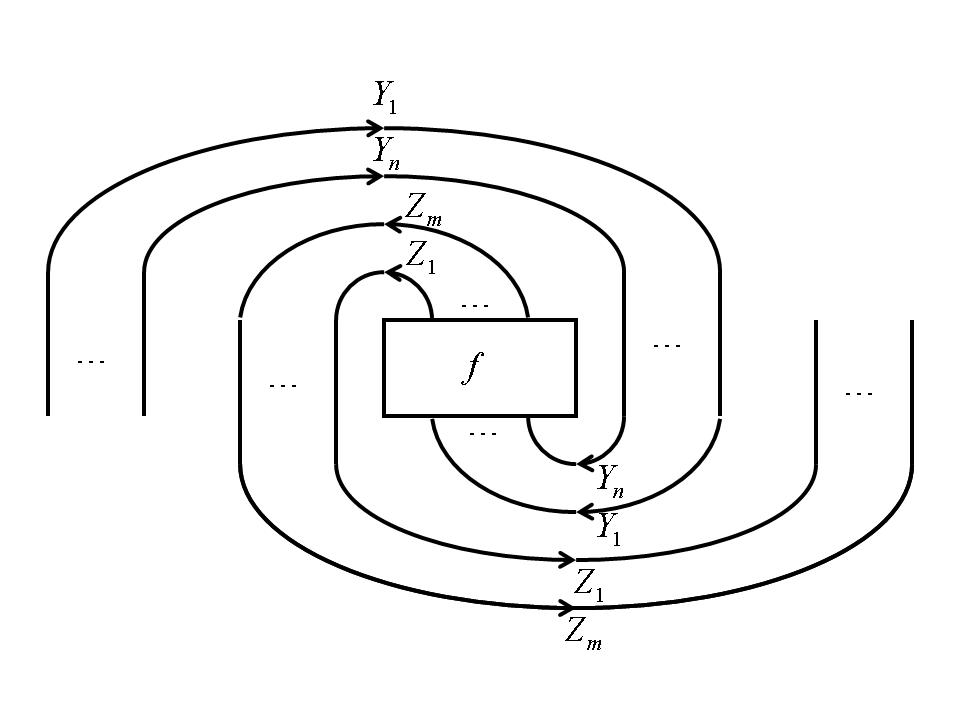}
$= \; f \; =$
\includegraphics[scale=0.3,bb= 0 250 750 550]
{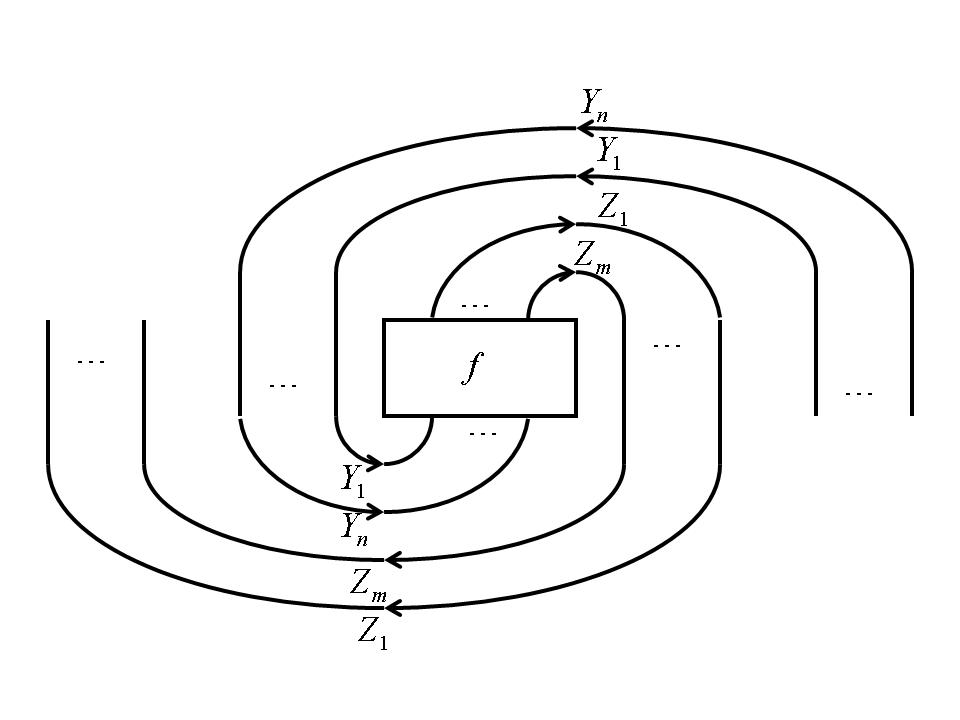}
\end{proposition}
\vskip 2.4cm
\begin{proof}
It is enough to show one of the identities (because applying the reverse
rotation and using Lemma \ref{Basic Lemma} (i), one can deduce the other
identity). We will sketch the proof of the first identity.

For the case $m=n=1$, the result follows trivially from the naturality of $a$.

Suppose $n=2$. Then LHS of the first identity\\
\newpage \noindent$=$
\includegraphics[scale=0.3,bb= 0 250 750 550]
{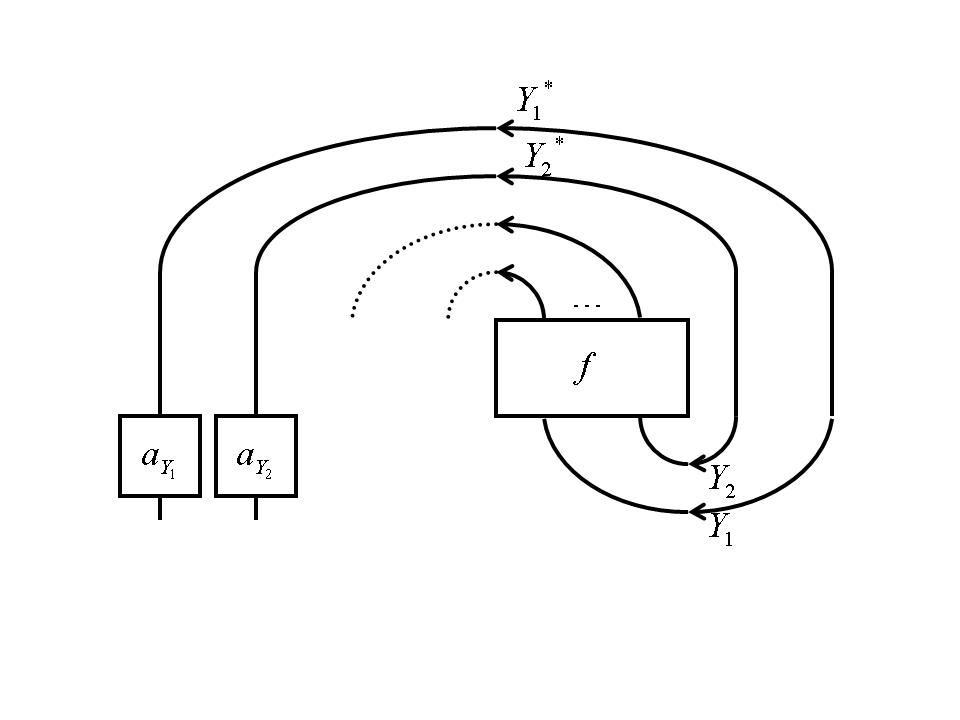}
$=$
\includegraphics[scale=0.3,bb= 0 250 750 550]
{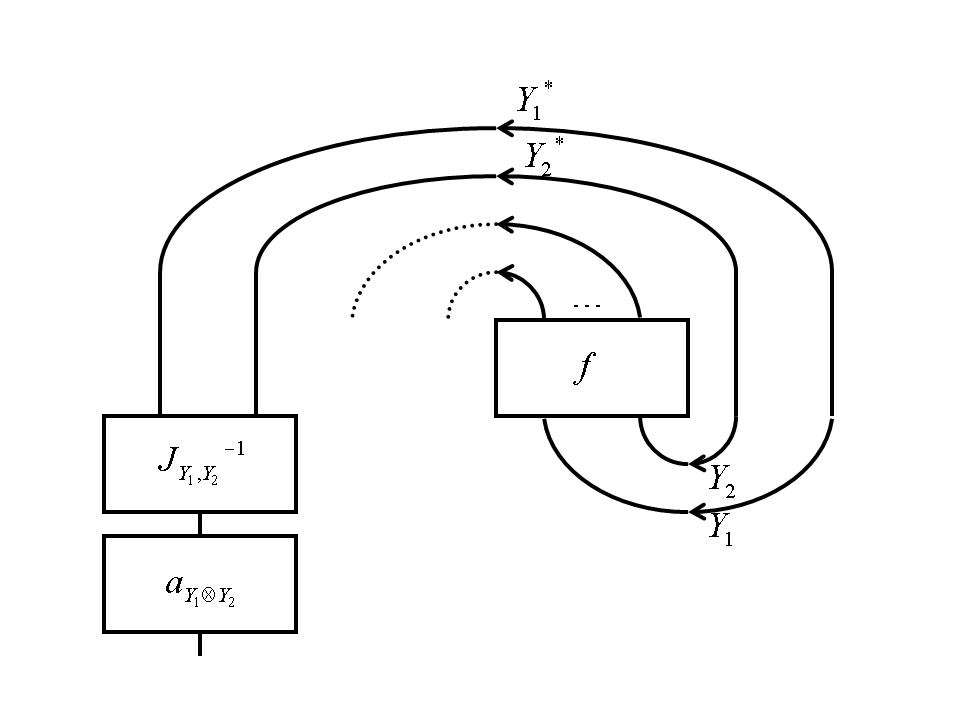}
\vskip 2.4cm
\noindent(using Lemma \ref{Basic Lemma} (v))\\
$=$
\includegraphics[scale=0.3,bb= 0 250 750 550]
{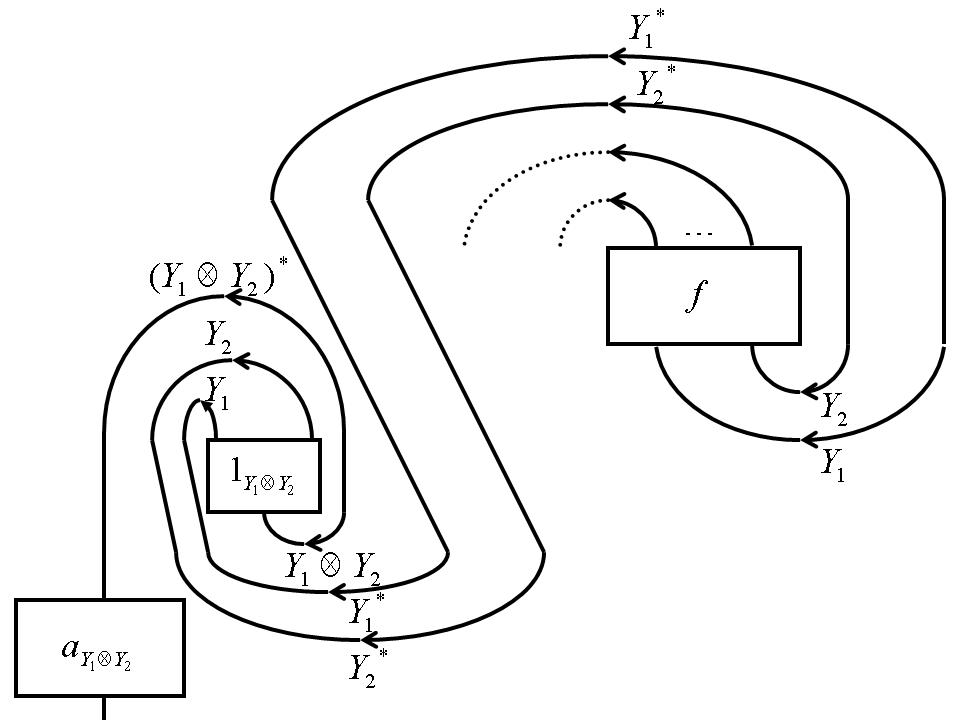}
$=$
\includegraphics[scale=0.3,bb= 0 250 750 550]
{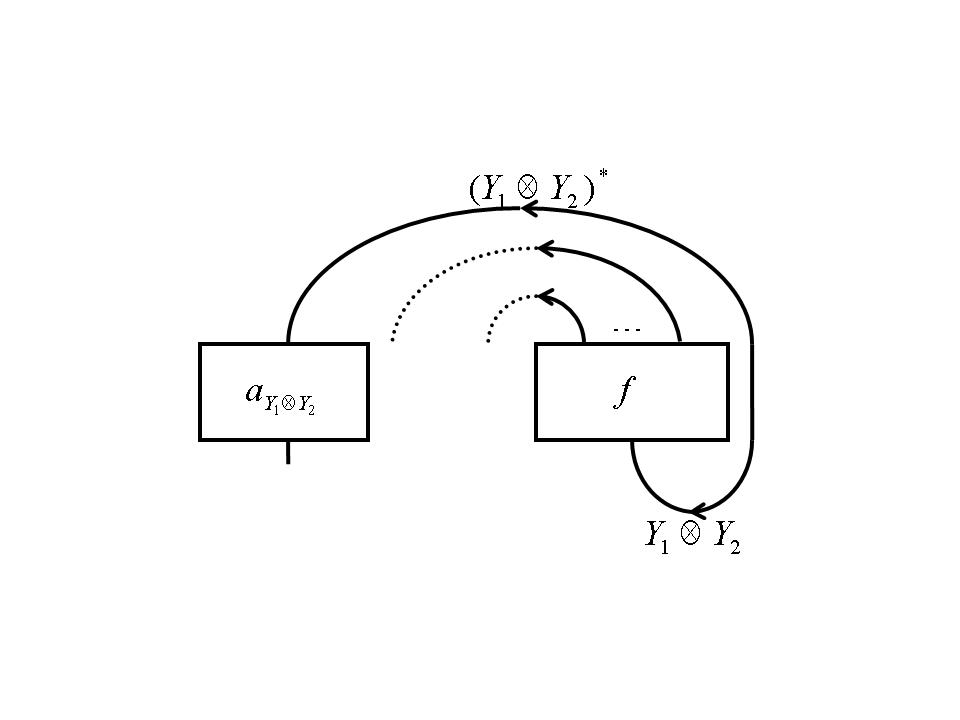}
\vskip 2.5cm
\noindent(using Lemma \ref{Basic Lemma} (i)).

For $n>2$, an analogous result (with $Y_1 \otimes Y_2$ replaced by
$Y_1 \otimes \cdots \otimes Y_n$) can be deduced by applying the above result
recursively. After working on the rest of the curves (emanating from the top of
the rectangle labelled with $f$) in the same way as above, the LHS of the first
identity\\
$=$
\includegraphics[scale=0.3,bb= 0 250 750 550]
{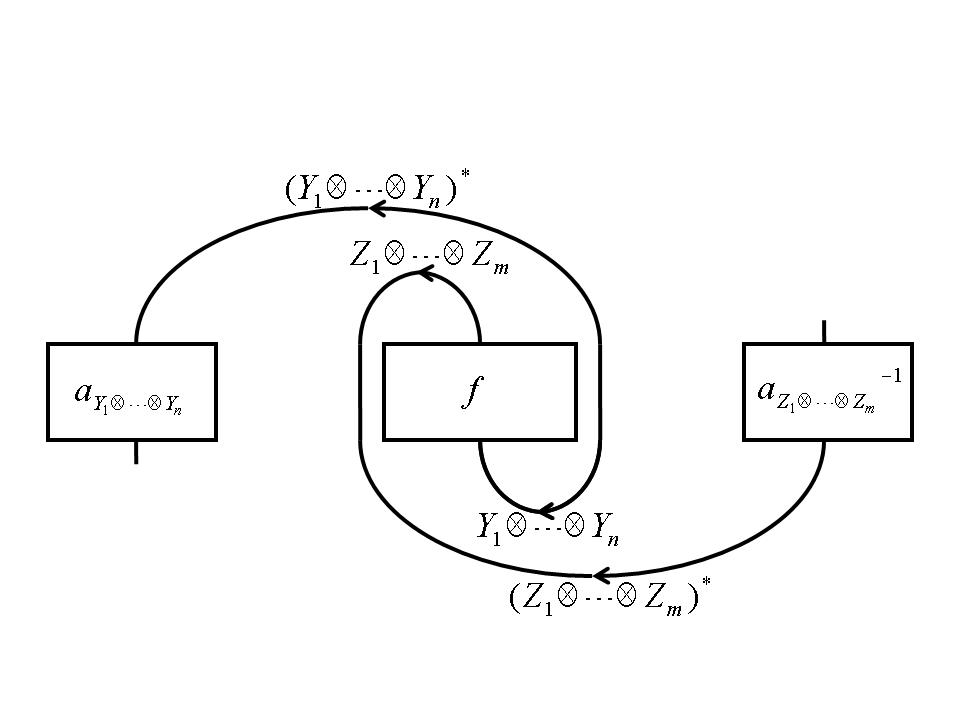}
\vskip 2cm
\noindent$=a_{Z_{1}\otimes \cdots \otimes Z_{m}}^{-1}\circ f^{\ast \ast }\circ
a_{Y_{1}\otimes \cdots \otimes Y_{n}}=f$ (using naturality of $a$).
\end{proof}
We now construct a planar algebra from a bicategory. Let $\mathcal{B}$ be
a pivotal $\mathbb{C}$-linear strict $2$-category with $\{+,-\}$ as the set of $0$-cells
and fix $X\in ob(\mathcal{B}(-,+))$. For each colour $(k,\varepsilon )$, set
\begin{equation*}
\begin{tabular}{ccc}
$X_{(k,\varepsilon )}$ & $=$ & $\left\{
\begin{tabular}{c}
$X\otimes X^{\ast }\otimes X\otimes X^{\ast }\otimes X\otimes \cdots k$ many
tensor factors if $\varepsilon =+$, \\
$X^{\ast }\otimes X\otimes X^{\ast }\otimes X\otimes X^{\ast }\otimes \cdots
k$ many tensor factors if $\varepsilon =-$,%
\end{tabular}%
\right. $%
\end{tabular}%
\end{equation*}

if $k\geq 1$ and $X_{(0,\varepsilon )}=1_{\varepsilon }\in ob(\mathcal{B}%
(\varepsilon ,\varepsilon ))$. Define $P_{(k,\varepsilon
)}=End(X_{(k,\varepsilon )})$.

Now, for a $(k,\varepsilon )$-planar tangle $\theta \in \mathcal{T(}%
(k_{1},\varepsilon _{1}),(k_{2},\varepsilon _{2}),\cdots ,(k_{n},\varepsilon
_{n});(k,\varepsilon )),$ we wish to define a multilinear map $P(\theta
):P_{(k_{1},\varepsilon _{1})}\times \cdots \times P_{(k_{n},\varepsilon
_{n})}\rightarrow P_{(k,\varepsilon )}$. For this we extensively use the
graphical calculus of the $2$-cells of $\mathcal{B}$.

For the ease of dealing with $2$-cells replaced by labelled rectangles, we
will consider the planar tangle $\theta $ as an isotopy class of pictures
where each disc (internal or external) is replaced by a rectangle with
first half of the strings being attached to one of the side (called the
\textit{top side}) and the remaining half of the strings attached to the
opposite side (called the \textit{bottom side}). Next, in the isotopy class
of $\theta $, we fix a picture $\Theta $ placed on $\mathbb{R}^{2}$ with the
bottom side of the external rectangle being parallel to the $X$-axis,
satisfying the following properties:

\begin{itemize}
\item the collection of strings in $\Theta $ must have finitely many local
maximas and minimas,

\item each internal rectangle is aligned in such a way that the top side of
the external rectangle is parallel and also nearer to the top side of the
internal rectangle than its bottom side,

\item the projections of the maxima, minima and one of the vertical sides
of each internal rectangle (that is, the sides other than the top and bottom
ones) on the vertical sides of the external rectangle of $\Theta $ are
disjoint.
\end{itemize}

We will say that an element $\Theta $ in the isotopy class of $\theta $ is
in \textit{standard form} if $\Theta $ satisfies the above conditions\textit{%
.} For example, a standard form representation of the tangle in Figure
\ref{egtang} will be the following diagram
\begin{center}
\includegraphics[scale=0.3]{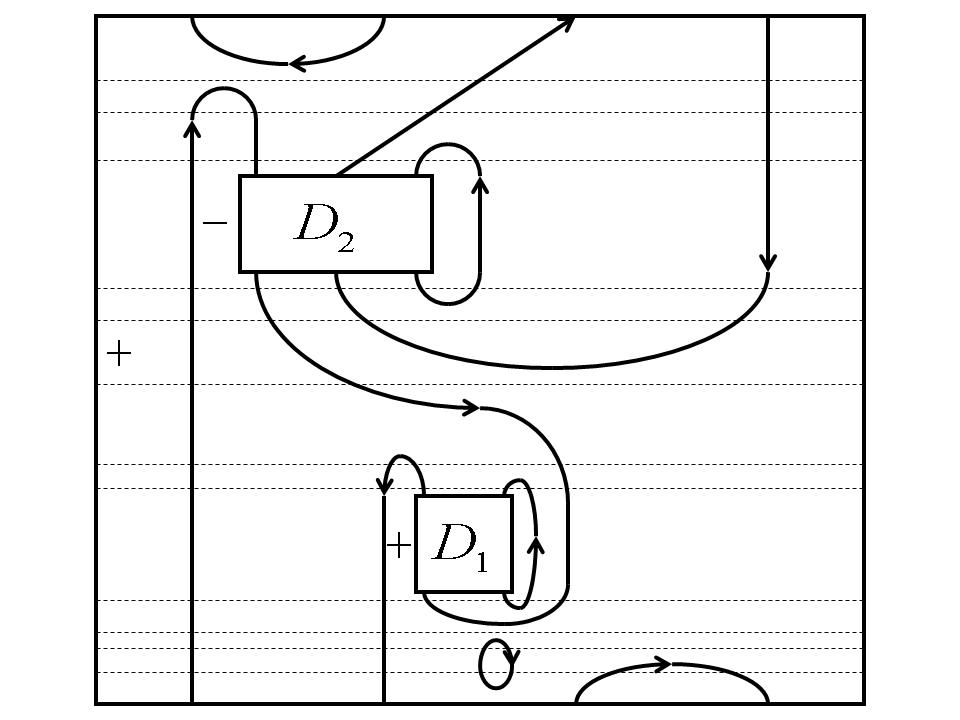}
\end{center}

Let $\Theta $ be an element in standard form of the isotopy class of $\theta
$. We now cut $\Theta $ into horizontal stripes so that every stripe should
have at most one local maxima, minima or internal rectangle. Each component
of every string in a horizontal stripe is labelled with $X$ or $X^{\ast }$
according as the orientation of the string is from the bottom side to the
top side of the horizontal stripe or reverse respectively; each local maxima
or minima is labelled with $X$ and the orientation is induced by the
orientation of the actual string in $\Theta $. For example,\\
\includegraphics[scale=0.24,bb= 0 250 750 550]
{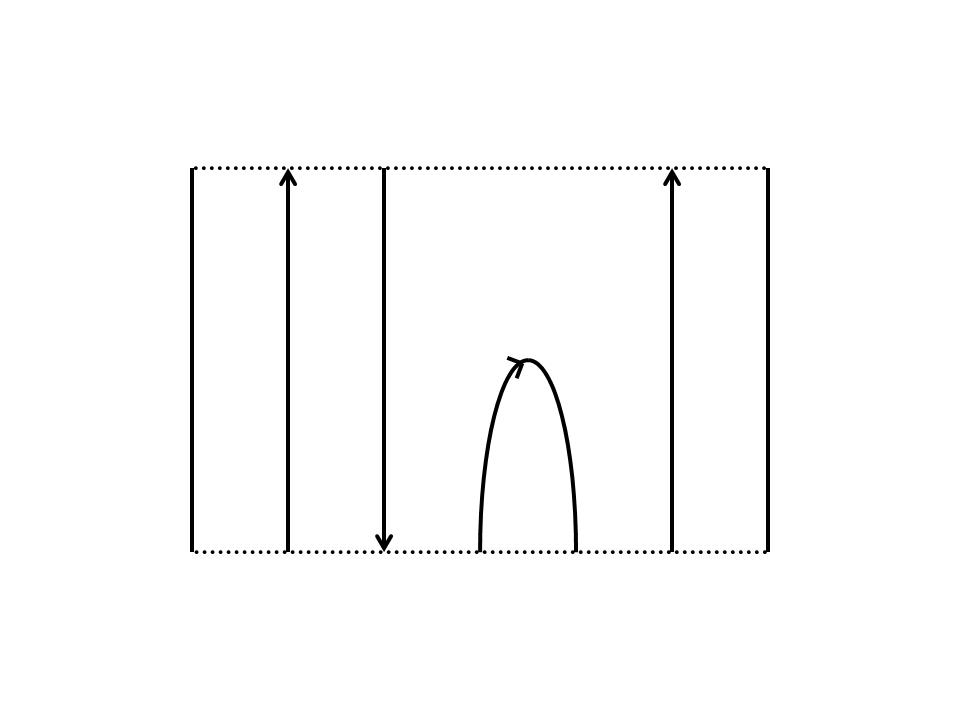}
will be replaced by
\includegraphics[scale=0.24,bb= 0 250 750 550]
{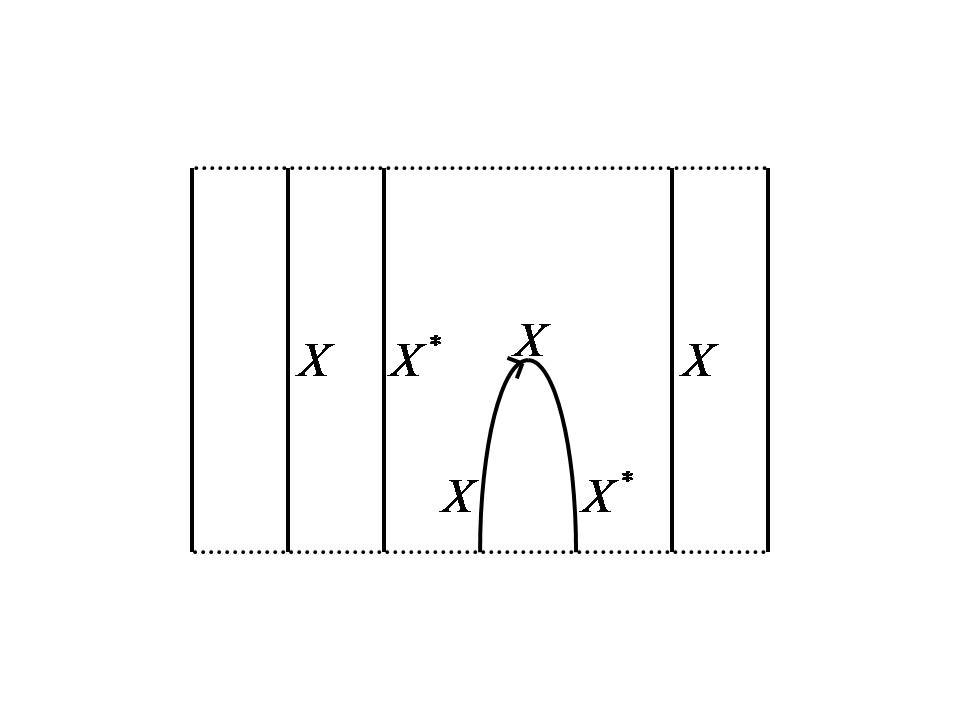}
\vskip 1.92cm
To define $P(\theta ):P_{(k_{1},\varepsilon _{1})}\times \cdots \times
P_{(k_{n},\varepsilon _{n})}\rightarrow P_{(k,\varepsilon )}$, we fix $2$%
-cells $f_{i}\in P_{(k,\varepsilon )}$ for $1\leq i\leq n$. We label the $i^{%
\text{th}}$ internal rectangle (contained in some horizontal stripe) with $%
f_{i}$. Now, each horizontal stripe makes sense as a $2$-cell according to
the notation already set up. We define $P(\theta )(f_{1},f_{2},\cdots
,f_{n}) $ as the composition of these $2$-cells. It is easy to easy that $%
P(\theta )$ is a multilinear map from $P_{(k_{1},\varepsilon _{1})}\times
\cdots \times P_{(k_{n},\varepsilon _{n})}$ to $P_{(k,\varepsilon )}$.
Natural question to ask will be why $P(\theta )(f_{1},f_{2},\cdots ,f_{n})$
is independent of the choice of $\Theta $ in the isotopy class of $\theta $.

For this, first note that one standard form representative of a tangle can be
obtained from another applying finitely many moves of the following three
types:

(i) \textit{Sliding Move:}
\includegraphics[scale=0.2,bb= 0 250 750 550]{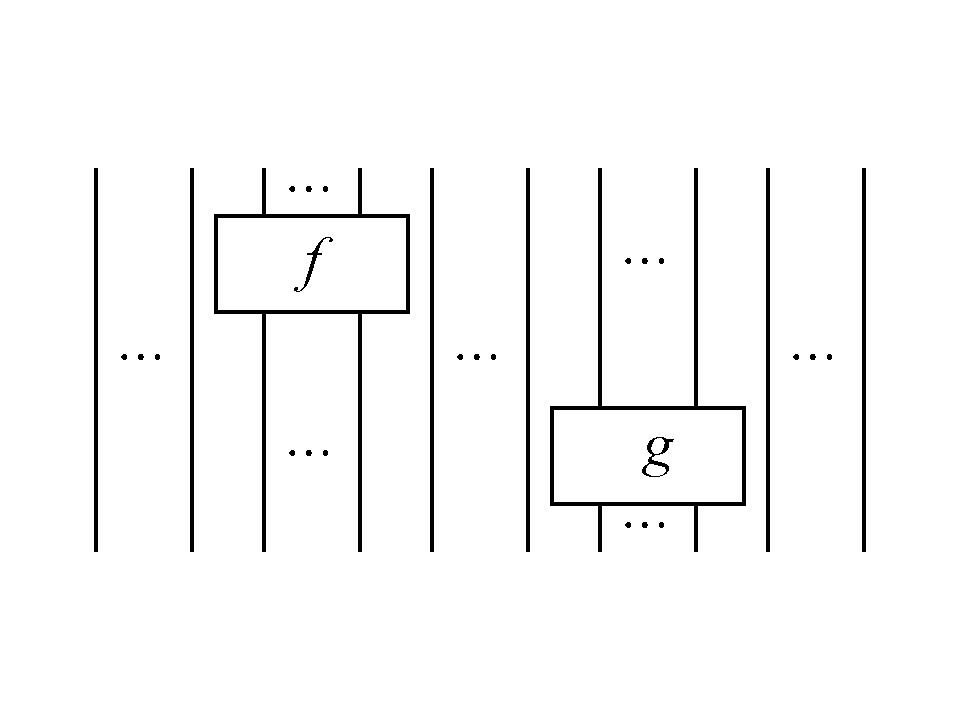}
$\sim$
\includegraphics[scale=0.2,bb= 0 250 750 550]{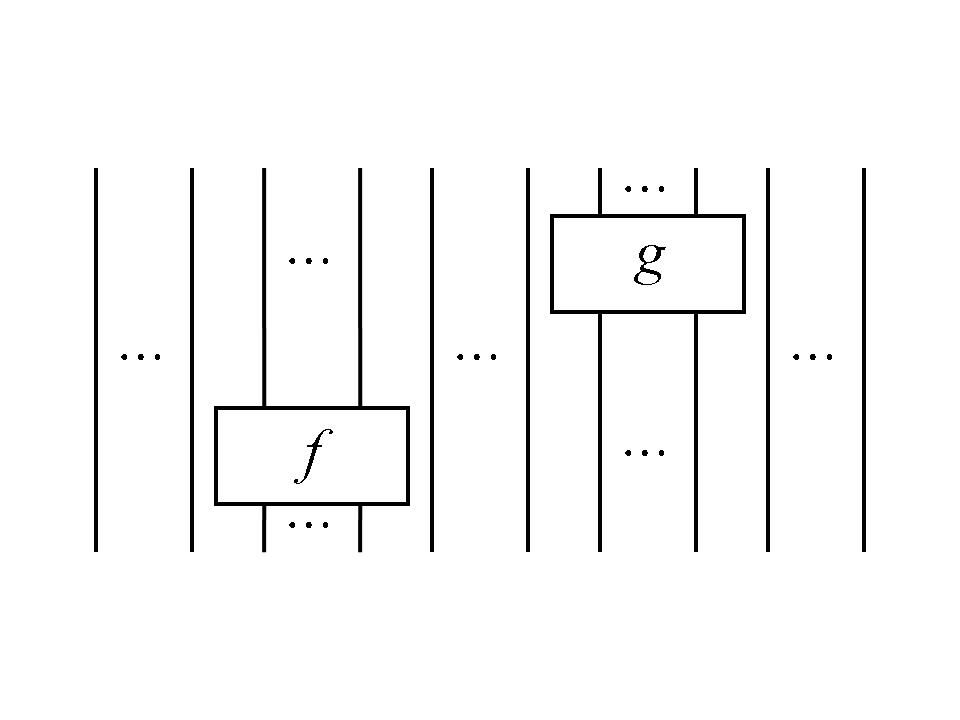}
\vskip 1.6cm
(ii) \textit{Rotation Move:}
\includegraphics[scale=0.2,bb= 0 250 750 550]{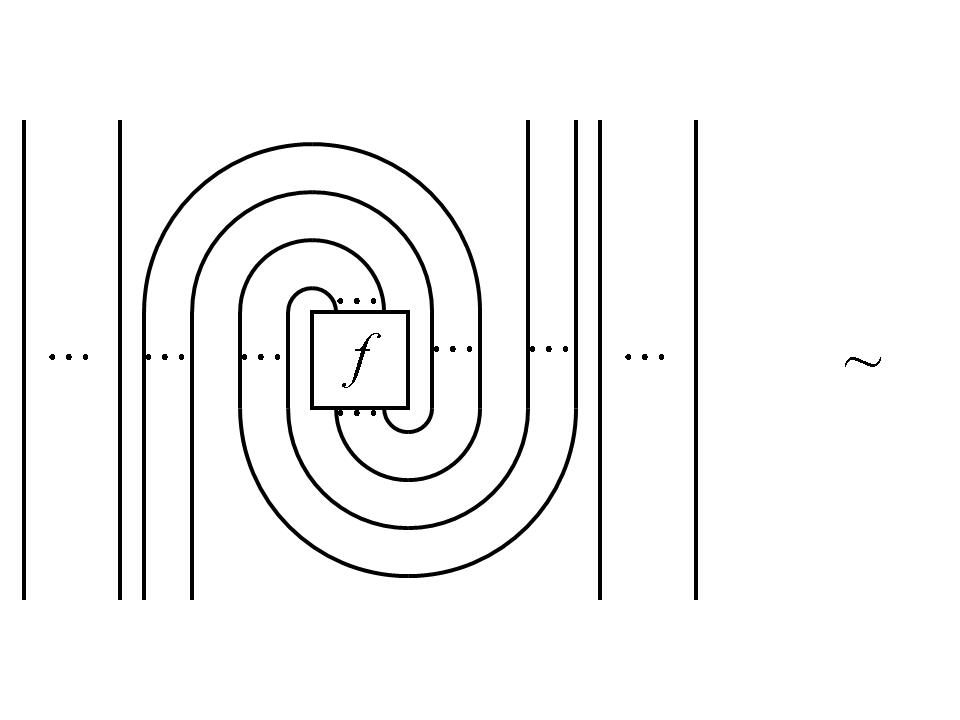}
\includegraphics[scale=0.2,bb= 0 250 750 550]{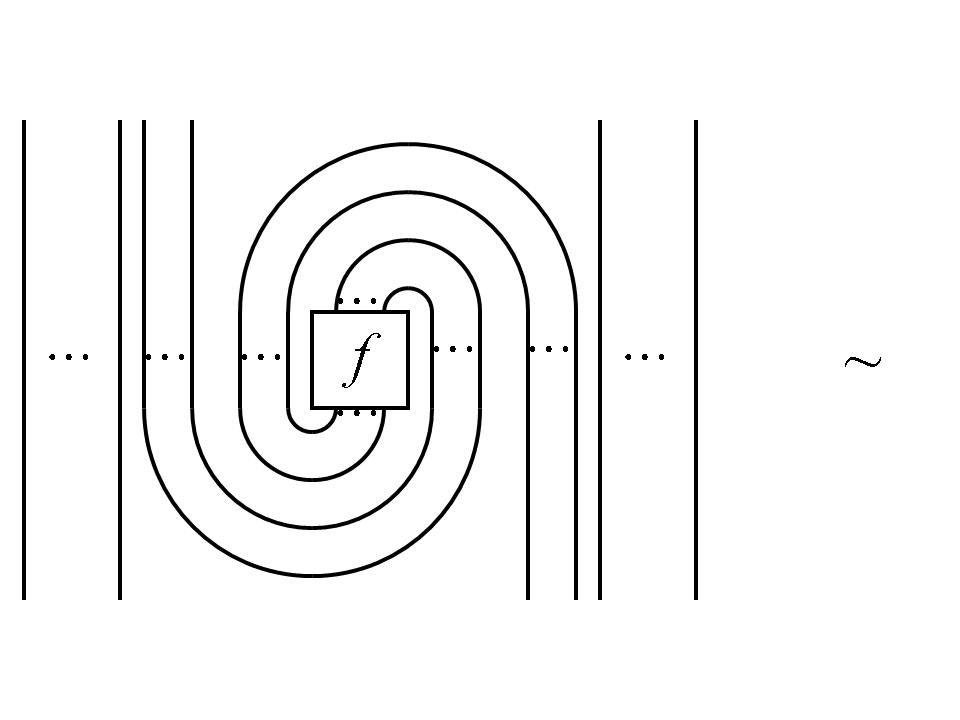}
\includegraphics[scale=0.2,bb= 0 250 750 550]{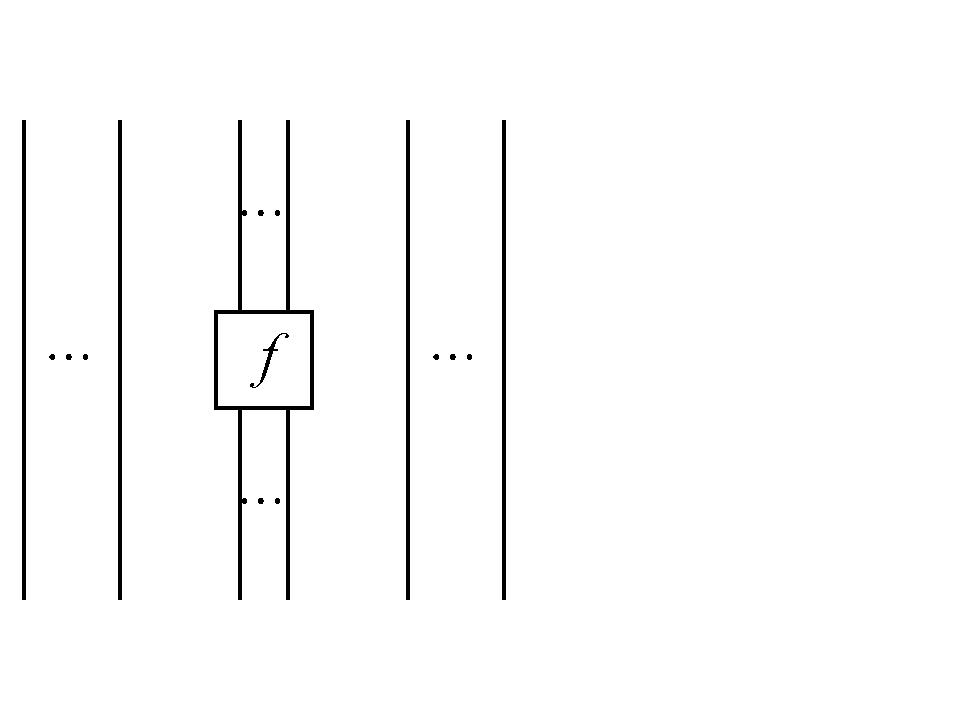}
\vskip 1.6cm
(iii) \textit{Wiggling Move:}
\includegraphics[scale=0.2,bb= 0 250 750 550]{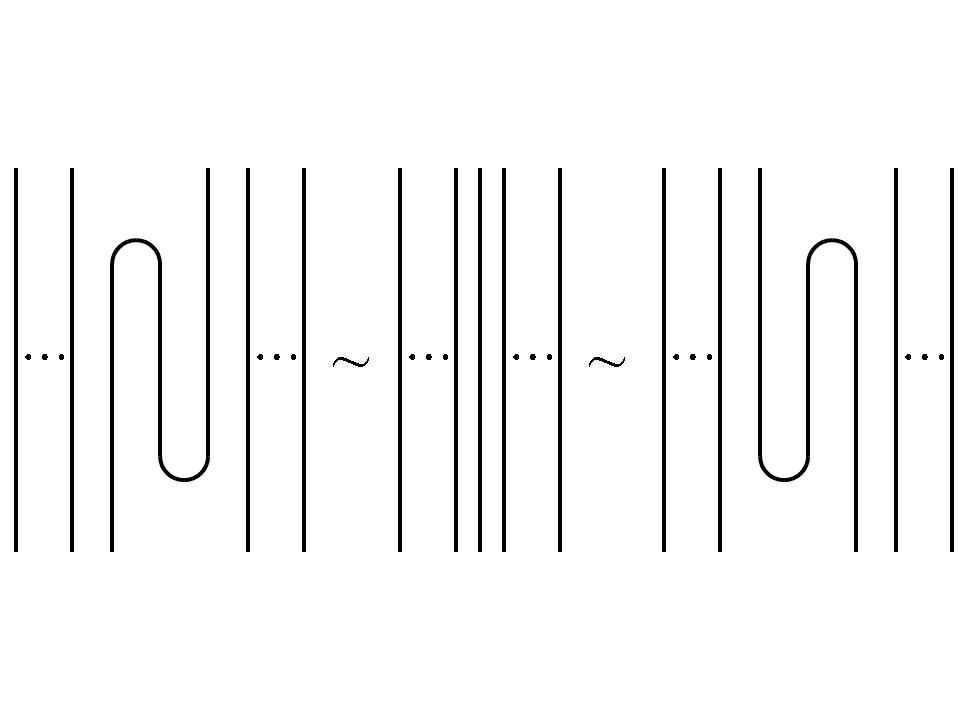}
\vskip 1.6cm
\noindent where $f$ and $g$ are $2$-cells. To show
$P(\theta )(f_{1},f_{2},\cdots ,f_{n})$ is
well-defined, it is enough to show that two standard form representatives
$\Theta_{1}$ and $\Theta _{2}$ of $\theta$ labelled with
$\{f_{1},f_{2},\cdots ,f_{n}\}$, differing by any of the above three moves,
will assign identical $2$-cell. Invariance under sliding moves hold from the
functoriality of $\otimes $ and rotation moves follow from Proposition
\ref{Rotation Invariance} and Corollary \ref{ax*}; finally, wiggling moves are
justified by Lemma \ref{Basic Lemma} (i). Thus, we have a
well-defined map $P:\mathcal{T}((k_{1},\varepsilon _{1}),\cdots
,(k_{n},\varepsilon _{n});(k,\varepsilon ))\rightarrow \mathcal{MV}%
ec((P_{(k_{1},\varepsilon _{1})},\cdots ,P_{(k_{n},\varepsilon
_{n})};P_{(k,\varepsilon )})$ (resp. $P:\mathcal{T}(\emptyset
;(k,\varepsilon ))\rightarrow \mathcal{MV}ec(\emptyset ;P_{(k,\varepsilon
)}) $). Finally, define the planar algebra $P:\mathcal{P\rightarrow MV}ec$ via:
\begin{itemize}
\item $P(k,\varepsilon )=P_{(k,\varepsilon )}$,
\item the linear map $P:\mathcal{P}((k_{1},\varepsilon _{1}),\cdots
,(k_{n},\varepsilon _{n});(k,\varepsilon ))\rightarrow \mathcal{MV}%
ec(P_{(k_{1},\varepsilon _{1})},\cdots ,P_{(k_{n},\varepsilon
_{n})};P_{(k,\varepsilon )})$ (resp. $P:\mathcal{P}(\emptyset
;(k,\varepsilon ))\rightarrow \mathcal{MV}ec(\emptyset ;P_{(k,\varepsilon
)}) $) is defined by extending the map $P:\mathcal{T}((k_{1},\varepsilon
_{1}),\cdots ,(k_{n},\varepsilon _{n});(k,\varepsilon ))\rightarrow \mathcal{%
MV}ec(P_{(k_{1},\varepsilon _{1})},\cdots ,P_{(k_{n},\varepsilon
_{n})};P_{(k,\varepsilon )})$ (resp. $P:\mathcal{T}(\emptyset
;(k,\varepsilon ))\rightarrow \mathcal{MV}ec(\emptyset ;P_{(k,\varepsilon
)}) $) linearly.
\end{itemize}
Clearly, $P(1_{(k,\varepsilon)}) = id_{P((k,\varepsilon))}$ for all
$(k,\varepsilon)$. To check $P$ preserves composition of morphisms, let us
consider two
tangles $T$ and $S$ such that the first internal disc $D_1$ of $T$ has color
$(k,\varepsilon)$ same as that of $S$. Choose $T_1$ and $S_1$ as standard form
representatives of $T$ and $S$ respectively such that dimension of the external
disc of $S_1$ along with the marked points match with that of $D_1$ in $T_1$.
Let $T_1 \underset{D_1}{\circ} S_1$ denote the picture obtained by replacing
$D_1$ by $S_1$ and then erasing the external boundary of $S_1$. Note that
$T_1 \underset{D_1}{\circ} S_1$ is a standard form representative of
$T \underset{D_1}{\circ} S$. Now, we consider $2$-cells for each internal discs
(except $D$) in $T$ and $S$ coming from the appropriate vector spaces and label
the corresponding rectangles in $T_1$, $S_1$, and
$T_1 \underset{D_1}{\circ} S_1$ with them. If we slice
$T_1 \underset{D_1}{\circ} S_1$ as described while defining the action of $P$
on the morphism spaces and induce the slicing of
$T_1 \underset{D_1}{\circ} S_1$ on $T_1$ and $S_1$, then the $2$-cells
corresponding to the slices appearing in $T_1 \underset{D_1}{\circ} S_1$ are
the same as those for $T_1$ with the slice containing $D_1$ being replaced by
the slices coming from $S_1$. Thus $P$ must preserve
composition.

This completes the construction of the planar algebra.
\section{Affine Representations of a Planar Algebra}\label{affrep}
In this section, we will introduce the notion of an \textit{affine
representation of a planar algebra }which is a generalization of the concept
of the \textit{Hilbert space representation of annular Temperley-Lieb} by
Vaughan Jones and Sarah Reznikoff (\cite{Aff TL}); one can also treat this
as an \textit{annular representations of a planar algebra} with rigid
boundaries. We then discuss some general theory of the affine representations
following exactly the way Jones developed the theory for annular
representations in \cite{ann rep}.

Before going into the definition of affine representations, we will first
introduce the \textit{affine category over a planar algebra}.
\begin{definition}
An $((m,\eta ),(n,\varepsilon ))$-affine tangle is an isotopy class of
pictures consisting of:
\begin{itemize}
\item the annulus $\mathcal{A}=\{z\in \mathbb{C}:1\leq z\leq 2\}$,
\item the set of points $\{2e^{-\frac{k\pi i}{m}}:0\leq k\leq 2m-1\}$ (resp.
$\{e^{-\frac{k\pi i}{n}}:0\leq k\leq 2n-1\}$) are numbered clockwise
starting from $2$ (resp. $1$) as the first points,
\item $\mathcal{A}$ consists of internal discs $D_{1}$, $D_{2}$,$\cdots $, $%
D_{l}$ with colour $(k_{1},\varepsilon _{1})$, $(k_{2},\varepsilon _{2})$,$%
\cdots $,$(k_{l},\varepsilon _{l})$ respectively and non-intersecting
oriented strings (just like in an ordinary planar tangle described in
Definition \ref{Def of pln tang}) so that the inner (resp. outer) boundary
of $\mathcal{A}$ gets the colour $\left( n,\varepsilon \right) $ (resp. $%
\left( m,\eta \right) $),
\item any isotopy should keep the boundary of $\mathcal{A}$ fixed.
\end{itemize}
\end{definition}
\begin{figure}[h]
\begin{center}
\includegraphics[scale=0.27]{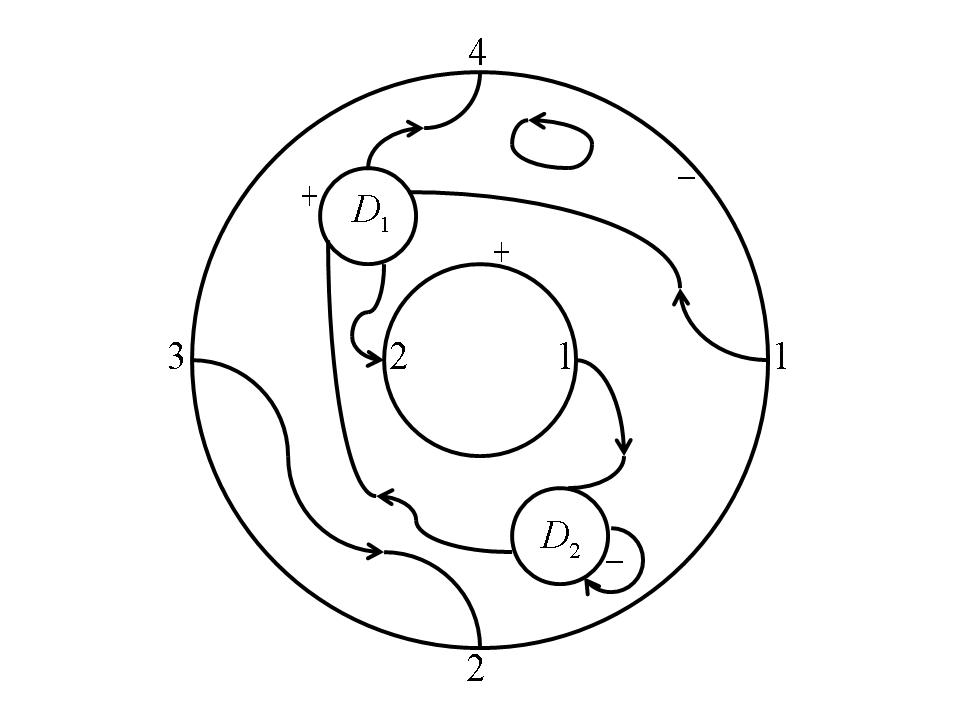}
\end{center}
\caption{Example of a $\left( \left(2,-\right) ,\left( 1,+\right)
\right)$-affine tangle}
\label{aff tang eg}
\end{figure}
Let $P$ be a planar algebra. An $((m,\eta ),(n,\varepsilon ))$-affine tangle
is said to be $P$\textit{-labelled} if, to each internal disc $D$ of $%
\mathcal{A}$ with colour $(k,\varepsilon ^{\prime })$, an element of $%
P_{(k,\varepsilon ^{\prime })}$ is assigned. Let $\mathcal{A}%
_{(n,\varepsilon )}^{(m,\eta )}$ denote the set of all $((m,\eta
),(n,\varepsilon ))$-affine tangles and $\mathcal{A}_{(n,\varepsilon
)}^{(m,\eta )}(P)$ denote the set of all $P$-labelled $((m,\eta
),(n,\varepsilon ))$-affine tangles. If $A\in \mathcal{A}_{(n,\varepsilon
)}^{(m,\eta )}$ and $B\in \mathcal{A}_{(l,\delta )}^{(n,\varepsilon )},$
then we can define $A\circ B$ ($\in \mathcal{A}_{(l,\delta )}^{(m,\eta )}$)
as the affine tangle obtained by considering the picture $\frac{1}{2}(2A\cup
B)$. We might have to smoothen out the strings which are attached with the
inner boundary of $2A$ and outer boundary of $B$; this can also be avoided
by requiring the strings to meet the inner and the outer boundaries radially
in the definition of an affine tangle.

We now set up a convenient way of sketching an affine tangle; instead of
marking the points on the inner (resp. outer) boundary at the roots of unity
(resp. twice the roots of unity), we will mark them close to each other on
the top with $1$ as the leftmost point. Further, with the help of isotopy,
every $A\in \mathcal{A}_{(n,\varepsilon )}^{(m,\eta )}$ can be expressed as:
\begin{center}
\includegraphics[scale=0.27]
{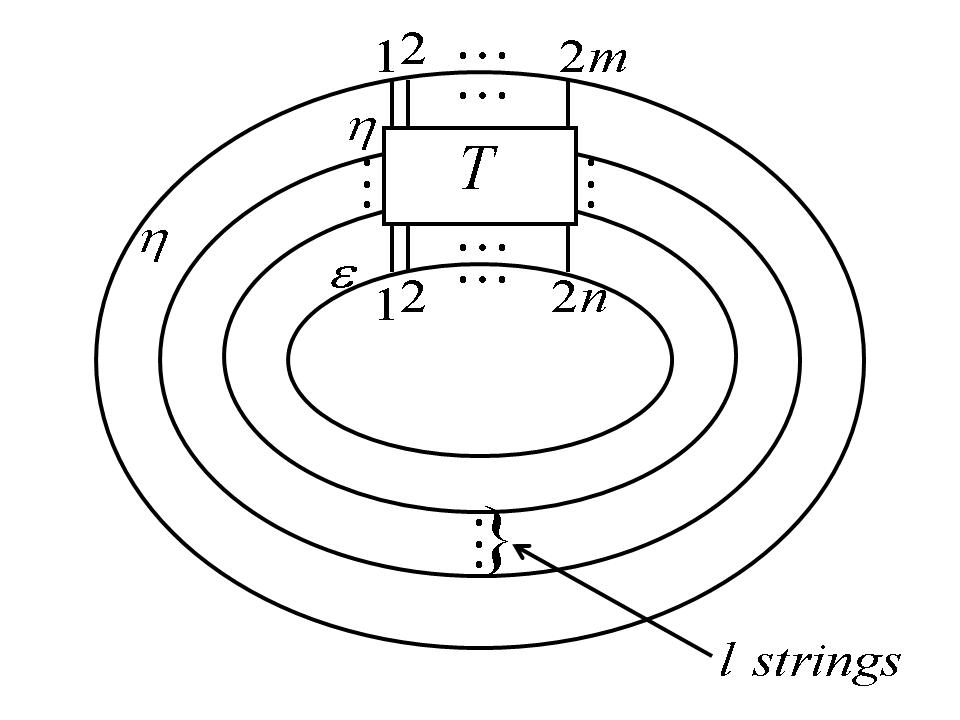}
\end{center}
for some $T\in \mathcal{
T}_{(m+n+l,\eta )}$. Note that $T$ and $l$ are not unique. For example, the
affine tangle in Figure \ref{aff tang eg} can be expressed as the above
annular tangle for $m=2$, $n=1$, $l=1$ and
$T \; =$
\includegraphics[scale=0.24,bb= 0 250 750 550]
{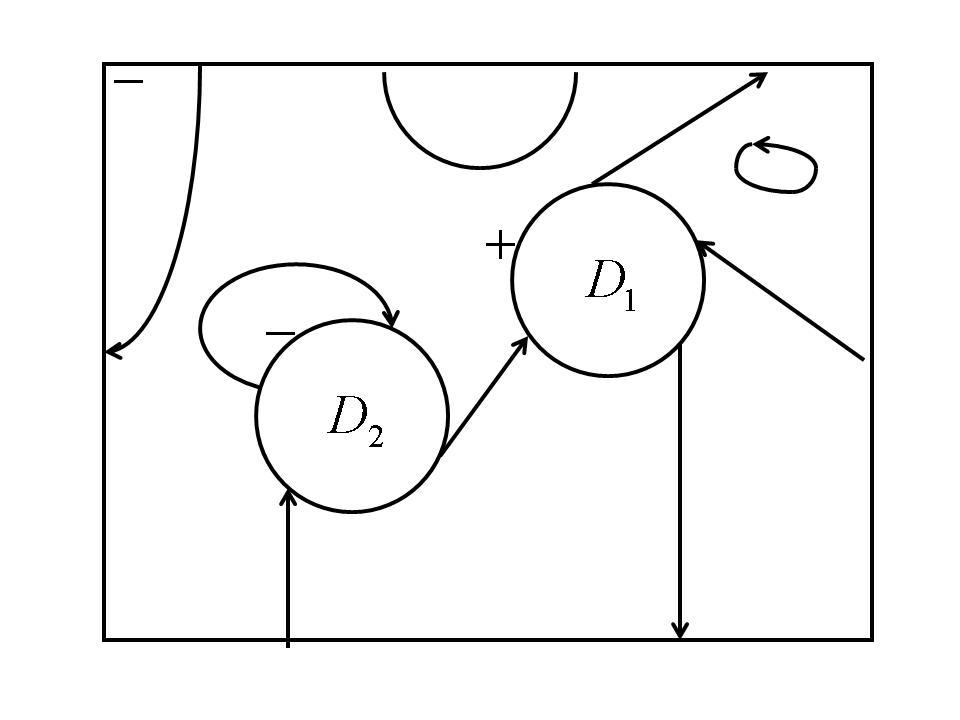}
\vskip 1.92cm
\noindent and $\eta = -$.

Let $(F\mathcal{A}P)_{(n,\varepsilon )}^{(m,\eta )}$ be the vector space
with $\mathcal{A}_{(n,\varepsilon )}^{(m,\eta )}(P)$ as a basis, $\mathcal{T}%
_{(k,\varepsilon )}(P)$ be the set of all $P$\textit{-labelled} $%
(k,\varepsilon )$-planar tangles, $\mathcal{P}_{(k,\varepsilon )}$ (resp. $%
\mathcal{P}_{(k,\varepsilon )}(P)$) be the vector space with $\mathcal{T}%
_{(k,\varepsilon )}$ (resp. $\mathcal{T}_{(k,\varepsilon )}(P)$) as a basis
and $\Psi _{(m,\eta ),(n,\varepsilon )}^{l}$ be the annular
tangle
\includegraphics[scale=0.27,bb= 0 250 700 550]
{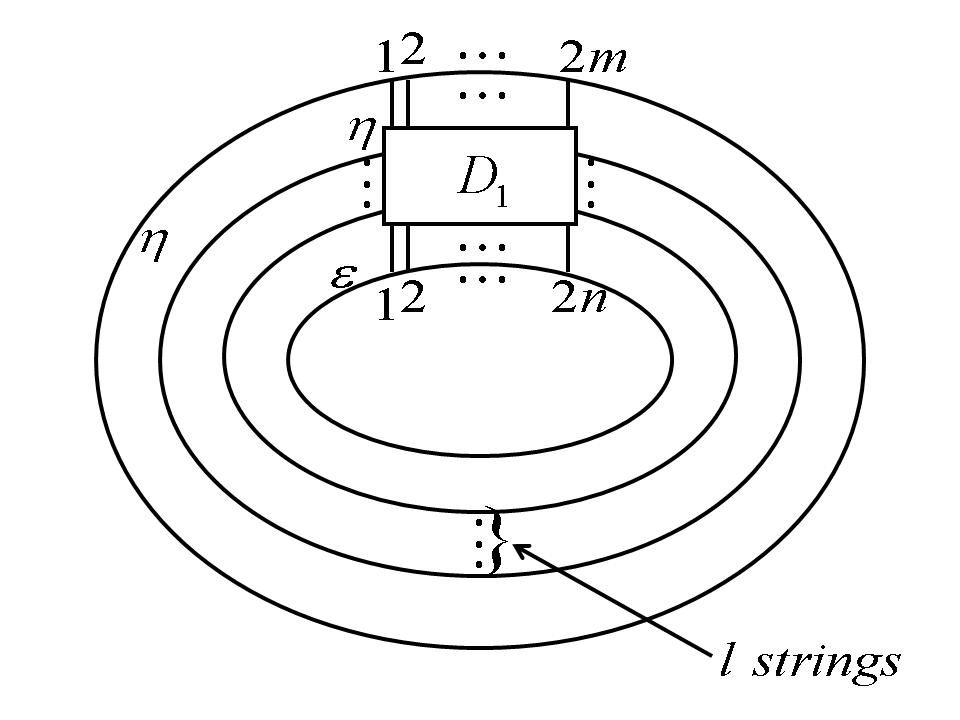}.
\vskip 2.16cm
\noindent Observe that $\Psi _{(m,\eta
),(n,\varepsilon )}^{l}$ induces a linear map $\psi _{(m,\eta
),(n,\varepsilon )}^{l}:\mathcal{P}_{(m+n+l,\eta )}\rightarrow (F\mathcal{A}%
P)_{(n,\varepsilon )}^{(m,\eta )}$ which also lifts to the level of the $P$%
-labelled ones. Moreover, for any $A\in (F\mathcal{A}P)_{(n,\varepsilon
)}^{(m,\eta )}$, there exists $l\in \mathbb{N}_{0}$ and $T\in \mathcal{P}%
_{(m+n+l,\eta )}(P)$ such that $A=\psi _{(m,\eta ),(n,\varepsilon )}^{l}(T)$.
Set
\begin{equation*}
\mathcal{W}_{(n,\varepsilon )}^{(m,\eta )}=\left\{ A\in (F\mathcal{A}
P)_{(n,\varepsilon )}^{(m,\eta )}\left\vert
\begin{array}{c}
A=\psi _{(m,\eta ),(n,\varepsilon )}^{l}(T)\text{ for some }l\in \mathbb{N}
_{0}\text{, } T \in \mathcal{P}_{(m+n+l,\eta )}(P)\\
\text{such that }P(T)=0 \in P(m+n+l,\eta)
\end{array}
\right. \right\} \text{.}
\end{equation*}
where we use the linear map $P:
\mathcal{P}_{(k,\varepsilon )}(P)\rightarrow P(k,\varepsilon )$ induced by the
map of multicategories $P$. It is a
fact that $\mathcal{W}_{(n,\varepsilon )}^{(m,\eta )}$ is a vector subspace
of $(F\mathcal{A}P)_{(n,\varepsilon )}^{(m,\eta )}$. For instance, if $%
A_{i}\in \mathcal{W}_{(n,\varepsilon )}^{(m,\eta )}$ and $l_{i}\in \mathbb{N}%
_{0}$, $T_{i}\in \mathcal{P}_{(m+n+l_{i},\eta )}(P)$ such that $A_{i}=\psi
_{(m,\eta ),(n,\varepsilon )}^{l_{i}}(T_{i})$, $P(T_{i})=0$ for $i=1$, $2$
and $l_{1}\leq l_{2}$, then one can obtain $\widetilde{T_{1}}\in \mathcal{P}%
_{(m+n+l_{2},\eta )}(P)$ such that $A_{1}=\psi _{(m,\eta ),(n,\varepsilon
)}^{l_{2}}(\widetilde{T_{1}})$ by wiggling back and forth a string emanating
from either of the vertical sides of $T_{1}$ around the inner disc of $A_{1}$
until the total number of strings around the inner disc of $A_{1}$ increases
from $l_{1}$ to $l_{2}$; finally, $A_{1}+A_{2}=\psi _{(m,\eta
),(n,\varepsilon )}^{l_{2}}(\widetilde{T_{1}}+T_{2})$.

Define the category $\mathcal{A}ffP$ by:
\begin{itemize}
\item $ob(\mathcal{A}ffP)=\{(k,\varepsilon ):k\in \mathbb{N}_{0},\varepsilon
\in \{+,-\}\}$
\item $Hom_{(\mathcal{A}ffP)}((n,\varepsilon ),(m,\eta ))=\frac{(F\mathcal{A}%
P)_{(n,\varepsilon )}^{(m,\eta )}}{\mathcal{W}_{(n,\varepsilon )}^{(m,\eta )}%
}=$ the quotient vector space of $(F\mathcal{A}P)_{(n,\varepsilon
)}^{(m,\eta )}$ over $\mathcal{W}_{(n,\varepsilon )}^{(m,\eta )}$ (also
denoted by $\left( \mathcal{A}ffP\right) _{(n,\varepsilon )}^{(m,\eta )}$),
\item the composition of affine tangles is linearly extended for $(F\mathcal{%
A}P)$'s; one can easily verify that $A\circ B$ $\in \mathcal{W}_{(l,\delta
)}^{(m,\eta )}$ whenever $A\in \mathcal{W}_{(n,\varepsilon )}^{(m,\eta )}$
and $B\in (F\mathcal{A}P)_{(l,\delta )}^{(n,\varepsilon )}$, or $A\in (F%
\mathcal{A}P)_{(l,\delta )}^{(n,\varepsilon )}$ and $B\in \mathcal{W}%
_{(n,\varepsilon )}^{(m,\eta )}$; this implies the composition is induced in
the level of quotient vector spaces as well,
\item the identity of $(k,\varepsilon )$ denoted by $1_{(k,\varepsilon )}$,
is given by an $((k,\varepsilon ),(k,\varepsilon ))$-affine tangle obtained
by joining the $i^{\text{th}}$ point of the inner boundary with the $i^{%
\text{th}}$ point of the outer boundary by a straight string for all $i$.
\end{itemize}
We will refer the category $\mathcal{A}ffP$ as \textit{affine category over }
$P$.
\begin{definition}
An additive functor $F:\mathcal{A}ffP\rightarrow \mathcal{V}ec$ is said to
be an affine representation of $P$.
\end{definition}
\begin{remark}
The functor induced by $P$ itself gives an affine representation of $P$;
this is called the `trivial' affine representation.
\end{remark}
\begin{lemma}
\label{aff rep basic lemma}If $F$ is an affine representation of $P,$ then

(a) $F(k,\varepsilon )\hookrightarrow F(k+1,\varepsilon )$,

(b) $F(k,\varepsilon )$ is isomorphic to $F(k,-\varepsilon )$\newline
for all colours $(k,\varepsilon )$.
\end{lemma}
\begin{proof}
The inclusion in part (a) is given by considering the $F$-image of the
inclusion tangle.

For part (b), consider the rotation tangle $R_{(k,\varepsilon )}\in \mathcal{%
A}_{(k,\varepsilon )}^{(k,-\varepsilon )}$ obtained by joining the points $%
e^{-\frac{l\pi i}{k}}$ and $2e^{-\frac{(l+1)\pi i}{k}}$ on the boundary of $%
\mathcal{A}$ by a string which does not make a full round about the inner
disc. $F(R_{(k,\varepsilon )})$ gives the desired isomorphism in (b)
\end{proof}
\begin{remark}
It may seem so that $\left( R_{(k,\varepsilon )}\right) ^{2k}$ is the
identity $((k,\varepsilon ),(k,\varepsilon ))$-affine tangle (that is, the
tangle obtained by joining the points $e^{-\frac{l\pi i}{k}}$ and $2e^{-%
\frac{l\pi i}{k}}$ by a straight line), but this is not true because of the
restriction of the isotopy being identity on boundary of $\mathcal{A}$. This
is the main difference between the annular representations of $P$ (in \cite%
{ann rep} and \cite{Ghosh}) and the affine representations.
\end{remark}
The \textit{weight} of an affine representation $F$ denoted by $wt(F)$, is
given by the smallest integer $k$ such that 
$\dim (F(k,\varepsilon ))\neq \{0\}$.
The $wt(F)$ is well-defined by Lemma \ref{aff rep basic lemma}.

An affine representation $F$ will be called \textit{locally finite} if $%
F(k,\varepsilon )$ is finite dimensional for all colours $(k,\varepsilon )$.
The \textit{dimension} of an affine representation $F$ is defined as a pair
of formal power series $(\Phi _{F}^{+},\Phi _{F}^{-})$ where%
\begin{equation*}
\Phi _{F}^{\varepsilon }(z)=\sum_{k=0}^{\infty }\dim (F(k,\varepsilon
))~z^{k}~~~~\text{ for }\varepsilon \in \{+,-\}\text{.}
\end{equation*}

\noindent \textbf{Question:}\textit{\ If }$P$\textit{\ is a planar
algebra with modulus }$(\delta ,\delta )$\textit{, is the radius of
convergence of the dimension of an affine representation greater than or
equal to }$\delta ^{-2}$\textit{?}

The above question appeared in \cite{ann rep} for annular representations of
a planar algebra. The question for annular
representations was answered in affirmative for the Temperley-Lieb planar
algebras by Jones (in \cite{ann rep}) and for the Group
Planar Algebras by Ghosh (in \cite{Ghosh}). We will show the same for affine 
representations of any finite depth planar algebra in the Section \ref{fdpa}.

Let $P$ be a $\ast $- or a $C^{\ast }$-planar algebra. Then,
$\mathcal{P}_{(k,\varepsilon )}(P)$ becomes a $\ast $-algebra where $\ast $ of 
a labelled tangle is given by $\ast $ of the unlabelled tangle whose internal 
discs are labelled with $\ast $ of the labels. One can define $\ast $ of an 
affine tangle by relecting it around a
circle concentric to inner or outer boundary and then isotopically stretch
or shirnk to fit into the annulus $\mathcal{A}$ such that the first point of
inner or outer boundary after reflection remains the same whereas the first
point of any internal disc after reflection is given by the reflection of
the last point and colours of all discs are preserved; this can be induced
in the $P$-labelled ones by labelling the internals discs of the reflected
tangle with $\ast $ of the labels. Note that $\ast $ is an involution.
Extending $\ast $ conjugate linearly, we can define the map $\ast :\left( F%
\mathcal{A}P\right) _{(n,\varepsilon )}^{(m,\eta )}\rightarrow \left( F%
\mathcal{A}P\right) _{(m,\eta )}^{(n,\varepsilon )}$ for all colours $%
(m,\eta )$, $(n,\varepsilon )$. It is easy to check that $\ast \left(
\mathcal{W}_{(n,\varepsilon )}^{(m,\eta )}\right) =\mathcal{W}_{(m,\eta
)}^{(n,\varepsilon )}$. This makes the category $\mathcal{A}ffP$ a $\ast $%
-category. An additive functor $F:\mathcal{A}ffP\rightarrow \mathcal{H}il$
is said to be an \textit{affine }$\ast $\textit{-representation if }$F$ is $%
\ast $ preserving, that is, $F\left( A^{\ast }\right) =\left( F(A)\right)
^{\ast }$ for all $A\in Mor_{\mathcal{A}ffP}$ where $\mathcal{H}il$ denotes
the category of Hilbert spaces.
\begin{remark}
Note that if $F$ is an affine $\ast $-representation, then
$\langle F(A)(v),w\rangle =\langle v,F(A^{\ast })(w)\rangle$
for all $A\in \left( \mathcal{A}ffP\right) _{(n,\varepsilon )}^{(m,\eta )}$,
$v\in F(n,\varepsilon )$, $w\in F(m,\eta )$.
\end{remark}
The category of affine representations of a planar algebra $P$ with natural
transformations as morphism space, forms an abelian category and the
dimension is additive with respect to direct sum. One can further talk about
\textit{irreducibilty} and \textit{indecomposability} of an affine
representation (see \cite{ann rep} for details). For example, the trivial 
affine representation of $P$ is
irreducible. However, if we restrict ourselves to the case of a locally
finite, non-degenerate $C^{\ast }$-planar algebra $P$ and the category of
locally finite affine $\ast $-representations, the notions of irreducibility
and indecomposability coincide. In this case, one can also talk about
\textit{orthogonality} of affine representations. These treatments for
annular representations can be found in more details in \cite{ann rep}.

Jones indicated a procedure of finding annular representations of a
locally finite $C^{\ast }$-planar algebra $P$ with modulus $(\delta ,\delta
) $ in \cite{ann rep}; the same works for the affine ones as well. For this,
we need to consider a subspace of the morphism space $\left( \mathcal{A}%
ffP\right) _{(k,\varepsilon )}^{(k,\varepsilon )}$, namely,%
\begin{equation*}
\left( \widehat{\mathcal{A}ffP}\right) _{(k,\varepsilon )}=\left\{ A\in
\left( \mathcal{A}ffP\right) _{(k,\varepsilon )}^{(k,\varepsilon
)}\left\vert
\begin{tabular}{c}
$A\text{ is a linear combination of}$ \\
$\text{elements of the form }B\circ C\text{ where}$ \\
$B\in \left( \mathcal{A}ffP\right) _{(n,\eta )}^{(k,\varepsilon )}\text{, }%
C\in \left( \mathcal{A}ffP\right) _{(k,\varepsilon )}^{(n,\eta )}$ \\
$\text{for some colour }(n,\eta )\text{ such that }n<k$%
\end{tabular}%
\right. \right\} \text{.}
\end{equation*}%
It is easy to see that $\left( \widehat{\mathcal{A}ffP}\right)
_{(k,\varepsilon )}$ is an ideal in $\left( \mathcal{A}ffP\right)
_{(k,\varepsilon )}^{(k,\varepsilon )}$. We list some common properties
shared by affine $\ast $-representations and annular $\ast $-representations
of $P$; the proofs can be found in \cite{ann rep}.

(i) An affine representation $F$ is irreducible iff $F(k,\varepsilon )$ is
irreducible as an $\left( \mathcal{A}ffP\right) _{(k,\varepsilon
)}^{(k,\varepsilon )}$-module for all colours $\left( k,\varepsilon \right) $%
.

(ii) If $W$ is an irreducible $\left( \mathcal{A}ffP\right) _{(k,\varepsilon
)}^{(k,\varepsilon )}$-submodule of $F(k,\varepsilon )$ for some colour $%
\left( k,\varepsilon \right) $, then $W$ generates an irreducible
subrepresentation of $F$.

(iii) Orthogonal $\left( \mathcal{A}ffP\right) _{(k,\varepsilon
)}^{(k,\varepsilon )}$-submodules of $F(k,\varepsilon )$ for some colour $%
\left( k,\varepsilon \right) $, generate orthogonal subrepresentations of $F$%
.

(iv) If $F$ and $G$ are representations with $F$ being irreducible and if $%
\theta :F(k,\varepsilon )\rightarrow G(k,\varepsilon )$ is a non-zero $%
\left( \mathcal{A}ffP\right) _{(k,\varepsilon )}^{(k,\varepsilon )}$-linear
homomorphism for some colour $\left( k,\varepsilon \right) $, then $\theta $
extends to an injective homomorphism from $\ F$ to $G$, that is, an
injective natural transformation from $F$ to $G$.

(v) If $W_{\left( k,\varepsilon \right) }=$ $span\{F(A)\left( v\right) :A\in
\left( \mathcal{A}ffP\right) _{(n,\eta )}^{(k,\varepsilon )},v\in F(n,\eta
),k>n\geq 0\} \subset F{\left( k,\varepsilon \right) }$, then
\begin{equation*}
\left( W_{\left( k,\varepsilon \right) }\right) ^{\perp }=\bigcap_{A\in
\left( \widehat{\mathcal{A}ffP}\right) _{(k,\varepsilon )}}kernel(F(A))\text{%
.}
\end{equation*}

From (v), we can conclude that for an affine $\ast $-representation $F$ with
weight $k$, we have

\begin{equation*}
F(k,\varepsilon )=\bigcap_{A\in \left( \widehat{\mathcal{A}ffP}\right)
_{(k,\varepsilon )}}kernel(F(A))
\end{equation*}%
since $W_{\left( k,\varepsilon \right) }$ turns out to be zero and hence $%
F(k,\varepsilon )$ forms a module over the quotient $\frac{\left( \mathcal{A}%
ffP\right) _{(k,\varepsilon )}^{(k,\varepsilon )}}{\left( \widehat{\mathcal{A%
}ffP}\right) _{(k,\varepsilon )}}$. We denote this quotient algebra by $%
\left( LWP\right) _{(k,\varepsilon )}$ (Lowest Weight algebra at $%
(k,\varepsilon )$).

By (i), if F is an irreducible affine $\ast $-representation with weight $k$%
, then $F(k,\varepsilon )$ is an irreducible module over $%
(LWP)_{(k,\varepsilon )}$. In order to find the irreducible affine $\ast $%
-representations of $P$, it suffices to do the following:

(i) find the irreducible representations of $(LWP)_{(k,\varepsilon )}$,

(ii) find which irreducible representation of $(LWP)_{(k,\varepsilon )}$
gives rise to an irreducible affine $\ast $-representation of the planar
algebra.

We will use this method to deduce some results on the irreducible affine $%
\ast $-representations of a \textit{finite depth planar algebra }in the next
section.
\section{Finite Depth Planar Algebras}\label{fdpa}
In this section, we will recall the notion of the \textit{depth} of a planar
algebra which is motivated from the \textit{depth of a finite index subfactor}.
We then prove some finiteness results for the category of affine representation
of subfactor-planar algebras. Finally, we answer the question mentioned in
Section \ref{affrep} for subfactor-planar algebras with finite depth.

Let $P$ be a planar algebra with modulus $(\delta _{+},\delta _{-})$. We
first define below a tangle called \textit{Jones projections}.\\
$E_{(k,\varepsilon )}=P\left(
\includegraphics[scale=0.16,bb= 0 250 700 550]{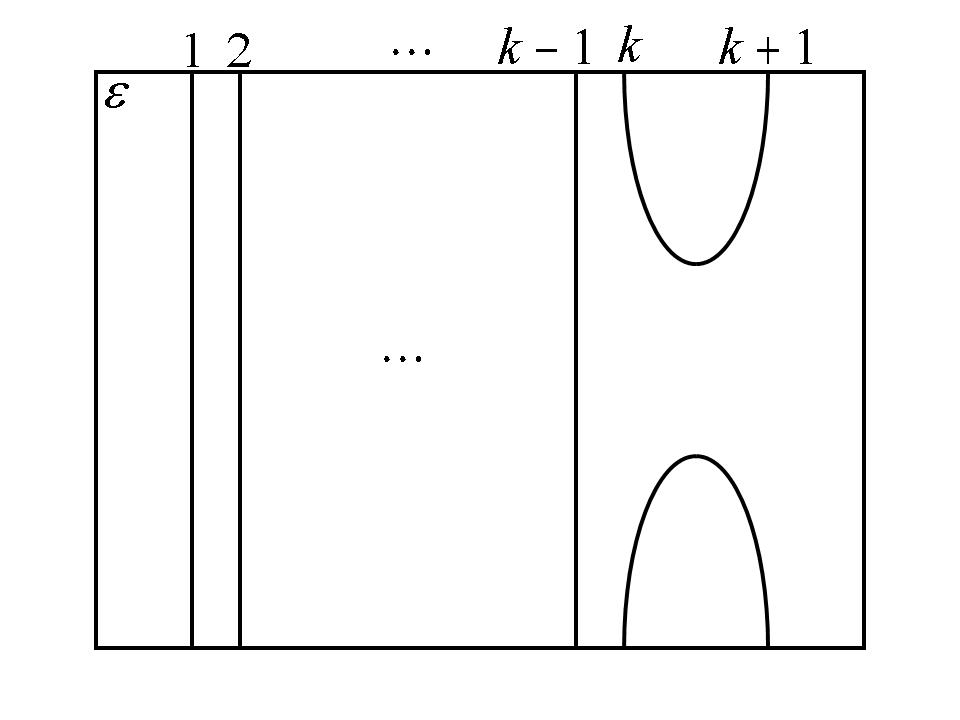}
\right) \in P_{(k+1,\varepsilon )}$
where $k\in \mathbb{N}$ and $\varepsilon \in \left\{ +,-\right\} $.\\
Note that
$E_{(k,\varepsilon )}^{2}=\left\{
\begin{array}{ll}
\delta _{\varepsilon} E_{(k,\varepsilon )} & \text{if } k \text{ is odd}\\
\delta _{-\varepsilon} E_{(k,\varepsilon )} & \text{if } k \text{ is even}
\end{array}
\right.$

From now on, we will work with the case $\delta _{+}=\delta _{-}=\delta $.
In this case, $e_{(k,\varepsilon )}=\frac{1}{\delta }~E_{(k,\varepsilon )}$
becomes an idempotent. Two more immediate consequences are:

(i) $E_{(k,\varepsilon )}\cdot E_{(k\pm 1,\varepsilon )}\cdot
E_{(k,\varepsilon )}=E_{(k,\varepsilon )}$,

(ii) $E_{(k,\varepsilon )}\cdot E_{(l,\varepsilon )}=E_{(l,\varepsilon
)}\cdot E_{(k,\varepsilon )}$ whenever $\left\vert k-l\right\vert \geq 2$

\noindent where $\cdot $ denotes the multiplication in the planar algebra $P$.
\begin{lemma}
\label{ideal}The subspace $I_{(k,\varepsilon )}=P_{(k,\varepsilon
)}e_{(k,\varepsilon )}P_{(k,\varepsilon )}=span\left\{ x\cdot
e_{(k,\varepsilon )}\cdot y:x,y\in P_{(k,\varepsilon )}\right\} $ is a
two-sided ideal of $P_{(k+1,\varepsilon )}$.
\end{lemma}
\begin{proof}
The proof of being right ideal easily follows by considering the tangle
\includegraphics[scale=0.16,bb= 0 250 700 550]{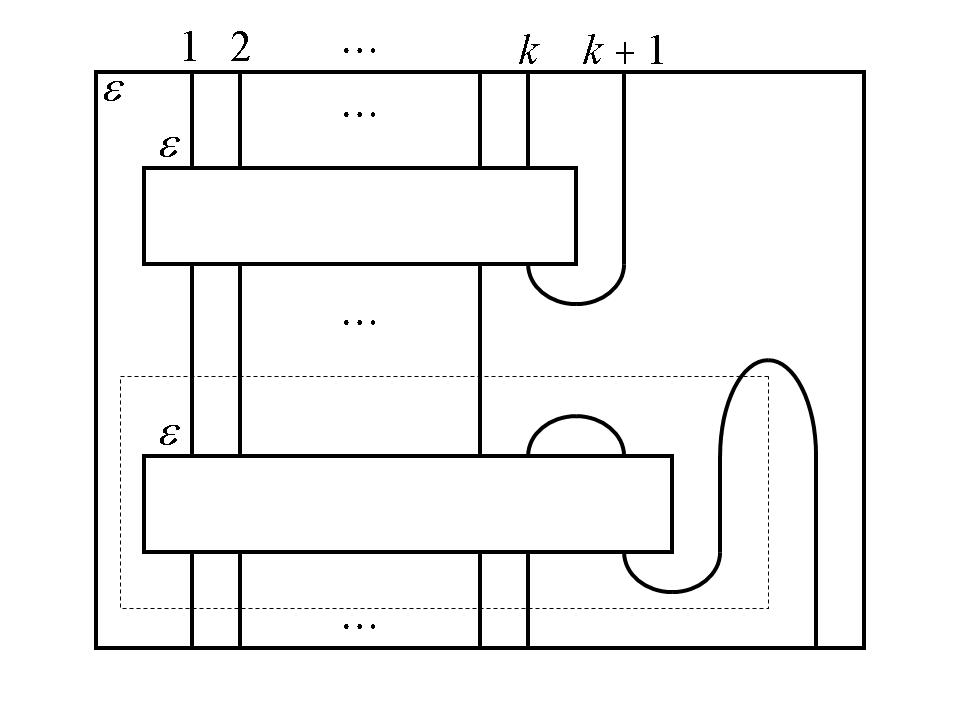}
\vskip 1.28cm
\noindent and noting that the range of the $P$ action of this tangles is inside
$I_{(k,\varepsilon )}$. Proof of left ideal follows
from the same tangle with upside down.
\end{proof}
\begin{lemma}
If $I_{(k,\varepsilon )}=P_{(k+1,\varepsilon )}$, then $I_{(k+1,\varepsilon
)}=P_{(k+2,\varepsilon )}$.
\end{lemma}
\begin{proof}
By Lemma \ref{ideal}, $I_{(k+1,\varepsilon )}$ is an ideal in $%
P_{(k+2,\varepsilon )}$. So, it is enough to show $1\in I_{(k+1,\varepsilon
)}$. Now, implies $1\in P_{(k+1,\varepsilon )}=I_{(k,\varepsilon
)}=P_{(k,\varepsilon )}E_{(k,\varepsilon )}P_{(k,\varepsilon
)}=P_{(k,\varepsilon )}E_{(k,\varepsilon )}E_{(k+1,\varepsilon
)}E_{(k,\varepsilon )}P_{(k,\varepsilon )}\subset P_{(k+1,\varepsilon
)}E_{(k,\varepsilon )}P_{(k+1,\varepsilon )}=I_{(k+1,\varepsilon )}$.
\end{proof}
\begin{lemma}\label{S lemma}
If $I_{(k,\varepsilon )}=P_{(k+1,\varepsilon )}$, then
$\left\{\begin{array}{ll}
I_{(k,-\varepsilon )}=P_{(k+1,-\varepsilon )} & \text{if } k
\text{ is even}\\
I_{(k+1,-\varepsilon )}=P_{(k+2,-\varepsilon )} & \text{if } k
\text{ is odd.}
\end{array}\right.$
\end{lemma}
\begin{proof}
Consider the tangles $S_{(k,\varepsilon )} =$
\includegraphics[scale=0.16,bb= 0 250 700 550]{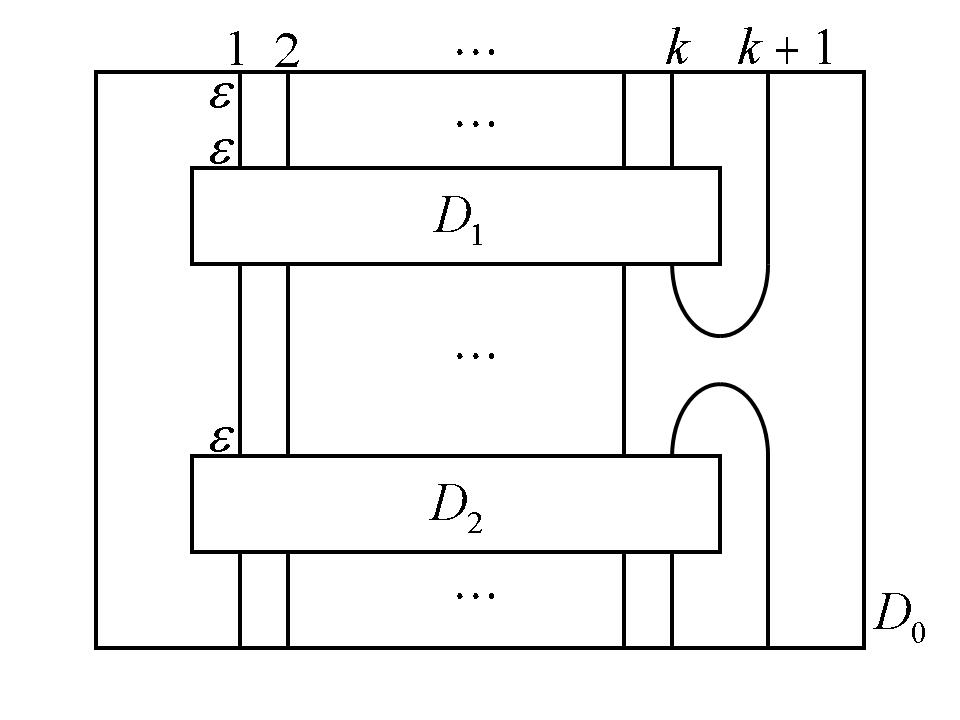}
$\in \mathcal{T} \left( (k,\varepsilon) , (k,\varepsilon) ; (k+1,\varepsilon)
\right)$
\vskip 1.2cm
\noindent and $R_k =$
\includegraphics[scale=0.16,bb= 0 250 700 550]{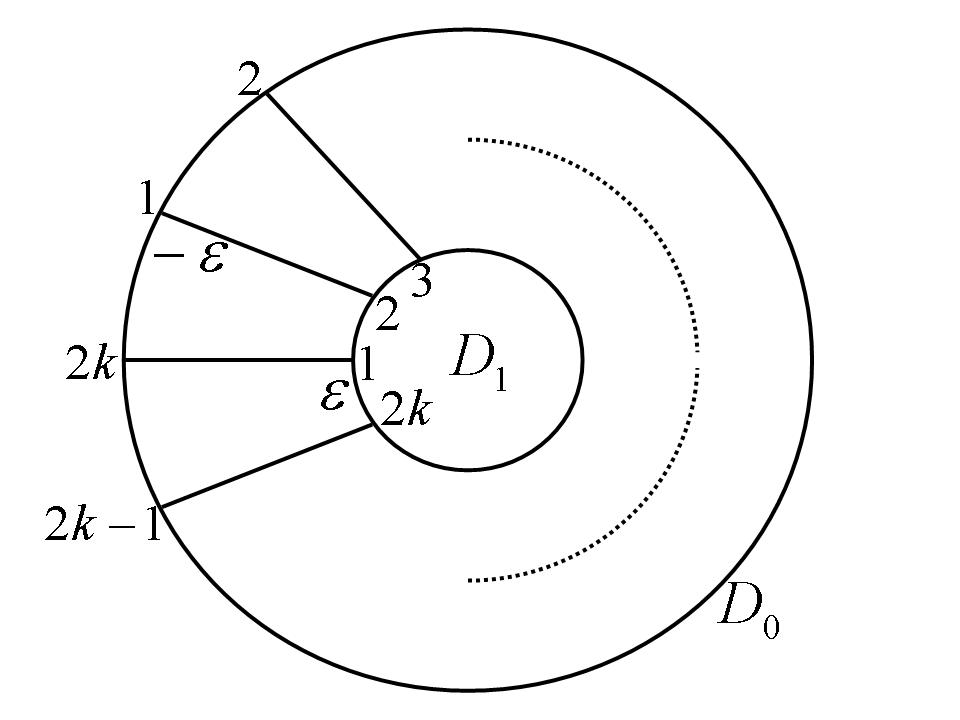}
$\in \mathcal{T} \left( (k,\varepsilon) ; (k, -\varepsilon) \right)$
\vskip 1.28cm
\noindent where $R_k$ of course depends on the $\varepsilon$ which will be
automatically determined from the context.
Note that $I_{(k,\varepsilon )}=P_{(k+1,\varepsilon )}$ if and only
if $P_{(k+1,\varepsilon )}=Range\left( P\left( S_{(k,\varepsilon )}\right)
\right) $ ($=$ span of the image of $P\left( S_{(k,\varepsilon )}\right) $).

We first consider the case $k$ being even. Then, isotopically $\left(
R_{k+1}\right) ^{k+1}\circ S_{(k,\varepsilon )}\circ \left(
\left( R_{k}\right) ^{k+1},\left( R_{k}\right)
^{k-1}\right) =S_{(k,-\varepsilon )}$. Hitting both sides with $P$ and using
the invertibility of the rotation tangles, we get the desired equality.

If $k$ is odd, then by Lemma \ref{ideal}, $I_{(k+1,\varepsilon
)}=P_{(k+2,\varepsilon )}$ where $k+1$ is even and thus we are through by
the first case.
\end{proof}
\begin{remark}
Note that in the above three lemmas, the modulus of $P$ is not used at all.
\end{remark}
\begin{definition}
A planar algebra $P$ is said to have finite depth if $I_{(l,\varepsilon
)}=P_{(l+1,\varepsilon )}$ for some $l\in \mathbb{N}$, $\varepsilon \in
\left\{ +,-\right\} $ and in that case, the depth of $P$ will be a pair of
natural numbers $\left( l_{+},l_{-}\right) $ such that $l_{\varepsilon }$ is
the smallest natural number such that $I_{(l_{\varepsilon },\varepsilon
)}=P_{(l_{\varepsilon }+1,\varepsilon )}$.
\end{definition}
\begin{remark}
From Lemma \ref{S lemma}, one can deduce that if $\left( l_{+},l_{-}\right)$
denotes the depth of $P$ and $l_{\varepsilon }$
is even (resp. odd), then $l_{-\varepsilon }\in \left\{ l_{\varepsilon
}-1,l_{\varepsilon }\right\} $ (resp. $l_{-\varepsilon }\in \left\{
l_{\varepsilon },l_{\varepsilon }+1\right\} $), that is, either both $l_{+}$
and $l_{-}$ are same or they are consecutive natural numbers with the larger
one being even.
\end{remark}
Let $S_{\left( k,\varepsilon \right) }^{(m)}$ denote the tangle
\includegraphics[scale=0.27,bb= 0 250 700 550]{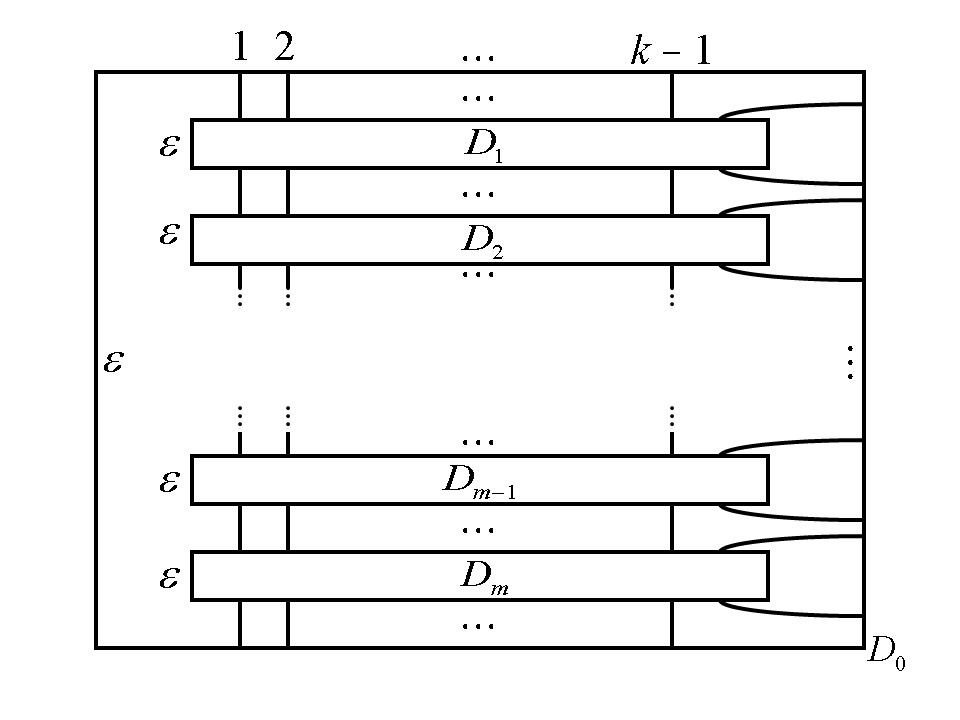}
$\in \mathcal{T}_{(k+m-1,
\varepsilon )}$.
\vskip 2.16cm
Note that $S_{\left( k,\varepsilon \right) }^{(1)}=1_{\left(
k,\varepsilon \right) }$ and $S_{\left( k,\varepsilon \right)
}^{(2)}=S_{\left( k,\varepsilon \right) }$ (defined in the proof of the Lemma
\ref{S lemma}).
\begin{lemma}
\label{general S lemma}If $P$ has finite depth with depth $\left(
l_{+},l_{-}\right) $, then
\begin{equation*}
Range\left\{ P\left( S_{\left( k,\varepsilon \right) }^{(m)}\right) \right\}
=P\left( k+m-1,\varepsilon \right)
\end{equation*}
whenever $k\geq l_{\varepsilon }=$ the $\varepsilon $-depth of $P$ and $m\in
\mathbb{N}$.
\end{lemma}
\begin{proof}
The case $m=1$ is trivial and $m=2$ follows from the proof of Lemma \ref{S
lemma}.

Suppose the statement of the lemma is true for all $m\leq n$. To show the
same for $m=\left( n+1\right) $, we consider the tangle $S_{\left(
k+n-1,\varepsilon \right) }^{(2)}\circ \left( S_{\left( k,\varepsilon \right)
}^{(n)},S_{\left( k,\varepsilon \right) }^{(n)}\right) $. Clearly,
\begin{equation*}
Range\left( P\left( S_{\left( k+n-1,\varepsilon \right) }^{(2)}\circ \left(
S_{\left( k,\varepsilon \right) }^{(n)},S_{\left( k,\varepsilon \right)
}^{(n)}\right) \right) \right) =P\left( k+n,\varepsilon \right)
\end{equation*}
since $k\geq l_{\varepsilon }$. Again, for even $n$ ($=2p$ say), the tangle $%
S_{\left( k+n-1,\varepsilon \right) }^{(2)}\circ \left( S_{\left(
k,\varepsilon \right) }^{(n)},S_{\left( k,\varepsilon \right) }^{(n)}\right) $
isotopically looks like:
\begin{center}
\includegraphics[scale=0.47,bb= 0 0 700 550]{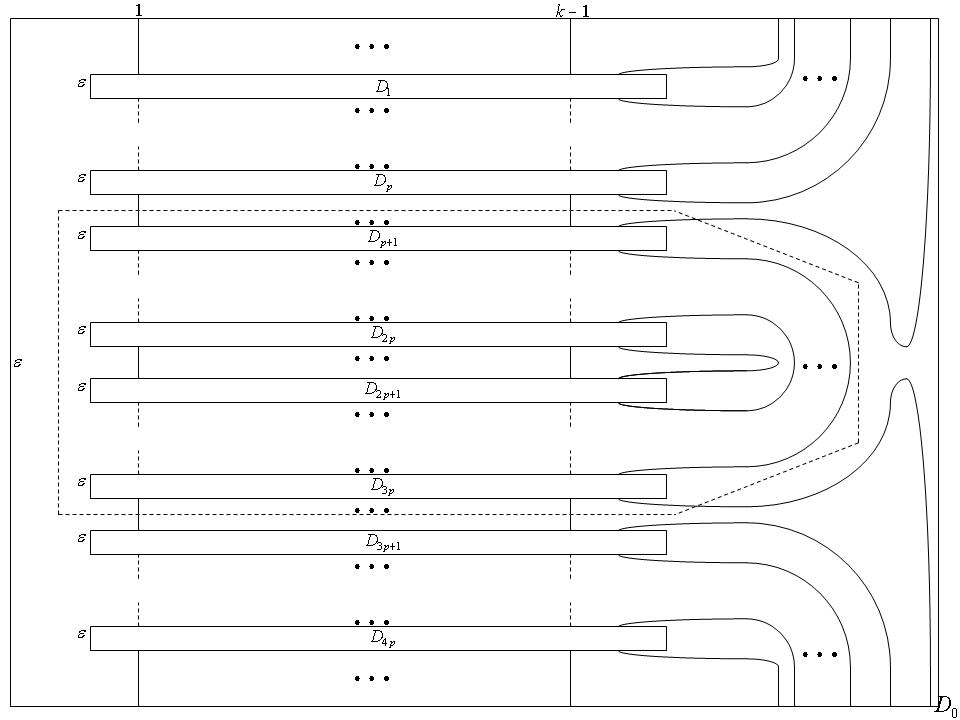}
\end{center}
Note
that the tangle bounded by the dotted line, denoted by $T$, is a
$(k,\varepsilon )$-planar tangle and thus $S_{\left( k+n-1,\varepsilon
\right) }^{(2)}\circ \left( S_{\left( k,\varepsilon \right) }^{(n)},S_{\left(
k,\varepsilon \right) }^{(n)}\right) =S_{\left( k,\varepsilon \right)
}^{(n+1)}\circ \left( 1_{\left( k,\varepsilon \right) },\cdots ,1_{\left(
k,\varepsilon \right) },T,1_{\left( k,\varepsilon \right) },\cdots
,1_{\left( k,\varepsilon \right) }\right) $ where $T$ sits in the $\left(
p+1\right) ^{\text{th}}$ position; this implies
\begin{equation*}
P\left( k+n,\varepsilon \right) =Range\left[ P\left( S_{\left(
k+n-1,\varepsilon \right) }^{(2)}\circ \left( S_{\left( k,\varepsilon \right)
}^{(n)},S_{\left( k,\varepsilon \right) }^{(n)}\right) \right) \right] \subset
Range\left[ P\left( S_{\left( k,\varepsilon \right) }^{(n+1)}\right) \right]
\subset P\left( k+n,\varepsilon \right) \text{.}
\end{equation*}
Hence,
$Range\left( P\left( S_{\left( k,\varepsilon \right) }^{(n+1)}\right) \right)
=Range\left( P\left( S_{\left( k+n-1,\varepsilon \right) }^{(2)}\circ \left(
S_{\left( k,\varepsilon \right) }^{(n)},S_{\left( k,\varepsilon \right)
}^{(n)}\right) \right) \right) =P\left( k+n,\varepsilon \right) $. Similar
arguments can be used to prove the same for odd $n$.
\end{proof}
\begin{proposition}
\label{affine rank}If $P$ is a finite depth planar algebra with $\left(
l_{+},l_{-}\right) $ as its depth, then
\begin{equation*}
\left( \mathcal{A}ffP\right) _{(q,\eta )}^{\left( p,\varepsilon \right)
}=span\left\{ \left( \mathcal{A}ffP\right) _{(s,\nu )}^{\left( p,\varepsilon
\right) }\circ \left( \mathcal{A}ffP\right) _{(q,\eta )}^{\left( s,\nu
\right) }\right\}
\end{equation*}
\comment{
\begin{equation*}
\left( \mathcal{A}ffP\right) _{(q,\eta )}^{\left( p,\varepsilon \right)
}=span\left\{ \left( \mathcal{A}ffP\right) _{(s,\nu )}^{\left( p,\varepsilon
\right) }\circ \left( \mathcal{A}ffP\right) _{(q,\eta )}^{\left( s,\nu
\right) }\right\}
\end{equation*}
}
for all colours $\left( p,\varepsilon \right) $,
$\left( q,\eta \right) $ and $
\nu \in \{+,-\}$ where
$s = \left[ \frac{1}{2}\min \left\{ l_{+},l_{-}\right\}\right] $. (
$\left[ \cdot \right] $ denotes the greatest integer
function.)
\end{proposition}
\begin{proof}
If either of $p$ and $q$ is less than or equal to $s$, then the equality can
easily be established by wiggling a string sufficiently and then decomposing
the affine tangle. One can also assume $\varepsilon =\eta $ because the case
when they are different can be deduced using rotation tangles. Without loss
of generality, let $l=l_{+}\leq l_{-}$, $p,q\geq s+1$ and $\eta =\varepsilon
=+$. Let $A\in \left( \mathcal{A}ffP\right) _{(q,+)}^{\left( p,+\right) }$.
Then, $A$ can be expressed as the equivalence class of the affine tangle $%
\Psi _{(p,+),(q,+)}^{r}$ such that the internal rectangle is labelled with
an element of $P{\left( p+q+r,+\right) }$ where $r$ can be chosen to exceed
$l$ (using wiggling around the inner disc). By Lemma \ref{general S lemma},
$A$ is a linear combination (l.c.) of equivalence class (eq. cl.) of labelled
$\psi _{(p,+),(q,+)}^{r}\left( S_{\left( k,+\right) }^{n}\right)\\
= \psi_{(p,+),(q,+)}^{r}\left(
\includegraphics[scale=0.27,bb= 0 250 700 550]{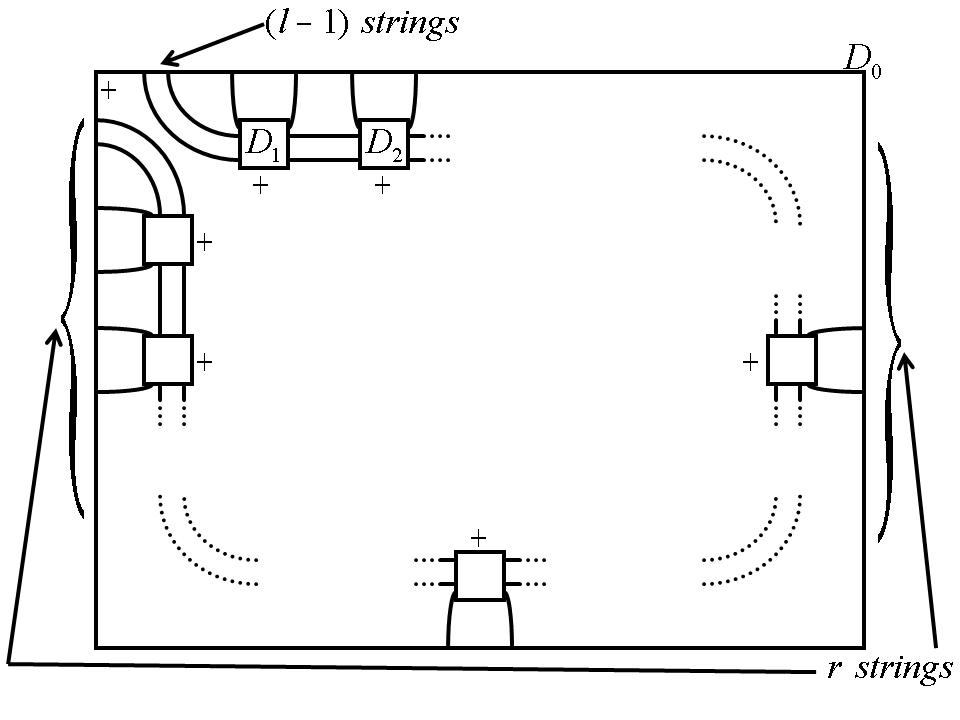}
\right)$. Now, we consider two cases.\\
\textbf{Case 1}: $l$ is odd, that is, $l-1=2s$. We can isotopically move the
internal rectangles attached to left side of the above tangle around the
inner disc and bring them to the to the right side. In this way, we express $%
A$ as:

l.c. of eq. cl. of $\psi _{(p,+),(q,+)}^{2s}\left(
\includegraphics[scale=0.27,bb= 0 250 700 550]{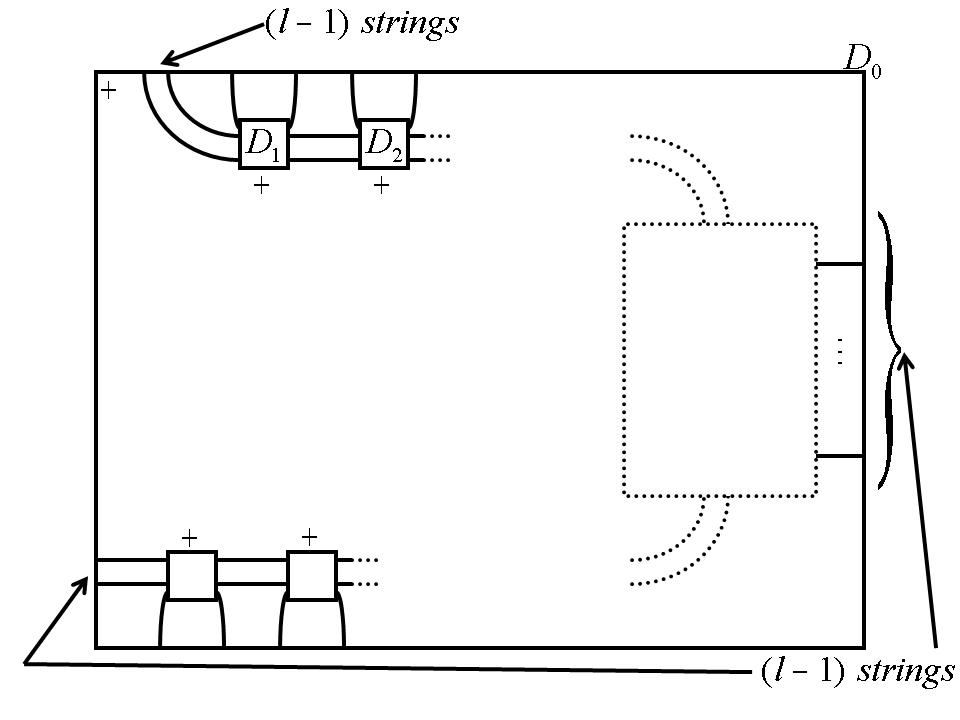}
\right)$\\
$=$ l.c. of eq. cl. of $\psi _{(p,+),(q,+)}^{2s}\left(
\includegraphics[scale=0.20,bb= 0 250 700 550]{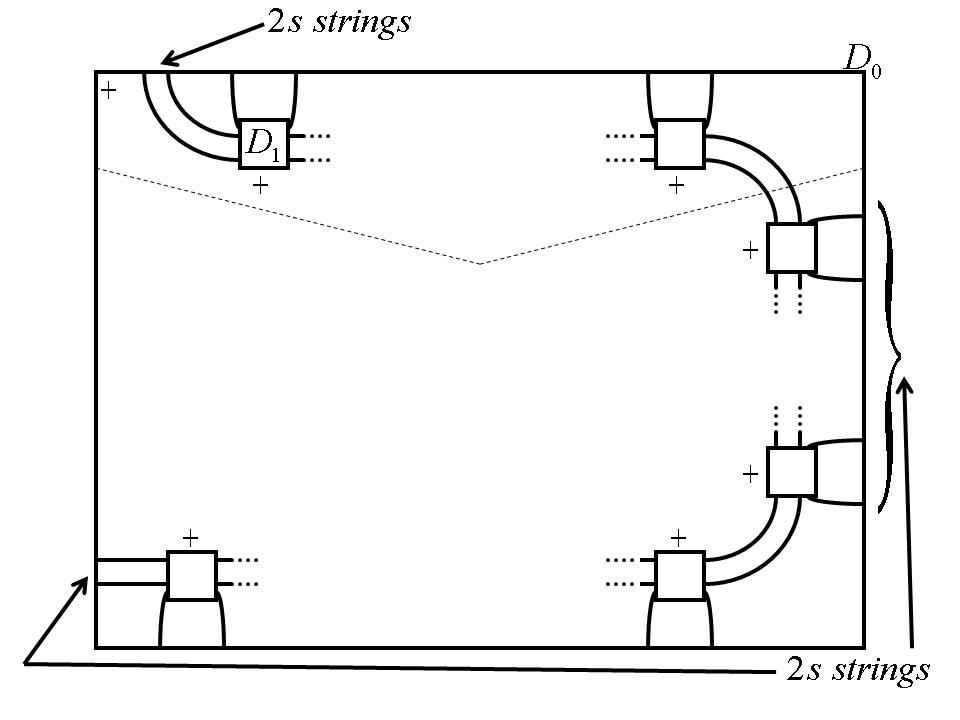}
\right) $. Identifying the
two vertical sides of the last tangle we get an affine tangle which we cut
along the dotted line; this cutting induces a decomposition of $A$. Note
that the dotted line intersects exactly $2s$ strings. Thus $A\in span\left\{
\left( \mathcal{A}ffP\right) _{(s,+)}^{\left( p,+\right) }\circ \left(
\mathcal{A}ffP\right) _{(q,+)}^{\left( s,+\right) }\right\} $.\\
\textbf{Case 2}: $l$ is even, that is, $l=2s$. Using similar arguments as in
Case 1, we
can conclude that $A$ is a l.c. of eq. cl. of
$\psi _{(p,+),(q,+)}^{2s}\left(
\includegraphics[scale=0.27,bb= 0 250 700 550]{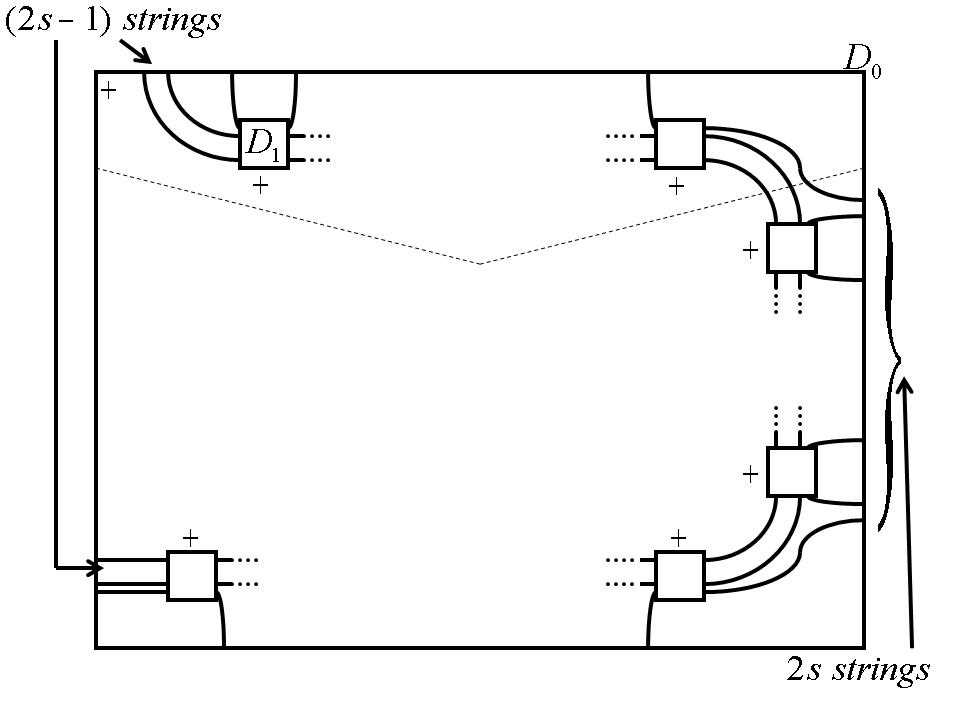}
\right) $. Cutting the tangle along the dotted line just like in Case 1, we can
decompose $A$ and get $A\in span\left\{ \left(
\mathcal{A}ffP\right) _{(s,+)}^{\left( p,+\right) }\circ \left(
\mathcal{A}ffP\right)
_{(q,+)}^{\left( s,+\right) }\right\} $.
\end{proof}
\begin{corollary}
\label{affine ideal}If $P$ is a finite depth planar algebra with $\left(
l_{+},l_{-}\right) $ being its depth, then $\left( \widehat{\mathcal{A}ffP}%
\right) _{(k,\varepsilon )}=\left( \mathcal{A}ffP\right) _{(k,\varepsilon
)}^{(k,\varepsilon )}$ for all colours $\left( k,\varepsilon \right) $ such
that $k>s=\left[ \frac{1}{2}\min \left\{ l_{+},l_{-}\right\} \right] $.
\end{corollary}
\begin{proof}
Follows immediately from the Proposition and definition of $\left( \widehat{%
\mathcal{A}ffP}\right) _{(k,\varepsilon )}$.
\end{proof}
\begin{theorem}
If $P$ is a finite depth subfactor-planar algebra with $\left(
l_{+},l_{-}\right) $ as its depth, then the affine $\ast $-representations
of $P$ can have weight atmost $s=\left[ \frac{1}{2}\min \left\{
l_{+},l_{-}\right\} \right] $.
\end{theorem}
\begin{proof}
Corollary \ref{affine ideal} implies that the lowest weight algebra $%
(LWP)_{(k,\varepsilon )}=\left\{ 0\right\} $ whenever $k>s$. Thus, from the
discussion of finding irreducible affine representations in section $4$, all
irreducible affine representations have weight atmost $s$. To prove the same
for non-irreducible ones, note that taking direct sums never increases the
weight.
\end{proof}
\begin{theorem}
If $P$ is a finite depth subfactor-planar algebra with modulus $\left(
\delta ,\delta \right) $, then every irreducible affine $\ast $%
-representation of $P$ is locally finite and the radius of convergence of
its dimension is at most $\frac{1}{\delta ^{2}}$. Moreover, the number of
irreducibles at each weight is finite.
\end{theorem}
\begin{proof}
Let $F$ be an irreducible affine $\ast $-representation with weight $k$. So,
$F\left( k,\varepsilon \right) $ is an irreducible module of $\left(
LWP\right) _{(k,\varepsilon )}$. Irreducibility of $F$ says that $F$ induces
an surjective linear map from $\left( \mathcal{A}ffP\right) _{(k,\varepsilon
)}^{\left( p,\eta \right) }\otimes F\left( k,\varepsilon \right) $ to $%
F(p,\eta )$. Therefore, we have $\dim \left( F(p,\eta )\right) \leq \dim
\left( \left( \mathcal{A}ffP\right) _{(k,\varepsilon )}^{\left( p,\eta
\right) }\right) \dim \left( F(k,\varepsilon )\right) $. We look back once
again into the two cases in the proof of Proposition \ref{affine rank}. Let $%
l$ and $s$ be as in Proposition\ \ref{affine rank} for the rest of the
proof. A careful observation on the two cases will say that there exists a
surjective linear map from $P\left( l,+\right) ^{\otimes \left( p+k\right) }$
(resp. $P\left( l,+\right) ^{\otimes \left( p+k+1\right) }$) to $\left(
\mathcal{A}ffP\right) _{(k,+)}^{\left( p,+\right) }$ when $l$ is odd (resp.
even). Therefore,
\begin{equation*}
\dim \left( \left( \mathcal{A}ffP\right) _{(k,\varepsilon )}^{\left( p,\eta
\right) }\right) =\dim \left( \left( \mathcal{A}ffP\right) _{(k,+)}^{\left(
p,+\right) }\right) \leq \left( p+k+1\right) ~\dim \left( P\left( l,+\right)
\right) <\infty
\end{equation*}
since $P$ is locally finite. The lowest weight algebras become finite
dimensional and hence there are finitely many irreducibles at each weight.
This also implies $F\left( k,\varepsilon \right) $ has finite dimension.
Thus $F$ is locally finite.

Next, consider the labelled affine tangle obtained by the action of $\psi
_{(p,+),(k,+)}^{l-1}$ (resp. $\psi _{(p,+),(k,+)}^{l}$) on the\linebreak
\newpage \noindent tangle
\includegraphics[scale=0.27,bb= 0 250 700 550]{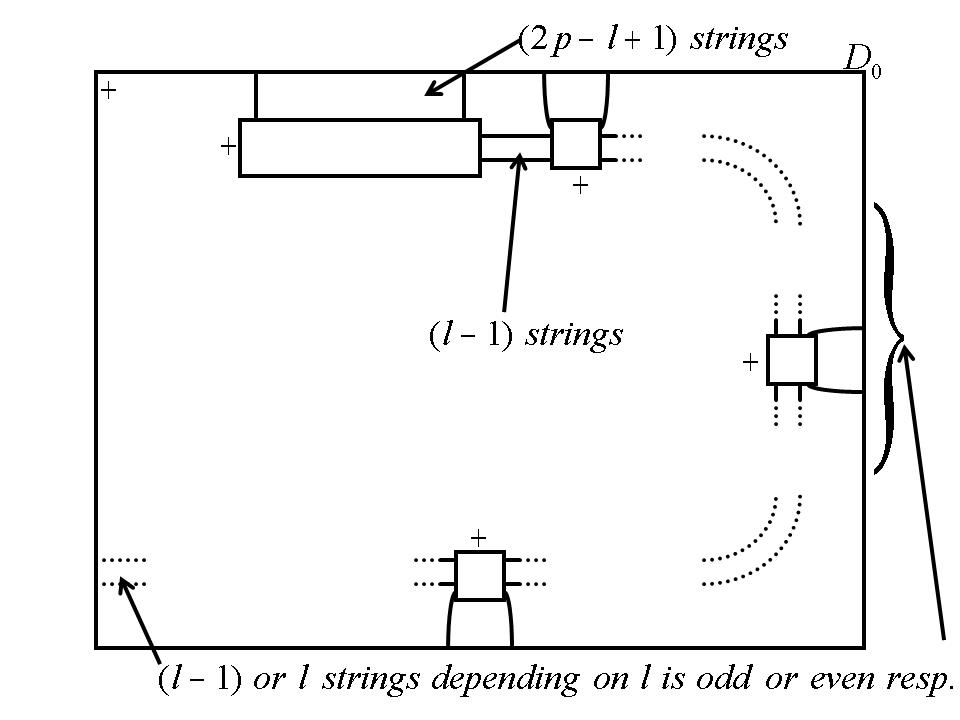}
if $l$ is odd (resp. even).
\vskip 2.16cm
\noindent By Lemma \ref{general S lemma} and proof of Proposition \ref{affine
rank}, eq.
cl. of such labelled tangles generate $\left( \mathcal{A}ffP\right)
_{(k,+)}^{\left( p,+\right) }$. Therefore,
\begin{equation*}
\dim \left( \left( \mathcal{A}ffP\right) _{(k,\varepsilon )}^{\left( p,\eta
\right) }\right) =\dim \left( \left( \mathcal{A}ffP\right) _{(k,+)}^{\left(
p,+\right) }\right) \leq \left( k+l\right) ~\dim \left( P\left( l,+\right)
\right) \dim \left( P\left( p,+\right) \right) \text{.}
\end{equation*}
So, $\dim \left( F(p,\eta )\right) \leq \left( k+l\right) ~\dim \left(
P\left( l,+\right) \right) \dim \left( P\left( p,+\right) \right) $. Now, we
try to find the limit of
\[\left( (k+l)\dim \left( P\left( l,+\right) \right)
\dim \left( P\left( p,+\right) \right) \right) ^{\frac{1}{p}}\]
as $p$ tends to infinity. Note that $(k+l)\dim \left( P\left( l,+\right)
\right) $ is constant. Next, $\lim_{p\longrightarrow \infty }\left( \dim
\left( P\left( p,+\right) \right) \right) ^{\frac{1}{p}}=$ norm of the
principal graph $=$ the index of the finite depth subfactor corresponding to
the planar algebra. By Jones' theorem, index of the subfactor is square of
the modulus. Hence, $\lim \sup_{p\longrightarrow \infty }\left( \dim \left(
F(p,\eta )\right) \right) ^{\frac{1}{p}}\leq \lim \sup_{p\longrightarrow
\infty }\left( \left( k+l\right) ~\dim \left( P\left( l,+\right) \right)
\dim \left( P\left( p,+\right) \right) \right) ^{\frac{1}{p}}=\delta ^{2}$
which implies radius of convergence of $\Phi _{F}^{\eta }$ is at least
$\frac{1}{\delta ^{2}}$. This ends the proof.
\end{proof}
\bibliographystyle{amsplain}

\end{document}